\documentclass{amsart}

\usepackage{amssymb
,amsthm
,amsmath
,amscd
,mathtools
}
\usepackage[all]{xy}
\usepackage[top=30truemm,bottom=30truemm,left=25truemm,right=25truemm]{geometry}
\allowdisplaybreaks

\newcommand{\C}{\mathbb{C}}

\newcommand{\G}{\mathbb{G}}

\newcommand{\Z}{\mathbb{Z}}

\renewcommand{\AA}{\mathcal{A}}
\newcommand{\BB}{\mathcal{B}}
\renewcommand{\H}{\mathcal{H}}
\newcommand{\LL}{\mathcal{L}}
\newcommand{\MM}{\mathcal{M}}
\newcommand{\RR}{\mathcal{R}}
\newcommand{\Sch}{\mathcal{S}}
\newcommand{\TT}{\mathcal{T}}
\newcommand{\VV}{\mathcal{V}}

\newcommand{\cent}{\mathrm{Cent}}
\newcommand{\Gal}{\mathrm{Gal}}
\newcommand{\Hom}{\mathrm{Hom}}
\newcommand{\Norm}{\mathrm{Norm}}

\newcommand{\im}{\mathrm{Im}}
\newcommand{\Irr}{\mathrm{Irr}}
\newcommand{\re}{\mathrm{Re}}

\newcommand{\Lie}{\mathrm{Lie}}
\newcommand{\Ad}{\mathrm{Ad}}
\newcommand{\Int}{\mathrm{Int}}

\newcommand{\ad}{\mathrm{ad}}

\newcommand{\disc}{\mathrm{disc}}
\newcommand{\id}{\mathrm{id}}
\newcommand{\gen}{\mathrm{gen}}
\newcommand{\temp}{\mathrm{temp}}

\newcommand{\Ind}{\mathrm{Ind}}
\newcommand{\tr}{\mathrm{tr}}
\newcommand{\vol}{\mathrm{vol}}

\newcommand{\Sh}{{\rm Sh}}

\newcommand{\inv}{{\rm inv}}

\newcommand{\sr}{{\rm sr}}
\newcommand{\even}{{\rm even}}
\newcommand{\odd}{{\rm odd}}

\newcommand{\WD}{{\it WD}}
\newcommand{\Frob}{\mathrm{Frob}}
\newcommand{\Sp}{\mathrm{Sp}}
\newcommand{\GSp}{\mathrm{GSp}}
\newcommand{\SL}{\mathrm{SL}}
\newcommand{\GL}{\mathrm{GL}}

\newcommand{\orth}{\mathrm{O}}
\newcommand{\SO}{\mathrm{SO}}
\newcommand{\Mp}{\widetilde{\Sp}}
\newcommand{\Sym}{\mathrm{Sym}}
\newcommand{\End}{\mathrm{End}}
\newcommand{\M}{\mathrm{M}}
\newcommand{\spl}{{\bf spl}}

\newcommand{\bs}{\backslash}
\newcommand{\ep}{\varepsilon}
\newcommand{\lam}{\lambda}
\newcommand{\Lam}{\Lambda}

\newcommand{\Lgp}[1]{{}^L{#1}}
\renewcommand{\1}{{\bf 1}}
\newcommand{\0}{{\bf 0}}
\newcommand{\bphi}{\underline{\phi}}
\newcommand{\epinv}{\varepsilon\text{-invariant}}
\newcommand{\pair}[1]{\langle #1 \rangle}
\newcommand{\half}[1]{\frac{#1}{2}}
\newcommand{\cl}[1]{\widetilde{#1}}
\newcommand{\ch}[1]{\check{#1}}
\newcommand{\iif}{&\quad&\text{if }}
\newcommand{\other}{&\quad&\text{otherwise}}
\newcommand{\resp}{resp.~}

\newcommand{\e}{\mathfrak{e}}

\newcommand{\oo}{\mathfrak{o}}

\newcommand{\w}{\mathfrak{w}}
\newcommand{\z}{\mathfrak{z}}

\newtheorem{thm}{Theorem}[section]
\newtheorem{lem}[thm]{Lemma}
\newtheorem{rem}[thm]{Remark}
\newtheorem{prop}[thm]{Proposition}
\newtheorem{cor}[thm]{Corollary}
\newtheorem{ass}[thm]{Assumption}
\newtheorem{conj}[thm]{Conjecture}
\newtheorem{question}[thm]{Question}
\newtheorem{example}[thm]{Example}

\makeatletter
\def\iddots{\mathinner{\mkern1mu\raise\p@
	\hbox{.}\mkern2mu\raise4\p@\hbox{.}\mkern2mu
	\raise7\p@\vbox{\kern7\p@\hbox{.}}\mkern1mu}}
\def\adots{\mathinner{\mkern2mu\raise\p@\hbox{.}
 \mkern2mu\raise4\p@\hbox{.}\mkern1mu
 \raise7\p@\vbox{\kern7\p@\hbox{.}}\mkern1mu}}
\makeatother

\title[Local theta correspondence and the Gan--Gross--Prasad conjecture for $\Sp_{2n}$-$\Mp_{2n}$]
{The local theta correspondence and 
the local Gan--Gross--Prasad conjecture for the symplectic-metaplectic case}
\author{Hiraku Atobe}
\date{}
\address{Department of mathematics, Kyoto University, Kitashirakawa-Oiwake-cho, Sakyo-ku, Kyoto, 606-8502, Japan}
\email{atobe@math.kyoto-u.ac.jp}

\begin{document}
\maketitle

\begin{abstract}
We prove the local Gan--Gross--Prasad conjecture for the symplectic-metaplectic case
under some assumptions.
This is the last case of the local Gan--Gross--Prasad conjectures
which has not been established.
We also prove two of Prasad's conjectures on the local theta correspondence 
in the almost equal rank case.
\end{abstract}

\tableofcontents

%\section{Introduction}\label{intro}
\section{Introduction}\label{intro}
In \cite{GP1, GP}, Gross and Prasad studied
a restriction problem for special orthogonal groups over a local field
and gave a precise conjecture.
They and Gan (\cite{GGP, GGP2}) extended
this conjecture to classical groups,
which are called the local Gan--Gross--Prasad conjectures (GGP).
These conjectures consist of four cases;
the orthogonal, hermitian, 
symplectic-metaplectic and skew-hermitian cases.
The orthogonal, hermitian and skew-hermitian cases were proven by
Waldspurger \cite{Wa2, Wa3, Wa4, Wa5}
and M\oe glin--Waldspurger \cite{MW},
Bezuart-Plessis \cite{BP1, BP2, BP3} 
and Gan--Ichino \cite{GI2}, respectively.
\par

In this paper, we consider the orthogonal and the symplectic-metaplectic cases.
Let $F$ be a non-archimedean local field of characteristic zero.
We denote by $(V_{m+1},\pair{\cdot,\cdot}_{V_{m+1}})$ 
(\resp $(W_{2n},\pair{\cdot,\cdot}_{W_{2n}})$)
an orthogonal space of dimension $m+1$ (\resp a symplectic space of dimension $2n$).
Let $V_m\subset V_{m+1}$ be a non-degenerate subspace of codimension $1$,
so that we have a natural inclusion $\SO(V_m)\hookrightarrow \SO(V_{m+1})$.
We let $\Mp(W_{2n})$ be the metaplectic group, i.e., the unique non-split
central extension of $\Sp(W_{2n})$:
\[
\begin{CD}
1@>>>\{\pm1\}@>>>\Mp(W_{2n})@>>>\Sp(W_{2n})@>>>1.
\end{CD}
\]
In particular, we have a diagonal embedding
\[
\Delta\colon \SO(V_m)\hookrightarrow \SO(V_m)\times\SO(V_{m+1})
\]
and a natural map
\[
\Delta\colon\Mp(W_{2n})\rightarrow \Mp(W_{2n})\times\Sp(W_{2n}),
\]
where the first factor is the identity map and the second factor is the projection map.
\par

In the orthogonal case,
for an irreducible smooth representation $\sigma$ of $\SO(V_m)\times\SO(V_{m+1})$,
one is interested in determining
\[
\dim_\C(\Hom_{\Delta\SO(V_{m})}(\sigma,\C)).
\]
We shall call this the Bessel case (B) of the GGP conjecture.
On the other hand, we fix a non-trivial additive character $\psi$ of $F$,
and let $\omega_{\psi}$ be a Weil representation of $\Mp(W_{2n})$,
which is given by the Heisenberg group associated to the symplectic space 
$(W_{2n},2\pair{\cdot,\cdot}_{W_{2n}})$.
In the symplectic-metaplectic case,
for an irreducible genuine smooth representation $\pi$ of $\Mp(W_{2n})\times\Sp(W_{2n})$,
one is interested in determining
\[
\dim_\C(\Hom_{\Delta\Mp(W_{2n})}(\pi\otimes\overline{\omega_{\psi}},\C)).
\]
We shall call this the Fourier--Jacobi case (FJ) of the GGP conjecture.
\par

By results of Aizenbud--Gourevitch--Rallis--Schiffmann \cite{AGRS}, Waldspurger \cite{AGRS2}
and Sun \cite{Sun}, 
it is known that the above $\Hom$ spaces have dimension at most $1$.
Hence, the main problems are to determine when the $\Hom$ spaces are nonzero.
In \cite{GGP}, an answer for these problems is formulated in the framework of the
local Langlands correspondence in a form proposed by Vogan \cite{Vo},
which treats the irreducible representations of all pure inner forms simultaneously.
\par

More precisely, a pure inner form of $\SO(V_m)$ is simply a group of the form
$\SO(V_m')$, where $V_m'$ is an orthogonal space with 
the same dimension and discriminant as $V_m$.
Hence, a pure inner form of $\SO(V_m)\times\SO(V_{m+1})$
is a group of the form 
\[
\SO(V_m')\times\SO(V_{m+1}').
\]
Wa say that this pure inner form is relevant if
$V_m'\subset V_{m+1}'$ and $V_{m+1}'/V_{m}'\cong V_{m+1}/V_m$
as orthogonal spaces.
In this case, we have a diagonal embedding
\[
\Delta\colon \SO(V_m')\hookrightarrow \SO(V_m')\times\SO(V_{m+1}').
\]
On the other hand, $\Mp(W_{2n})\times\Sp(W_{2n})$
has no non-trivial pure inner forms.
\par

Let $G=\SO(V_m)\times\SO(V_{m+1})$ or $G=\Mp(W_{2n})\times\Sp(W_{2n})$.
For an $L$-parameter $\phi$ of $G$,
we should obtain a Vogan $L$-packet $\Pi_\phi$ 
consisting of certain irreducible smooth representations
of $G$ and its (not necessarily relevant) pure inner forms $G'$.
Here, to parametrize irreducible genuine representations of $\Mp(W_{2n})$, 
we use the theta correspondence for $(\Mp(W_{2n}),\orth(V_{2n+1}))$.
Hence, by an $L$-parameter of $\Mp(W_{2n})$, we mean one of $\SO(V_{2n+1})$, 
i.e., a symplectic representation
\[
\phi_M\colon\WD_F\rightarrow \Sp(M)
\]
of the Weil--Deligne group $\WD_F$ of $F$ with $\dim_\C(M)=2n$.
Moreover, for a fixed Whittaker datum $\w=(B,\mu)$ of $G$,
which is a conjugacy class of a pair of an $F$-rational Borel subgroup $B=TU$ of $G$
and a generic character $\mu$ of the unipotent radical $U(F)$ of $B(F)$,
there exists a natural bijection
\[
\iota_\w\colon\Pi_\phi\rightarrow \Irr(\pi_0(S_\phi)),
\]
where we put $S_\phi=\cent(\im(\phi),\widehat{G})$.
For $\eta\in\Irr(\pi_0(S_\phi))$, we write $\pi(\eta)=\iota_\w^{-1}(\eta)$.
\par

The local Langlands conjecture has been established 
for quasi-split symplectic and special orthogonal groups
by Arthur \cite{A} under an assumption on the stabilization of the twisted trace formula.
For this assumption, see also two books of M{\oe}glin--Waldspurger \cite{Stab}.
In \cite{GS1}, Gan--Savin gave a parametrization 
(depending on the choice of a non-trivial additive character $\psi$ of $F$) 
of irreducible genuine representations of $\Mp(W_{2n})$ by using
the theta correspondence and the local Langlands correspondence for $\SO(V_{2n+1})$.
\par

The GGP conjectures can be roughly stated as follows:
\begin{conj}[B]
Let $G=\SO(V_m)\times\SO(V_{m+1})$.
\begin{enumerate}
\item
Given a generic $L$-parameter $\phi$ of $G$, 
there exists a unique representation $\pi(\eta)\in\Pi_\phi$ such that
$\pi(\eta)$ is a representation of a relevant pure inner form 
$G'=\SO(V_m')\times\SO(V_{m+1}')$ and 
such that $\Hom_{\Delta\SO(V_m')}(\pi(\eta),\C)\not=0$;
\item
There is a precise recipe for the unique character $\eta$.
\end{enumerate}
\end{conj}
\begin{conj}[FJ]
Let $G=\Mp(W_{2n})\times\Sp(W_{2n})$ and fix a non-trivial additive character $\psi$ of $F$.
\begin{enumerate}
\item
Given a generic $L$-parameter $\phi$ of $G$, 
there exists a unique representation $\pi(\eta)\in\Pi_\phi$
such that $\Hom_{\Delta\Mp(W_{2n})}(\pi(\eta)\otimes\overline{\omega_{\psi}},\C)\not=0$;
\item
There is a precise recipe for the unique character $\eta$.
\end{enumerate}
\end{conj}
In \S \ref{conjectures},
we will recall the recipes for the unique characters.
Waldspurger \cite{Wa2, Wa3, Wa4, Wa5}
showed the orthogonal case for tempered $L$-parameters,
and M\oe glin--Waldspurger \cite{MW}
extended this result for generic $L$-parameters.
\par

In fact, there are GGP conjectures in general codimension cases 
(Conjectures 17.1 and 17.3 in \cite{GGP}).
However, \cite[Theorem 19.1]{GGP} says that
the general codimension cases follow from the basic cases (B) and (FJ).
Hence we consider only the basic cases in this paper.
\par

The purpose of this paper is to establish the symplectic-metaplectic case (FJ),
as well as two conjectures of D. Prasad concerning local theta correspondence
in almost equal rank cases.
To do these, we use:
\begin{itemize}
\item[(LLC)]
the local Langlands correspondence
for symplectic groups and special orthogonal groups
in a form proposed by Vogan (see \S \ref{LLC});
\item[(GPR)]
a conjecture of Gross--Prasad and Rallis,
which characterizes the generic $L$-parameters in terms of the local adjoint $L$-function
(see \S \ref{LLCproperty});
\item[(B)]
the works of Waldspurger \cite{Wa2, Wa3, Wa4, Wa5}
and M\oe glin--Waldspurger \cite{MW}
on the Bessel case of the GGP conjecture
(see \S \ref{GGPforSO});
\item[(NQ)]
a property of the local Langlands correspondence
for non-quasi-split special orthogonal groups
(see \S \ref{sec.intertwining});
%\item[(IS)]
%an irreducibility condition for standard modules of $\Mp(W_{2n})$
%(see \S \ref{Mp}).
\end{itemize}
In this paper, we show the following:
\begin{thm}\label{main}
Assume $(LLC)$, $(GPR)$ and $(B)$.
Let $\phi$ and $\cl{\phi}$ be generic $L$-parameters of $\Sp(W_{2n})$ and $\Mp(W_{2n})$,
respectively.
\begin{enumerate}
\item\label{main1}
If $\phi$ and $\cl{\phi}$ are tempered, then
$(FJ)$ is true for $\phi\times\cl{\phi}$.
\item\label{main2}
If we further assume $(NQ)$ %, $(IS)$ 
and that there exists a quadratic character $\chi$ of $F^\times$ such that
the local $L$-function $L(s,\phi\otimes\chi)$ is regular at $s=1$,
then $(FJ)$ is true for $\phi\times\cl{\phi}$.
\end{enumerate}
\end{thm}
The argument in the proof of \cite[Theorem 19.1]{GGP} works when
we restrict to the above cases respectively.
Namely, we can deduce the following corollary from Theorem \ref{main}.
\begin{cor}
Under the same assumptions as Theorem $\ref{main}$, 
Conjectures $17.1$ and $17.3$ in \cite{GGP} are true.
\end{cor}
\par

We have some remarks on these assumptions.
The recent result of Arthur \cite{A}, 
(LLC) for quasi-split symplectic and special orthogonal groups
is almost completely known.
For any quasi-split connected reductive groups, 
(GPR) has been established in \cite[Appendix B]{GI2}.
The property (NQ) is an intertwining relation for non-quasi-split special orthogonal groups.
For quasi-split orthogonal groups and symplectic groups, 
the intertwining relation easily follows from results of Arthur \cite{A}.
%The properties analogous to (IS) have been established 
%by M\oe glin--Waldspurger \cite{MW} for $\Sp(W_{2n})$ and $\SO(V_{m})$, and
%by Heiermann \cite{H} for any quasi-split connected reductive groups.
\par

We describe the main idea of the proof of Theorem \ref{main}.
The method is analogous to the work of Gan--Ichino \cite{GI2}.
As in that paper, by the local theta correspondence,
the Bessel and the Fourier--Jacobi cases of GGP conjectures are
related.
More precisely, there exists a see-saw diagram
\[
\xymatrix{
   \Mp(W_{2n})\times_{\{\pm1\}}\Mp(W_{2n})   \ar@{-}[d]    &  &  \orth(V_{2n+2}) \ar@{-}[d]     \\
     \Sp(W_{2n}) \ar@{-}[urr]      &  &  \orth(V_{2n+1})\times\orth(V_1)\ar@{-}[ull]   \\
}
\]
with $\disc(V_1)=-1$, and the associated see-saw identity holds:
\[
\Hom_{\Sp(W_{2n})}(\Theta_{\psi,V_{2n+1},W_{2n}}(\tau)\otimes\omega_{-\psi},\pi)
\cong
\Hom_{\orth(V_{2n+1})}(\Theta_{\psi,V_{2n+2},W_{2n}}(\pi),\tau)
\]
for irreducible smooth representations $\pi$ of $\Sp(W_{2n})$ and $\tau$ of $\orth(V_{2n+1})$.
The left-hand side of the see-saw identity is related to the Fourier--Jacobi case (FJ),
whereas
the right-hand side is related to the Bessel case (B).
Therefore, if we knew the local theta correspondence for $(\Sp(W_{2n}),\orth(V_{2n+2}))$ and
$(\orth(V_{2n+1}),\Mp(W_{2n}))$ explicitly, 
then the see-saw identity would give the precise relation of (FJ) and (B).
\par

More precisely, one considers the following statements:
\begin{enumerate}
\item[($\Theta$)]
If $\phi_{\pi}$ (\resp $\phi_\tau$) is a generic $L$-parameter of $\Sp(W_{2n})$ 
(\resp $\SO(V_{2n+1})$)
and $\pi\in\Pi_{\phi_\pi}$ (\resp $\tau\in\Pi_{\tau}$),
then the big theta lift $\Theta_{\psi,V_{2n+2},W_{2n}}(\pi)$
(\resp $\Theta_{\psi,V_{2n+1},W_{2n}}(\tau)$)  
is irreducible (if it is nonzero).
\item[(Mp)]
If $\tau|\SO(V_{2n+1})$ has parameter $(\phi_\tau,\eta_\tau)$ and 
$\theta_{\psi,V_{2n+1},W_{2n}}(\tau)$ has parameter $(\cl{\phi},\cl{\eta})$,
then $(\cl{\phi},\cl{\eta})$ can be described in terms of $(\phi_\tau,\eta_\tau)$ explicitly.
\item[(P1)]
Likewise, if $\pi$ has parameter $(\phi_\pi,\eta_\pi)$ and 
$\sigma\coloneqq\theta_{\psi,V_{2n+2},W_{2n}}(\pi)|\SO(V_{2n+2})$ 
has parameter $(\phi_\sigma,\eta_\sigma)$,
then $(\phi_\sigma,\eta_\sigma)$ can be described in terms of $(\phi_\pi,\eta_\pi)$ explicitly.
\item[(G)]
If $\phi_\pi$ and $\cl{\phi}$ are generic parameters, then 
so are $\phi_\sigma$ and $\phi_\tau$.
\end{enumerate}
Note that these statements are not true in general.
See Propositions \ref{G} and \ref{TMP} below.
\par

The statement (G) is needed to use the Bessel case (B).
For $\cl\phi$, by definition of generic parameters for $\Mp(W_{2n})$, 
the statement (G) is true.
(See \S \ref{Mp}.)
However, (G) is not always true for $\phi_\pi$.
The last assumption in Theorem \ref{main} $(2)$ implies (G) for $\phi_\pi$.
More precisely, by using (GPR) for $\Sp(W_{2n})$ and $\SO(V_{2n+2})$, 
we have the following:
\begin{prop}[Lemma \ref{L(1,Ad)}]\label{G}
Let $\phi_\pi$ be a generic $L$-parameter of $\Sp(W_{2n})$
and $\chi$ be the discriminant character of $V_{2n+2}$.
Then the $L$-parameter $\phi_\sigma$ of $\SO(V_{2n+2})$ given by $\phi_\pi$
is generic if and only if
the local $L$-function $L(s,\phi_\pi\otimes\chi)$ is regular at $s=1$.
\end{prop}
In particular, if $\phi_\pi$ is tempered, then $\phi_\sigma$ is generic.
Indeed, $\phi_\sigma$ is also tempered.
\par

In \cite{P1}, D. Prasad has formulated a precise conjecture regarding (P1).
He also has formulated one regarding (P2) as follows:
\begin{enumerate}
\item[(P2)]
Let $\cl{\sigma}$ be an irreducible representation of $\orth(V_{2n})$.
If an irreducible constituent of $\sigma\coloneqq\cl{\sigma}|\SO(V_{2n})$ 
has parameter $(\phi_{\sigma},\eta_\sigma)$ and 
$\theta_{\psi,V_{2n+1},W_{2n}}(\cl{\sigma})$ has parameter $(\phi_\pi,\eta_\pi)$,
then $(\phi_\pi,\eta_\pi)$ can be described in terms of $(\phi_\sigma,\eta_\sigma)$ explicitly.
\end{enumerate}
We shall also denote by (weak P1) (\resp (weak P2)) the part of Conjecture (P1) (\resp (P2))
concerning only the correspondence of $L$-parameters.
\par

\begin{prop}\label{TMP}
We have the following:
\begin{enumerate}
\item$(\text{\cite[Theorem $8.1$ $(\mathrm{ii})$]{GS1} and 
\cite[Proposition $C.4$ $(\mathrm{i})$]{GI1}})$
The statement $(\Theta)$ holds for tempered parameters.
\item$(\text{\cite[Proposition $C.4$ $(\mathrm{ii})$]{GI1} $+$
\cite[p. 40 Th\'eor\`eme $(\mathrm{i})$]{MW}})$
If $\phi_\sigma$ is generic, then $(\Theta)$ holds for $\phi_\pi$.
\item$(\text{\cite[Theorem $8.1$ $(\mathrm{iii})$]{GS1} and Lemma $\ref{IS-irr}$})$
%The assumption $(IS)$ implies 
If $\phi_\tau$ is generic, then $(\Theta)$ for $\phi_\tau$.
\item$(\text{\cite[Corollary $1.2$]{GS1} and \cite[Theorem $C.5$]{GI1}})$
The statements $(Mp)$, $(weak\ P1)$ and $(weak\ P2)$ hold.
\end{enumerate}
\end{prop}
\par

Using the see-saw identity above and Proposition \ref{TMP},
we reduce Theorem \ref{main} and (P2) to Conjecture (P1) and (B).
In view of 
results of Waldspurger \cite{Wa2, Wa3, Wa4, Wa5}
and M\oe glin--Waldspurger \cite{MW},
it is enough to show the statement (P1).
Finally, we show that:
\begin{thm}
Conjecture $(P1)$ holds. 
\end{thm}
\par
In \cite{GI2}, Gan--Ichino proved Prasad's conjecture for unitary dual pairs, 
and conclude the GGP conjecture for the skew-hermitian case.
The proof of Prasad's conjecture for unitary dual pairs (for tempered case) 
is an induction on the dimension of hermitian spaces, 
which consists of two steps; 
a global argument with using Arthur's multiplicity formula, and
a local argument with using intertwining relations.
The use of Arthur's multiplicity formula is the main novelty of that paper.
Of course, one can prove (P1) by a similar way, but the global argument is so difficult
for several reasons.
In this paper, 
we prove (P1) by a new method not using global arguments but
using only a local argument.
It makes the proof of (P1) much easier.
\par

The proof of (P1) for tempered parameters in this paper is as follows:
For given tempered $\pi$, $\sigma = \theta_{\psi,V_{2n+2},W_{2n}}(\pi) \not=0$ 
and a semi-simple element $a \in S_{\phi_\pi}$, 
consider a suitable irreducible tempered representation $\tau_a$ of $\GL_k(F)$ and
the induced representations 
\[
\pi_a = \Ind_Q^{\Sp(W_{2n+2k})}(\tau_a\chi_V \otimes \pi)
\quad\text{and}\quad 
\sigma_a = \Ind_P^{\orth(V_{2n+2+2k})}(\tau_a \otimes \sigma).
\]
One can show that:
\begin{itemize}
\item
$\pi_a$ and $\sigma_a$ are irreducible; 
\item
$\sigma_a = \theta_{\psi,V_{2n+2+2k},W_{2n+2k}}(\pi_a)$;
\item
$\pi_a$ has parameter $(\phi_\pi',\eta_{\pi}')$ with 
$\pi_0(S_{\phi_\pi'}) \supset \pi_0(S_{\phi_{\pi}})$ and $\eta_{\pi}'|\pi_0(S_{\phi_{\pi}}) = \eta_\pi$;
\item
$\sigma_a|\SO(V_{2n+2+2k})$ has parameter $(\phi_\sigma',\eta_{\sigma}')$ with 
$\pi_0(S_{\phi_\sigma'}) \supset \pi_0(S_{\phi_{\sigma}})$ and 
$\eta_{\sigma}'|\pi_0(S_{\phi_{\sigma}}) = \eta_\sigma$.
\end{itemize}
Finally, we obtain a relation between $\eta'_\sigma$ and $\eta'_\pi$
by using the intertwining relations of $\pi_a$ and $\sigma_a$.
Since $a \in S_{\phi_\pi}$ is arbitrary, we can describe $\eta_\sigma$ in terms of $\eta_\pi$ explicitly.
Conjecture (P1) for general parameters follows from that for tempered parameters
by \cite[Proposition C.4]{GI1}.
\par

In the proof of Prasad's conjecture, we need to know $\iota_{\w_0}\circ\iota_\w^{-1}$
for two Whittaker data $\w$ and $\w_0$ of $\Sp(W_{2n})$ or $\SO(V_{2n+2})$.
By Kaletha \cite{Ka1},
it is known for the local Langlands correspondence 
for quasi-split symplectic and special orthogonal groups.
However, in the local Langlands correspondence proposed by Vogan, certain issues arise.
\par

We explain these issues in more general setting.
Let $G$ be a quasi-split connected reductive group over 
a non-archimedean local field $F$ of characteristic zero.
A pure inner twist of $G$ is a pair $(G', \psi, z)$, where
$G'$ is an algebraic group over $F$, $z\in Z^1(F,G)$ is a $1$-cocycle, and
\[
\psi\colon G\rightarrow G'
\]
is an isomorphism over $\overline{F}$ such that
$\psi^{-1}\circ\sigma(\psi)=\Int(z_\sigma)$ for any $\sigma\in \Gal(\overline{F}/F)$.
We call such a $G'$ a pure inner form of $G$.
In contrast to inner forms, $G'$ may be also quasi-split over $F$ 
even if the pure inner twist is non-trivial.
For example, if $V_{2n}$ is an orthogonal space with non-trivial discriminant,
then $\SO(V_{2n})$ has two pure inner forms both of which are quasi-split.
Note that in this situation, $G$ and $G'$ are isomorphic over $F$.
Hence we may regard an $L$-parameter $\phi$ of $G$ as an $L$-parameter $\phi'$ of $G'$,
and so that we should obtain two Vogan $L$-packets 
$\Pi_\phi$ of $G$ and $\Pi_{\phi'}$ of $G'$, respectively.
These are subsets of the same set $\sqcup_{G_1}\Irr(G_1(F))$, 
where $G_1$ runs over pure inner forms of $G$ (hence of $G'$).
Do these two $L$-packets coincide?
If so, for Whittaker data $\w$ and $\w'$ for $G$ and $G'$, respectively,
what is the map
\[
\Irr(\pi_0(S_\phi))\xrightarrow{\iota_{\w}^{-1}}
\Pi_\phi=\Pi_{\phi'}
\xrightarrow{\iota_{\w'}}\Irr(\pi_0(S_\phi))?
\] 
Note that Kaletha's result is the case when $G'=G$,
but the problems we need to know are the case
when $G'$ is not necessarily the trivial pure inner form of $G$.
In Appendix \ref{app}, we will give an answer in a special case.
\par

In Appendix \ref{appB}, we show two results on representations of metaplectic groups.
The one is an analogue of a conjecture of Gross--Prasad and Rallis (Theorem \ref{GPRforMp}).
The other is a sufficient condition for the irreducibility of standard modules (Theorem \ref{IS}).

\subsection*{Acknowledgments}
The author is grateful to Professor Atsushi Ichino for his helpful comments and suggestions. 
Thanks are also due to the referees for helpful comments.
This work was supported by the Foundation for Research Fellowships of Japan Society for the Promotion of Science for Young Scientists (DC1) Grant 26-1322. 
\par

\subsection*{Notations}
Let $F$ be a non-archimedean local field with characteristic zero, 
$\oo_F$ be the ring of integers of $F$, 
$\varpi$ be a uniformizer, 
$q$ be the number of elements in the residue class field $\oo_F/\varpi\oo_F$ 
and $|\cdot|_F$ be the normalized absolute value on $F$ so that $|\varpi|_F=q^{-1}$.
We denote by $\Gamma$, $W_F$ and $\WD_F=W_F\times\SL_2(\C)$ 
the absolute Galois, Weil and Weil--Deligne groups of $F$, respectively.
Fix a non-trivial additive character $\psi$ of $F$. 
For $c\in F^\times$, 
we define an additive character $\psi_c$ or $c\psi$ of $F$ by
\[
\psi_c(x)=c\psi(x)=\psi(cx).
\]
Moreover we set $\chi_c=(\cdot,c)$ to be the quadratic character of $F^\times$ 
associated to $c\in F^\times/F^{\times2}$.
Here, $(\cdot,\cdot)$ is the quadratic Hilbert symbol of $F$.
For a totally disconnected locally compact group $G$, 
we denote the set of equivalence classes of irreducible smooth
representations of $G$ by $\Irr(G)$.
If $G$ is the group of $F$-points of a linear algebraic group over $F$,
we denote by $\Irr_\temp(G)$ the subset of $\Irr(G)$
of classes of irreducible tempered representations.
As a representation of $G$ (or its double cover),
we use the symbol $\pi$, $\cl{\pi}$, $\sigma$, $\cl{\sigma}$ or $\tau$.
The contragredient representation of $\pi$ is denoted by $\pi^\vee$.
One should not confuse $\cl{\pi}$ with $\pi^\vee$.
For a topological group $H$,
we define the component group of $H$ by $\pi_0(H)=H/H^\circ$, 
where $H^\circ$ is the identity component of $H$.
The Pontryagin dual (i.e., the character group)
of a finite abelian group $A$ is denoted by $A^D$ or $\widehat{A}$.
\par

%\section{Preliminaries}\label{section spaces}
\section{Preliminaries}\label{section spaces}
In this section, we recall some notions.

%\subsection{Whittaker data}
\subsection{Whittaker data}
Let $G$ be a quasi-split connected reductive group over $F$ and
$B=TU$ be an $F$-rational Borel subgroup of $G$,
where $T$ is a maximal $F$-torus and $U$ is the unipotent radical of $B$.
We denote the center of $G$ by $Z$.
We call a character $\mu$ of $U(F)$ generic if
its stabilizer in $T(F)$ is equal to $Z(F)$. 
A Whittaker datum of $G$ is a conjugacy class of a pair $\w=(B,\mu)$,
where
$B=TU$ is an $F$-rational Borel subgroup of $G$ and
$\mu$ is a generic character of $U(F)$.
We say that $\pi\in\Irr(G(F))$ is $\w$-generic
if 
\[
\Hom_{U(F)}(\pi,\mu)\not=0.
\]
If $\pi$ is $\w$-generic for some Whittaker datum $\w$ of $G$,
then we say that $\pi$ is generic.

%\subsection{Orthogonal spaces}\label{Quadratic}
\subsection{Orthogonal spaces}\label{Quadratic}
Let $V=V_m$ be  an orthogonal space of dimension $m$ over $F$, i.e., 
a vector space equipped with a non-degenerate symmetric bilinear form
\[
\pair{\cdot,\cdot}_V\colon V\times V\rightarrow F.
\]
We take a basis $\{e_1,\dots,e_m\}$ of $V$,
and define the discriminant of an orthogonal space $V$ by
\[
\disc(V)=\disc(V,\pair{\cdot,\cdot}_V)
=2^{-m}
(-1)^{\frac{m(m-1)}{2}}\det((\pair{e_i,e_j}_V)_{i,j})\bmod F^{\times2}\in F^\times/F^{\times2}.
\]
Note that this differs from $\disc(V,q)$ as in \cite[p. 41]{GGP},
which is not used in this paper.
Let $\chi_V=(\cdot,\disc(V))$ 
be the character of $F^\times$ associated with $F(\sqrt{\disc(V)})/F$.
We call $\chi_V$ the discriminant character of $V$.
\par

We denote the anisotropic kernel of $V$ by $V_{{\rm an}}$. 
The special orthogonal group $\SO(V)$ 
is quasi-split if and only if $\dim(V_{{\rm an}})\leq2$.
In this case, we choose a subset $\{v_i,v_i^*|i=1,\dots,n\}$ of $V$ such that
\[
\pair{v_i,v_j}_V=\pair{v_i^*,v_j^*}_V=0,\quad
\pair{v_i,v_j^*}_V=\delta_{i,j},
\]
where $n$ is the integer such that
$m=2n+1$ or $m=2n+2$.
We set
\[
X_k=Fv_1+\dots+Fv_k
\quad\text{and}\quad
X_k^*=Fv_1^*+\dots+Fv_k^*
\]
for $1\leq k\leq n$.
Let $V_{m-2n}$ be the orthogonal complement of $X_n\oplus X_n^*$ in $V$,
so that $V_{m-2n}$ is an orthogonal space of dimensional $m-2n$.
We denote by $B=TU$ the $F$-rational Borel subgroup of $\SO(V)$ stabilizing the complete flag
\[
0\subset \pair{v_1}\subset \pair{v_1,v_2}\subset\dots\subset\pair{v_1,\dots,v_n}=X_n,
\]
where $T$ is the $F$-rational torus stabilizing the lines $Fv_i$ for $i=1,\dots,n$.
\par

If $m=2n+1$, then there is a unique $T$-orbit of generic characters of $U$.
Let $V_1=Fe$.
We choose a generic character $\mu$ of $U$ such that
\[
\mu(u)=\psi(\pair{uv_2,v_1^*}_V+\dots+\pair{uv_n,v_{n-1}^*}_V+\pair{ue,v_n^*}_V)
\]
for $u\in U$.
We define a Whittaker datum $\w$ of $\SO(V)$ by $\w=(B,\mu)$.
\par

Next, we suppose that $m=2n+2$ and $\SO(V)$ is quasi-split.
Then there exist $c,d\in F^\times$ such that
\[
V_2\cong F[X]/(X^2-d)
\]
as vector spaces, and 
the pairing of $V_2$ is given by
\[
(\alpha,\beta)\mapsto \pair{\alpha,\beta}_{V_2}\coloneqq c\cdot\tr(\alpha\overline{\beta}),
\]
where $\beta\mapsto\overline{\beta}$ is the involution on $F[X]/(X^2-d)$
induced by $a+bX\mapsto a-bX$.
In this case,
we say that $V$ is type $(d,c)$. 
Note that $\disc(V)=d\bmod F^{\times2}$.
If $V$ is type $(d,c)$, then we can take $e,e'\in V_2$ such that 
\[
\pair{e,e}_V=2c,\quad 
\pair{e',e'}_V=-2cd
\quad\text{and}\quad
\pair{e,e'}_V=0.
\] 
Then we define a generic character $\mu_{c}$ of $U$ by
\[
\mu_{c}(u)=\psi(\pair{uv_2,v_1^*}_V+\dots+\pair{uv_n,v_{n-1}^*}_V+\pair{ue,v_n^*}_V).
\]
By \cite[\S12]{GGP}, the map $c'\mapsto \mu_{c'}$ gives a bijection (not depending on $\psi$)
\[
cN_{E/F}(E^\times)/F^{\times2}\rightarrow \text{\{$T$-orbits of generic characters of $U$\}},
\]
where $E=F(\sqrt{d})$.
Note that $V$ is both type $(d,c)$ and $(d,c')$ if and only if $c'\in c N_{E/F}(E^\times)$.
We define a Whittaker datum of $\SO(V)$ by $\w_c=(B,\mu_c)$.
Note that $\w_c$ does not depend on the choice of $\psi$.

%\subsection{Symplectic spaces}\label{Symplectic}
\subsection{Symplectic spaces}\label{Symplectic}
Let $W=W_{2n}$ be a symplectic space of dimension $2n$ over $F$, i.e., 
a vector space equipped with a non-degenerate symplectic form
\[
\pair{\cdot,\cdot}_W\colon W\times W\rightarrow F.
\]
The symplectic group $\Sp(W)$ is a split algebraic group over $F$.
We choose a basis $\{w_i,w_i^*|i=1,\dots,n\}$ of $W$ such that
\[
\pair{w_i,w_j}_W=\pair{w_i^*,w_j^*}_W=0,\quad
\pair{w_i,w_j^*}_W=\delta_{i,j}.
\]
We set
\[
Y_k=Fw_1+\dots+Fw_k
\quad\text{and}\quad
Y_k^*=Fw_1^*+\dots+Fw_k^*.
\]
\par

Let $B'=T'U'$ be the $F$-rational Borel subgroup of $\Sp(W)$ stabilizing the complete flag
\[
0\subset \pair{w_1}\subset \pair{w_1,w_2}\subset\dots\subset\pair{w_1,\dots,w_n}=Y_n,
\]
where $T'$ is the $F$-split torus stabilizing the lines $Fw_i$ for $i=1,\dots,n$.
For $c\in F^\times$, we define a generic character $\mu'_c$ of $U'$ by
\[
\mu'_c(u')=\psi(\pair{u'w_2,w_1^*}_W+\dots+\pair{u'w_n,w_{n-1}^*}_W+c\pair{u'w_n^*,w_n^*}_W).
\]
By \cite[\S12]{GGP}, the map $c\mapsto \mu'_c$ gives a bijection (depending on $\psi$)
\[
F^\times/F^{\times2}\rightarrow \text{\{$T'$-orbits of generic characters of $U'$\}}.
\]
We define a Whittaker datum of $\Sp(W)$ by $\w'_c=(B',\mu'_c)$.
We emphasize that $\w_c'$ depends on $\psi$.

%\subsection{Parabolic subgroups}\label{parabolic}
\subsection{Parabolic subgroups}\label{parabolic}
Fix a positive integer $k$. 
Let $V'=V_{2m'}$ be an orthogonal space with $\dim(V')=2m'=2m+2k$ and type $(d,c)$.
We set $X=X_k$ and $X^*=X_k^*$.
Let $V$ be the orthogonal complement of $X\oplus X^*$ in $V'$, so that
$V$ is an orthogonal space of dimensional $2m$ over $F$.
Let $P=P_k=M_PU_P$ be the maximal parabolic subgroup of $\orth(V')$ 
stabilizing $X$, where $M_P$ is the Levi component of $P$ stabilizing $X^*$.
We have
\begin{align*}
M_P&=\{m_P(a)\cdot h\ |\ a\in\GL(X), h\in\orth(V)\}, \\
U_P&=\{u_P(b)\cdot u_P(c)\ |\ b\in\Hom(V,X),c\in\Sym(X^*,X)\},
\end{align*}
where 
\[
m_P(a)=
\begin{pmatrix}
a&&\\
&1_{V}&\\
&&(a^*)^{-1}
\end{pmatrix},\ 
u_P(b)=
\begin{pmatrix}
1_X&b&-\frac{1}{2}bb^*\\
&1_{V}&-b^*\\
&&1_{X^*}
\end{pmatrix},\ 
u_P(c)=
\begin{pmatrix}
1_X&&c\\
&1_{V}&\\
&&1_{X^*}
\end{pmatrix}
\]
and 
\[
\Sym(X^*,X)=\{c\in\Hom(X^*,X)\ |\ c^*=-c\}.
\]
Here, the elements $a^*\in\GL(X^*)$, $b^*\in\Hom(X^*,V)$, and $c^*\in\Hom(X^*,X)$
are defined by requiring that
\[
\pair{ax,x'}_{V'}=\pair{x,a^*x'}_{V'},\ 
\pair{bv,x'}_{V'}=\pair{v,b^*x'}_{V'},\ 
\pair{cx',x''}_{V'}=\pair{x',c^*x''}_{V'}
\]
for $x\in X$, $x',x''\in X^*$ and $v\in V$.
Let $P^\circ=P\cap\SO(V')$ and $M_P^\circ=M_P\cap\SO(V')$,
so that $M_P^\circ\cong\GL(X)\times\SO(V)$.
We put
\[
\rho_P=\frac{2m+k-1}{2},\quad
w_P=
\begin{pmatrix}
&&-I_X\\
&1_{V}&\\
-I_X^{-1}&&
\end{pmatrix},
\]
where $I_X\in\Hom(X^*,X)$ is defined by $I_Xv_i^*=v_i$ for $1\leq i \leq k$.
Then the modulus character $\delta_P$ of $P$ is given by
\[
\delta_P(m_P(a)hu_P)=|\det(a)|_F^{2\rho_P}
\]
for $a\in\GL(X)$, $h\in\orth(V)$ and $u_P\in U_P$.
\par

Similarly, for a fixed positive integer $k$,
we let $W'=W_{2n'}$ be a symplectic space with $\dim(W')=2n'=2n+2k$
and we set $Y=Y_k$ and $Y^*=Y_k^*$.
We define $W$ with $\dim(W)=2n$, $Q=M_QU_Q\subset \Sp(W')$, 
$m_Q(a)$, $u_Q(b)$, $u_Q(c)$, 
$\Sym(Y^*,Y)=\{c\in\Hom(Y^*,Y)\ |\ c^*=-c\}$ and $I_Y$ as above.
We put
\[
\rho_Q=\frac{2n+k+1}{2},\quad
w_Q=
\begin{pmatrix}
&&-I_Y\\
&1_{W}&\\
I_Y^{-1}&&
\end{pmatrix}.
\]
Then the modulus character $\delta_Q$ of $Q$ is given by
\[
\delta_Q(m_Q(a')gu_Q)=|\det(a)|_F^{2\rho_Q}
\]
for $a'\in\GL(Y')$, $g\in \Sp(W)$ and $u_Q\in U_Q$.

%\subsection{Representations of $\SO(V)$ and $\orth(V)$}\label{SOvsO}
\subsection{Representations of $\SO(V)$ and $\orth(V)$}\label{SOvsO}
Let $V=V_m$ be an orthogonal space over $F$ of dimension $m$.
In this subsection, we recall some results about representations of $\SO(V)$ and $\orth(V)$.
Note that any irreducible representation of $\orth(V)$ is self-dual
by a result in \cite[Chapter 4.~$\mathrm{II}$.~1]{MVW}. 
We fix $\ep\in\orth(V)\setminus\SO(V)$.
\par

First we assume that $m=2n+1$ is odd.
Then we can take $\ep$ in the center of $\orth(V)$ and we have
\[
\orth(V)=\SO(V)\times\pair{\ep}.
\]
Hence, if $\cl{\sigma}\in\Irr(\orth(V))$, then 
$\cl{\sigma}|\SO(V)$ is also irreducible.
Moreover for $\sigma\in\Irr(\SO(V))$, 
there are exactly two extensions of $\sigma$ to $\orth(V)$.
In particular, $\sigma$ is self-dual.
\par

Next we assume that $m=2n$ is even.
For $\sigma\in\Irr(\SO(V))$, we denote by $\sigma^\ep$ the representation given by conjugating
$\sigma$ by $\ep$.
The next proposition follows from the Clifford theory
(e.g., see \cite[Lemma 4.1]{BJ}).
\begin{prop}\label{Clifford}
\begin{enumerate}
\item\label{epinv}
For $\sigma\in\Irr(\SO(V))$, the following are equivalent:
\begin{itemize}
\item
$\sigma^\ep\cong\sigma$;
\item
there exists $\cl{\sigma}\in\Irr(\orth(V))$ such that $\cl{\sigma}|\SO(V)\cong \sigma$;
\item
the induction $\Ind_{\SO(V)}^{\orth(V)}(\sigma)$ is reducible;
\item
$\Ind_{\SO(V)}^{\orth(V)}(\sigma)\cong \cl\sigma\oplus(\cl\sigma\otimes\det)$
for any $\cl\sigma\in\Irr(\orth(V))$ with $\cl{\sigma}|\SO(V)\cong \sigma$.
\end{itemize}
\item
For $\cl{\sigma}\in\Irr(\orth(V))$, the following are equivalent:
\begin{itemize}
\item
$\cl{\sigma}\otimes\det\cong\cl{\sigma}$;
\item
there exists $\sigma\in\Irr(\SO(V))$ such that $\Ind_{\SO(V)}^{\orth(V)}(\sigma)\cong \cl\sigma$;
\item
the restriction $\cl{\sigma}|\SO(V)$ is reducible;
\item
$\cl\sigma|\SO(V)\cong \sigma\oplus\sigma^\ep$
for any $\sigma\in\Irr(\SO(V))$ with $\Ind_{\SO(V)}^{\orth(V)}(\sigma)\cong\cl\sigma$.
\end{itemize}
\end{enumerate}
\end{prop}
If $\sigma\in\Irr(\SO(V))$ satisfies the conditions of Proposition \ref{Clifford} (\ref{epinv}),
we say that $\sigma$ is $\epinv$.
We summarize these results in Table \ref{tab1} below.
\begin{table}[htb]
\begin{center}
\caption{}\label{tab1}
\begin{tabular}{|c|c|c|} \hline
   $\sigma$ &$\epinv$&not $\epinv$\\ \hline\hline
     $\sigma^\ep$  &$\sigma\cong\sigma^\ep$  &$\sigma\not\cong\sigma^\ep$\\ \hline
     $\Ind_{\SO(V)}^{\orth(V)}(\sigma)$ &  $\cl\sigma\oplus(\cl\sigma\otimes\det)$   & irreducible  \\ \hline
     $\cl{\sigma}|\SO(V)=\sigma$ &exists& not exist\\\hline
    \end{tabular}
\end{center}
\end{table}
\par

Since any irreducible representation of $\orth(V)$ is self-dual,
for $\sigma\in\Irr(\SO(V))$, we have
\[
\sigma^\vee\cong \sigma 
\quad\text{or}\quad
\sigma^\vee\cong \sigma^\ep.
\]
More precisely, we have the following proposition.
\begin{prop}\label{contragredient}
Let $\sigma\in\Irr(\SO(V))$. 
Then we have
\[
\sigma^\vee\cong
\left\{
\begin{aligned}
&\sigma		\iif \dim(V)\equiv 0 \bmod 4,\\
&\sigma^\ep	\iif \dim(V)\equiv 2 \bmod 4.
\end{aligned}
\right.
\]
\end{prop}
\begin{proof}
See \cite[Proposition 5.3]{GP}.
\end{proof}
\par

The group $\orth(V)$ acts on $\Irr(\SO(V))$ by conjugation.
\begin{lem}\label{well-defined}
Let $n>0$ and $V=V_{2n}$ be an orthogonal space.
\begin{enumerate}
\item\label{well1}
For $(\sigma,\VV_{\sigma})\in\Irr(\SO(V))$ and $\ep\in\orth(V)\setminus\SO(V)$, we have
\[
\Ind_{\SO(V)}^{\orth(V)}(\sigma)\cong\Ind_{\SO(V)}^{\orth(V)}(\sigma^\ep).
\]
\item\label{well2}
We put $V'=X_1+V+X_1^*$ with $\dim(V')=2n+2$.
Let $P=P_1=M_PU_P\subset\orth(V')$ be a parabolic subgroup
with $M_P\cong \GL_1(F)\times \orth(V)$.
Then for $\sigma\in\Irr_\temp(\SO(V))$, 
$\ep\in\orth(V)\setminus\SO(V)$ 
and a unitary character $\chi$ of $F^\times=\GL_1(F)$, we have
\[
\Ind_{P^\circ}^{\SO(V')}(\chi \otimes \sigma)\cong
\Ind_{P^\circ}^{\SO(V')}(\chi^{-1} \otimes \sigma^\ep).
\]
\item\label{well3}
For $k > 0$, we put $V'=X_k+V+X_k^*$.
Let $P=P_k=M_PU_P\subset\orth(V')$ be a parabolic subgroup
with $M_P\cong \GL(X_k)\times\orth(V)$.
Fix $\tau\in\Irr(\GL(X_k))$ and $\cl{\sigma}\in\Irr(\orth(V))$.
Assume that $\sigma\coloneqq\cl{\sigma}|\SO(V)$ is irreducible.
Then there is a canonical isomorphism
\[
\Ind_{P}^{\orth(V')}(\tau\otimes\cl{\sigma})|\SO(V')
\cong
\Ind_{P^\circ}^{\SO(V')}(\tau\otimes \sigma).
\]
of representations of $\SO(V')$.
\end{enumerate}
\end{lem}
\begin{proof}
(\ref{well1})
The induced representation $\Ind_{\SO(V)}^{\orth(V)}(\sigma)$ can be realized on
\[
X_{\sigma}\coloneqq\{f\colon \orth(V)\rightarrow \VV_\sigma \text{ smooth}\ |\ 
f(hg)=\sigma(h)f(g)\text{ for any }h\in\SO(V),g\in\orth(V)\}
\]
by $(g\cdot f)(x)=f(xg)$.
Then the map
\[
f\mapsto [x\mapsto f(\ep^{-1}x)]
\]
gives an isomorphism $X_\sigma\rightarrow X_{\sigma^\ep}$.
\par

(\ref{well2})
Note that $\Ind_{P^\circ}^{\SO(V')}(\chi \otimes \sigma)$ is 
a direct sum of irreducible representations of $\SO(V')$.
In particular, we have
\[
\Ind_{P^\circ}^{\SO(V')}(\chi \otimes \sigma)^\vee \cong 
\left\{
\begin{aligned}
&\Ind_{P^\circ}^{\SO(V')}(\chi \otimes \sigma)^\epsilon 
\cong \Ind_{P^\circ}^{\SO(V')}(\chi \otimes \sigma^\epsilon)
\iif \dim(V) \equiv 0 \bmod 4,\\
&\Ind_{P^\circ}^{\SO(V')}(\chi \otimes \sigma) \iif \dim(V) \equiv 2 \bmod 4
\end{aligned}
\right.
\]
by Proposition \ref{contragredient}.
Again, by the same proposition, we have
\[
\Ind_{P^\circ}^{\SO(V')}(\chi^{-1} \otimes \sigma^\vee) \cong
\left\{
\begin{aligned}
&\Ind_{P^\circ}^{\SO(V')}(\chi^{-1} \otimes \sigma) \iif \dim(V) \equiv 0 \bmod 4,\\
&\Ind_{P^\circ}^{\SO(V')}(\chi^{-1} \otimes \sigma^\epsilon) \iif \dim(V) \equiv 2 \bmod 4.
\end{aligned}
\right.
\]
Hence the assertion follows that
\[
\Ind_{P^\circ}^{\SO(V')}(\chi \otimes \sigma)^\vee
\cong
\Ind_{P^\circ}^{\SO(V')}(\chi^{-1} \otimes \sigma^\vee).
\]
\par

(\ref{well3})
The representation $\Ind_{P}^{\orth(V')}(\tau\otimes\cl{\sigma})$
can be realized on a space of smooth functions $f$ on $\orth(V')$.
The map $f\mapsto f|\SO(V')$ gives a desired isomorphism.
\end{proof}
\par

We define equivalent relations $\sim_{\det}$ on $\Irr(\orth(V))$ and
$\sim_\ep$ on $\Irr(\SO(V))$ by
\[
\cl\sigma\sim_{\det}\cl\sigma\otimes\det
\quad\text{and}\quad
\sigma\sim_{\ep}\sigma^\ep
\]
for $\cl\sigma\in\Irr(\orth(V))$ and $\sigma\in\Irr(\SO(V))$.
Note that $\cl\sigma|\SO(V)\cong(\cl\sigma\otimes\det)|\SO(V)$ and
$\Ind_{\SO(V)}^{\orth(V)}(\sigma)\cong\Ind_{\SO(V)}^{\orth(V)}(\sigma^\ep)$.
Hence, by Proposition \ref{Clifford}, the restriction and the induction give a canonical bijection
\[
\Irr(\orth(V))/\sim_{\det}\longleftrightarrow \Irr(\SO(V))/{\sim_\ep}.
\]
For $\sigma\in\Irr(\SO(V))$, we denote the equivalence class of $\sigma$ by
\[
[\sigma]=\{\sigma,\sigma^\ep\}\in\Irr(\SO(V))/{\sim_\ep}.
\]

%\subsection{Theta lifts}
\subsection{Theta lifts}
We introduce the local theta correspondence induced by a Weil representation 
$\omega_{\psi,V,W}$ of $\Sp(W)\times\orth(V)$ 
when $\dim(W)=2n$ and $\dim(V)=2m$, 
and recall some basic general results.
\par

We fix a non-trivial additive character $\psi$ of $F$.
Let $W=W_{2n}$ and $V=V_{2m}$. 
We denote a Weil representation of $\Sp(W)\times\orth(V)$ by $\omega=\omega_{\psi,V,W}$.
For $\pi\in\Irr(\Sp(W))$,
the maximal $\pi$-isotypic quotient of $\omega$
is of the form
\[
\pi\boxtimes \Theta(\pi),
\]
where $\Theta(\pi)=\Theta_{\psi,V,W}(\pi)$ is a smooth finite length representation of $\orth(V)$.
Hence, there exists an exact sequence of $\Sp(W)\times\orth(V)$-modules:
\[
\begin{CD}
1@>>>\Sch[\pi]@>>>\omega @>>>\pi\boxtimes \Theta(\pi)@>>>1,
\end{CD}
\]
where the kernel $\Sch[\pi]$ is given by
\[
\Sch[\pi]=\bigcap_{f\in\Hom_{\Sp(W)}(\omega,\pi)}\ker(f).
\]
\par

Similarly, for $\cl{\sigma}\in\Irr(\orth(V))$, 
we obtain a smooth finite length representation 
$\Theta(\cl{\sigma})=\Theta_{\psi,V,W}(\cl{\sigma})$ of $\Sp(W)$.
The maximal semi-simple quotient of $\Theta(\pi)$ (\resp $\Theta(\cl{\sigma})$)
is denoted by $\theta(\pi)=\theta_{\psi,V,W}(\pi)$ 
(\resp $\theta(\cl{\sigma})=\theta_{\psi,V,W}(\cl{\sigma})$).
The Howe duality conjecture, 
which was proven by Waldspurger \cite{Wa1} when the residue characteristic is not $2$
and by Gan--Takeda \cite{GT1, GT2} in general,
says that $\theta(\pi)$ and $\theta(\cl{\sigma})$
are irreducible (if they are nonzero).
Moreover, we have the following:
\begin{prop}[{\cite[Proposition C.4 $(\mathrm{i})$]{GI1}}]\label{Theta}
We set $V=V_{2m}$ and $W=W_{2n}$. 
Let $\pi\in\Irr_\temp(\Sp(W))$ $($\resp $\cl{\sigma}\in\Irr_\temp(\orth(V)))$ be a tempered representation.
\begin{enumerate}
\item
Suppose that $m=n+1$.
Then $\Theta(\pi)$ is either zero or an irreducible tempered representation of $\orth(V)$
so that $\Theta(\pi)=\theta(\pi)$.
\item
Suppose that $m=n$.
Then $\Theta(\cl{\sigma})$ is either zero or an irreducible tempered representation of $\Sp(W)$
so that $\Theta(\cl{\sigma})=\theta(\cl{\sigma})$.
\end{enumerate}
\end{prop}

%\section{Local Langlands correspondence}\label{LLC}
\section{Local Langlands correspondence}\label{LLC}
Through this paper, 
we assume the local Langlands correspondence for symplectic and special orthogonal groups. 
See \cite{A}.
In this section, we summarize some of its properties which are used in this paper.
See also Appendix \ref{expectation} and \cite{GGP}.

%\subsection{Representations of $\WD_F$}
\subsection{Representations of $\WD_F$}
Let $M$ be a finite dimensional vector space over $\C$.
We say that a homomorphism $\phi\colon\WD_F\rightarrow \GL(M)$ is 
a representation of $\WD_F$
if
\begin{itemize}
\item
$\phi(\Frob)$ is semi-simple, where $\Frob\in W_F$ is a geometric Frobenius;
\item
the restriction of $\phi$ to $W_F$ is smooth;
\item
the restriction of $\phi$ to $\SL_2(\C)$ is algebraic.
\end{itemize}
We call $\phi$ tempered if the image of $W_F$ is bounded.
We say that $\phi$ is orthogonal or self-dual with sign $+1$ 
(\resp symplectic or self-dual with sign $-1$) 
if there exists a non-degenerate bilinear form 
$B\colon M\times M\rightarrow \C$ such that
\[
\left\{
\begin{aligned}
&B(\phi(w)x,\phi(w)y)=B(x,y),\\
&B(y,x)=\delta \cdot B(x,y)
\end{aligned}
\right.
\]
with $\delta = +1$ (\resp $\delta = -1$)
for $x,y\in M$ and $w\in\WD_F$.
In this case, $\phi$ is self-dual, i.e.,
$\phi$ is equivalent to its contragredient $\phi^\vee$.
\par

For any positive integer $n$, there exists a unique irreducible algebraic
representation $\nu_n$ of $\SL_2(\C)$ with dimension $n$.
It is easy to see that
\[
\nu_n\text{ is }
\left\{
\begin{aligned}
&\text{orthogonal}	\iif \text{$n$ is odd},\\
&\text{symplectic}	\iif \text{$n$ is even}.
\end{aligned}
\right.
\]
\par

For a representation $M$ of $\WD_F$, 
we define the $L$-factor $L(s,M)$ and the $\ep$-factor $\ep(s,M,\psi)$
as in \cite{T}.
It is well-known that
if $M$ is a symplectic representation of $\WD_F$,
then $\ep(M)=\ep\left(1/2,M,\psi\right)$ 
does not depend on the choice of $\psi$, and $\ep(M)\in\{\pm1\}$
(see, e.g., \cite[Proposition 5.1]{GGP}).
\par

%\subsection{Properties}\label{LLCproperty}
\subsection{Properties}\label{LLCproperty}

Let $G$ be a quasi-split connected reductive algebraic group over $F$.
We denote the Langlands dual group and the $L$-group of $G$ by
$\widehat{G}$ and $\Lgp{G}=\widehat{G}\rtimes W_F$, respectively.
We call an $L$-homomorphism
\[
\bphi\colon\WD_F\rightarrow \Lgp{G}
\]
an $L$-parameter of $G$.
We say that two $L$-parameters are equivalent if
they are conjugate by an element in $\widehat{G}$.
We call an $L$-parameter $\bphi$ tempered
if $\bphi(W_F)$ projects onto a relatively compact subset of $\widehat{G}$.
We denote the set of equivalence classes of $L$-parameters of $G$ by $\Phi(G)$.
Let $\Phi_\temp(G)$ be the subset of $\Phi(G)$ of classes of tempered $L$-parameters.
For $\bphi\in\Phi(G)$, we put
\[
S_{\bphi}=\cent(\im(\bphi),\widehat{G}).
\]
\par

A pure inner twist of $G$ is a tuple $(G',f,z)$, 
where
$G'$ is an algebraic group over $F$, 
$f\colon G\rightarrow G'$ is an isomorphism over $\overline{F}$
and $z\in Z^1(F,G)$ is a $1$-cocycle such that $f^{-1}\circ \gamma(f)=\Int(z_\gamma)$
for any $\gamma\in\Gamma$.
Here, $z_\gamma \in G$ is the value at $\gamma \in \Gamma$ of 
$z \colon \Gamma \rightarrow G(\overline{F})$
and $\Int(z_\gamma)g = z_\gamma g z_\gamma^{-1}$.
Then we call $G'$ a pure inner form of $G$.
There exists a canonical bijection
\[
H^1(F,G)\longleftrightarrow\{\text{The isomorphism classes of pure inner twists of $G$}\}.
\]
More precisely, see Appendix \ref{A.2}.
Moreover, since $F$ is $p$-adic, 
there exists a canonical bijection
(Kottwitz's isomorphism)
\[
H^1(F,G)\rightarrow \pi_0(Z(\widehat{G})^{\Gamma})^D.
\]
See \cite[Proposition 6.4]{Kot3}.
Note that $\pi_0(Z(\widehat{G})^{\Gamma})$ is a central subgroup of $\pi_0(S_{\bphi})$.
\par

We are now ready to describe the desiderata for the Langlands correspondence.
\begin{enumerate}
\item\label{des1}
There exists a canonical surjection
\[
\bigsqcup_{G'}\Irr(G'(F))\rightarrow \Phi(G),
\]
where $G'$ runs over all pure inner forms of $G$.
For $\bphi\in\Phi(G)$, we denote by $\Pi_{\bphi}$ 
the inverse image of $\bphi$ under this map,
and call $\Pi_{\bphi}$ the $L$-packet of $\bphi$.
The $L$-packet $\Pi_{\bphi}$ is a finite set whose order is equal to $\#\Irr(\pi_0(S_{\bphi}))$.
If $\bphi\in\Phi_\temp(G)$, then $\Pi_{\bphi}\subset\sqcup_{G'}\Irr_\temp(G')$.
\item\label{des2}
For a Whittaker datum $\w$ of $G$,
there exists a bijection
\[
\iota_\w\colon \Pi_{\bphi}\rightarrow \Irr(\pi_0(S_{\bphi}))
\]
which satisfies some character identities. 
More precisely, see Appendix \ref{expectation}.
\item\label{des4}
The local adjoint $L$-function $L(s,\bphi,\Ad)$ is regular at the point $s=1$
if and only if
$\iota_\w^{-1}(\1) \in \Pi_{\bphi}$ is $\w$-generic for any Whittaker datum $\w$.
In this case, we say that the $L$-parameter $\bphi$ is generic.
\item\label{des5}
By the restriction to $\pi_0(Z(\widehat{G})^{\Gamma}) \subset \pi_0(S_{\bphi})$, 
an irreducible representation $\eta$ of $\pi_0(S_{\bphi})$ gives a character of 
$\pi_0(Z(\widehat{G})^{\Gamma})$,
which corresponds to a class of $H^1(F,G)$, i.e., an isomorphism class of pure inner twists $(G',f,z)$.
Then $\iota_\w^{-1}(\eta)\in\Irr(G'(F))$.
\end{enumerate}
Once we fix a Whittaker datum $\w$ of $G$,
we denote by $\pi(\eta)$ the representation in $\Pi_{\bphi}$
corresponding to $\eta\in\Irr(\pi_0(S_{\bphi}))$ via $\iota_\w$.
\par

The property (\ref{des4}) is a conjecture of Gross--Prasad and Rallis (GPR), which was proven
in general in \cite[Appendix B]{GI2}.
Note that if $\bphi\in\Phi_\temp(G)$, then the $L$-function $L(s,\bphi,\Ad)$ is regular 
in the half-plane $\re(s)>0$, so that $\bphi$ is generic by (\ref{des4}).
\par

In the rest of this section, we explain the desiderata of $L$-packets more precisely
for symplectic and special orthogonal groups,
which will be assumed in the rest of this paper. 

%\subsection{Type $B_n$}
\subsection{Type $B_n$}
Let $n\geq1$ and $G=\SO(V_{2n+1})$ be a quasi-split special orthogonal group.
Then $\widehat{G}=\Sp_{2n}(\C)$ and $\Lgp{G}=\Sp_{2n}(\C)\times W_F$.
An $L$-parameter $\bphi$ of $\SO(V_{2n+1})$ gives a self-dual representation 
$\phi\colon\WD_F\rightarrow \GL(M)$
with sign $-1$, $\dim_\C(M)=2n$ and $\det(M)=\1$.
The map $\bphi\mapsto \phi$ gives a bijection
\[
\Phi(\SO(V_{2n+1}))\rightarrow
\cl\Phi(\SO(V_{2n+1}))\coloneqq
\{\phi\colon\WD_F\rightarrow \Sp(M)\ |\ \dim_\C(M)=2n\}/\cong.
\]
We identify $\bphi$ with $\phi$ via this bijection.
Let $\cl\Phi_\temp(\SO(V_{2n+1}))$ be the subset of $\cl\Phi(\SO(V_{2n+1}))$ of 
tempered representations.
\par

Let $\phi\in\cl\Phi(\SO(V_{2n+1}))$.
We denote the centralizer of $\im(\phi)$ in $\Sp(M)$ by $C_\phi$
and its component group by $A_\phi=C_\phi/C_\phi^\circ$.
Then $\pi_0(S_{\bphi})\cong A_\phi$ if $\bphi\mapsto \phi$.
If 
\[
M\cong \bigoplus_{i\in I^+}(V_i\otimes M_i)
\oplus\bigoplus_{i\in I^-}(W_i\otimes N_i)
\oplus\bigoplus_{j\in J}(U_j\otimes (P_j+P_j^\vee)),
\]
where
\begin{itemize}
\item $M_i$ (\resp $N_i$) is an irreducible self-dual representation of $\WD_F$ 
with sign $-1$ (\resp $+1$),
\item $P_j$ is an irreducible representation of $\WD_F$ which is not self-dual,
\end{itemize}
such that $M_i$, $N_i$ and $P_j$ are pairwise inequivalent, 
and $V_i$, $W_i$ and $U_j$ are multiplicity spaces,
then the bilinear form on $M$ gives a non-degenerate bilinear form on $V_i$ (\resp $W_i$) of sign $+1$ (\resp $-1$).
Moreover, we have
\[
C_\phi\cong\prod_{i\in I^+}\orth(V_i)\times\prod_{i\in I^-}\Sp(W_i)\times\prod_{j\in J} \GL(U_j).
\]
For a semi-simple element $a\in C_\phi$, we denote by $M^a$
the $(-1)$-eigenspace of $a$ on $M$.
Then $M^a$ gives a symplectic representation of $\WD_F$.
If $a_i\in\orth(V_i)\setminus\SO(V_i)$, then $M^{a_i}=V_i^{a_i}\otimes M_i$.
We have
\[
A_\phi=\bigoplus_{i\in I^+}(\Z/2\Z)a_i\cong(\Z/2\Z)^m,
\]
where $m=\#I^+$.
\par

It is known that $\#H^1(F,\SO(V_{2n+1}))=2$ (which also follows from Kottwitz's isomorphism).
Hence there are exactly two pure inner forms of $G=\SO(V_{2n+1})$.
They are $\SO(V_{2n+1})$ and $\SO(V_{2n+1}')$, where
$\dim(V_{2n+1})=\dim(V_{2n+1}')=2n+1$ and $\disc(V_{2n+1})=\disc(V_{2n+1}')$.
We see that
$\SO(V_{2n+1})$ is split, but 
$\SO(V_{2n+1}')$ is not quasi-split.
\par

By \cite[Theorem 1.5.1 (b)]{A},
for $\phi\in\cl\Phi_\temp(\SO(V_{2n+1}))$, 
we obtain an $L$-packet $\Pi_\phi^A$, which is a finite subset of $\Irr_\temp(\SO(V_{2n+1}))$,
and a canonical bijection $\iota_\w^A\colon \Pi_\phi^A\rightarrow \pi_0(C_\phi/\{\pm1\})\widehat{\ }$.
We assume that there is a Vogan $L$-packet 
$\Pi_\phi\subset \Irr_\temp(\SO(V_{2n+1}))\sqcup\Irr_\temp(\SO(V_{2n+1}'))$ such that
$\Pi_\phi\cap \Irr_\temp(\SO(V_{2n+1}))=\Pi_\phi^A$
together with an extension $\iota_\w\colon \Pi_\phi\rightarrow \widehat{A_\phi}$ of $\iota_\w^A$.
We extend $L$-packets for general $\phi\in\cl\Phi(\SO(V_{2n+1}))$ as follows.
In general, we can write
\[
\phi=\phi_1\oplus\dots\oplus\phi_r\oplus\phi_0
\oplus\phi_r^\vee\oplus\dots\oplus\phi_1^\vee,
\]
where
\begin{itemize}
\item
$\phi_i$ is a $k_i$-dimensional representation of $\WD_F$ for $i=1,\dots,r$,
which is of the form
$\phi_i=\phi_i'\otimes|\cdot|_F^{s_i}$
for some tempered representation $\phi_i'$ of $\WD_F$ and some real number $s_i$ such that
\[
s_1>\dots>s_r>0,\quad n=n_0+(k_1+\dots+k_r);
\]
\item
$\phi_0\in\cl\Phi_\temp(\SO(V_{2n_0+1}))$ if $n_0\geq1$, and $\phi_0=0$ if $n_0=0$.
\end{itemize}
We write
\[
V_{2n+1}^\bullet=X_1\oplus\dots\oplus X_r\oplus V_{2n_0+1}^\bullet
\oplus X_r^*\oplus\dots\oplus X_1^*,
\]
where
$X_i$ and $X_i^*$ are $k_i$-dimensional totally isotropic subspaces of $V_{2n+1}^\bullet$
which are mutually orthogonal, and
such that $X_i\oplus X_i^*$ is non-degenerate
and orthogonal to $V_{2n_0+1}^\bullet$.
Let $P^\circ$ be the parabolic subgroup of $\SO(V_{2n+1}^\bullet)$ stabilizing the flag
\[
X_1\subset X_1\oplus X_2\subset\dots\subset X_1\oplus\dots\oplus X_r
\]
and $M_P^\circ$ be the Levi component of $P^\circ$ stabilizing the flag
\[
X_1^*\subset X_1^*\oplus X_2^*\subset\dots\subset X_1^*\oplus\dots\oplus X_r^*,
\]
so that
\[
M_P^\circ\cong \GL(X_1)\times\dots\times \GL(X_r)\times\SO(V_{2n_0+1}^\bullet).
\]
Then the $L$-packet $\Pi_\phi$ consists of
the unique irreducible quotients $\pi$ of the standard modules
\[
\Ind_{P^\circ}^{\SO(V_{2n+1}^\bullet)}(\tau_1\otimes\dots\otimes\tau_r\otimes\pi_0),
\]
where $\tau_i$ is the irreducible essentially tempered representation of $\GL(X_i)$
corresponding to $\phi_i$ for $i=1,\dots,r$, and
$\pi_0$ runs over elements in $\Pi_{\phi_0}$
(or $\pi_0=\1$ if $n_0=0$ so that $\SO(V_{2n_0+1})=\{1\}$).
In particular, if $n_0=0$, then $\Pi_\phi$ is a singleton.
If $n_0\geq1$, then the natural map $A_{\phi_0}\rightarrow A_\phi$ is an isomorphism,
and we define $\iota_\w(\pi)\in\widehat{A_\phi}$ by
\[
\iota_\w(\pi_0)
=\iota_\w(\pi)|A_{\phi_0}.
\]
\par

We say that $\phi\in\cl\Phi(\SO(V_{2n+1}))$ is generic if $\Pi_\phi$ has a generic representation.
Let $\cl\Phi_\gen(\SO(V_{2n+1}))$ be the subset of $\cl\Phi(\SO(V_{2n+1}))$ consisting of
generic representations.
Then we have a sequence
\[
\cl\Phi_\temp(\SO(V_{2n+1}))\subset \cl\Phi_\gen(\SO(V_{2n+1}))\subset \cl\Phi(\SO(V_{2n+1})).
\]

%\subsection{Type $C_n$}
\subsection{Type $C_n$}\label{Cn}
Let $n\geq1$ and $G=\Sp(W_{2n})$ be a (split) symplectic group.
Then $\widehat{G}=\SO_{2n+1}(\C)$ and $\Lgp{G}=\SO_{2n+1}(\C)\times W_F$.
An $L$-parameter $\bphi$ of $G$ gives a self-dual representation 
$\phi\colon\WD_F\rightarrow \GL(N)$
with sign $+1$, $\dim_\C(N)=2n+1$ and $\det(N)=\1$.
The map $\bphi\mapsto \phi$ gives a bijection
\[
\Phi(\Sp(W_{2n}))\rightarrow
\cl\Phi(\Sp(W_{2n}))\coloneqq
\{\phi\colon\WD_F\rightarrow \SO(N)\ |\ \dim_\C(N)=2n+1\}/\cong.
\]
We identify $\bphi$ with $\phi$ via this bijection.
Let $\cl\Phi_\temp(\Sp(W_{2n}))$ be the subset of $\cl\Phi(\Sp(W_{2n}))$ 
of tempered representations.
\par

Let $\phi\in\cl\Phi(\Sp(W_{2n}))$.
We denote the centralizer of $\im(\phi)$ in $\orth(N)$ by $C_\phi$
and its component group by $A_\phi=C_\phi/C_\phi^\circ$.
The centralizer of $\im(\phi)$ in $\SO(N)$ is denoted by $C_\phi^+$.
We put $A_\phi^+=\im(C_\phi^+\rightarrow A_\phi)$.
Then $\pi_0(S_{\bphi})\cong A_\phi^+$ if $\bphi\mapsto \phi$.
If 
\[
N\cong \bigoplus_{i\in I^+}(V_i\otimes N_i)
\oplus\bigoplus_{i\in I^-}(W_i\otimes M_i)
\oplus\bigoplus_{j\in J}(U_j\otimes (P_j+P_j^\vee))
\]
where 
\begin{itemize}
\item $M_i$ (\resp $N_i$) is an irreducible self-dual representation of $\WD_F$ 
with sign $-1$ (\resp $+1$),
\item $P_j$ is an irreducible representation of $\WD_F$ which is not self-dual,
\end{itemize}
such that $M_i$, $N_i$ and $P_j$ are pairwise inequivalent, 
and $V_i$, $W_i$ and $U_j$ are multiplicity spaces,
then the bilinear form on $N$ gives a non-degenerate bilinear form on $V_i$ (\resp $W_i$) of sign $+1$ (\resp $-1$).
Moreover, we have
\[
C_\phi\cong\prod_{i\in I^+}\orth(V_i)\times\prod_{i\in I^-}\Sp(W_i)\times\prod_{j\in J} \GL(U_j).
\]
For a semi-simple element $a\in C_\phi$, we denote by $N^a$
the $(-1)$-eigenspace of $a$ on $N$.
Then $N^a$ gives an orthogonal representation of $\WD_F$.
If $a_i\in\orth(V_i)\setminus\SO(V_i)$, then $N^{a_i}=V_i^{a_i}\otimes N_i$.
We have
\[
A_\phi=\bigoplus_{i\in I^+}(\Z/2\Z)a_i\cong(\Z/2\Z)^m,
\]
where $m=\#I^+$.
We see that
$A_\phi^+$ is the kernel of the character
\[
A_\phi\ni a_i\mapsto (-1)^{\dim_\C(N_i)}\in\{\pm1\}.
\]
Hence we have $A_\phi^+\cong(\Z/2\Z)^{m-1}$
since some $N_i$ has odd dimension.
\par

It is known that $\#H^1(F,\Sp(W_{2n}))=1$ (which also follows from Kottwitz's isomorphism).
Hence there are no non-trivial pure inner forms of $G=\Sp(W_{2n})$.
\par

By \cite[Theorem 1.5.1 (b)]{A},
for $\phi\in\cl\Phi_\temp(\Sp(W_{2n}))$, 
we obtain an $L$-packet $\Pi_\phi$, which is a finite subset of $\Irr_\temp(\Sp(W_{2n}))$,
and a canonical bijection 
$\iota_{\w_c'}\colon \Pi_\phi\rightarrow \widehat{A_\phi^+}$
for each Whittaker datum $\w_c'$ of $\Sp(W_{2n})$.
We extend $L$-packets for general $\phi\in\cl\Phi(\Sp(W_{2n}))$ as follows.
In general, we can write
\[
\phi=\phi_1\oplus\dots\oplus\phi_r\oplus\phi_0
\oplus\phi_r^\vee\oplus\dots\oplus\phi_1^\vee,
\]
where
\begin{itemize}
\item
$\phi_i$ is a $k_i$-dimensional representation of $\WD_F$ for $i=1,\dots,r$,
which is of the form
$\phi_i=\phi_i'\otimes|\cdot|_F^{s_i}$
for some tempered representation $\phi_i'$ of $\WD_F$ and some real number $s_i$ such that
\[
s_1>\dots>s_r>0,\quad n=n_0+(k_1+\dots+k_r);
\]
\item
$\phi_0\in\cl\Phi_\temp(\Sp(W_{2n_0}))$ if $n_0\geq1$, and $\phi_0=\1$ if $n_0=0$.
\end{itemize}
We write
\[
W_{2n}=Y_1\oplus\dots\oplus Y_r\oplus W_{2n_0}\oplus Y_r^*\oplus\dots\oplus Y_1^*,
\]
where
$Y_i$ and $Y_i^*$ are $k_i$-dimensional totally isotropic subspaces of $W_{2n}$
which are mutually orthogonal, and
such that $Y_i\oplus Y_i^*$ is non-degenerate and orthogonal to $W_{2n_0}$.
Let $Q$ be the parabolic subgroup of $\Sp(W_{2n})$ stabilizing the flag
\[
Y_1\subset Y_1\oplus Y_2\subset\dots\subset Y_1\oplus\dots\oplus Y_r
\]
and $M_{Q}$ be the Levi component of $Q$ stabilizing the flag
\[
Y_1^*\subset Y_1^*\oplus Y_2^*\subset\dots\subset Y_1^*\oplus\dots\oplus Y_r^*,
\]
so that
\[
M_{Q}\cong \GL(Y_1)\times\dots\times \GL(Y_r)\times\Sp(W_{2n_0}).
\]
Then the $L$-packet $\Pi_\phi$ consists of
the unique irreducible quotients $\pi$ of the standard modules
\[
\Ind_{Q}^{\Sp(W_{2n})}(\tau_1\otimes\dots\otimes\tau_r\otimes\pi_0),
\]
where $\tau_i$ is the irreducible essentially tempered representation of $\GL(Y_i)$
corresponding to $\phi_i$ for $i=1,\dots,r$, and
$\pi_0$ runs over elements in $\Pi_{\phi_0}$ if $n_0>0$.
If $n_0=0$, then we ignore $\pi_0$.
In particular, if $n_0=0$, then $\Pi_\phi$ is a singleton.
If $n_0\geq1$, then the natural map $A_{\phi_0}^+ \rightarrow A_\phi^+$ is an isomorphism,
and we define $\iota_{\w'_c}(\pi)\in\widehat{A_\phi^+}$ by
\[
\iota_{\w'_c}(\pi_0)
=\iota_{\w'_c}(\pi)|A_{\phi_0}^+.
\]
\par

Recall that the set of Whittaker data of $\Sp(W_{2n})$ is
parametrized by $F^\times/F^{\times2}$ 
(once we fix a non-trivial additive character $\psi$ of $F$).
For $c\in F^\times$ and a semi-simple element $a$ in $C_\phi^+$, we put
$\eta_{c}(a)\coloneqq\det(N^a)(c)$.
By \cite[\S 4]{GGP}, this gives a character $\eta_{c}$ of $A_\phi^+$.
The following proposition is a special case of \cite[Theorem 3.3]{Ka1}.
\begin{prop}\label{whittaker1}
For $c_1,c_2\in F^\times$, the diagram
\[
\begin{CD}
\Pi_\phi @>\iota_{\w_{c_1}'}>> \widehat{A_\phi^+}\\
@|@VV\cdot\otimes\eta_{c_2/c_1} V\\
\Pi_\phi @>\iota_{\w_{c_2}'}>> \widehat{A_\phi^+}
\end{CD}
\]
is commutative.
\end{prop}
\par

Proposition \ref{whittaker1} determines the $L$-parameter and the character of 
the contragredient representation $\pi^\vee$ of $\pi$.
\begin{cor}\label{Sp-contragredient}
Let $\pi \in \Irr(\Sp(W_{2n}))$ and $\phi \in \cl\Phi(\Sp(W_{2n}))$ such that $\pi \in \Pi_{\phi}$.
Then $\pi^\vee \in \Pi_\phi$ and
\[
\iota_{\w'_c}(\pi^\vee)=\iota_{\w'_c}(\pi)\otimes\eta_{-1}. 
\]
\end{cor}
\begin{proof}
We decompose $W_{2n}=Y_n\oplus Y_n^*$ as in \S \ref{parabolic}.
We define $\delta\in\GL(W_{2n})$ by $\delta|Y_n=\1_{Y_n}$ and $\delta|Y_n^*=-\1_{Y_n^*}$.
Then $\delta$ be an element in $\GSp(W_{2n})$ with similitude factor $-1$, and satisfies
\[
\pair{\delta w,\delta w'}_{W_{2n}}=\pair{w',w}_{W_{2n}}
\]
for $w,w'\in W_{2n}$.
By \cite[Chapter 4.~$\mathrm{II}$.~1]{MVW}
we have $\pi^\vee\cong\pi^\delta$ for $\pi\in \Irr(\Sp(W_{2n}))$.
Hence $\pi^\vee \in \Pi_\phi$.
The map $\Sp(W_{2n}) \ni x \mapsto \delta^{-1}x\delta \in \Sp(W_{2n})$ 
transfers $\w'_c$ to $\w'_{-c}$.
See also Proposition \ref{Whittaker} below.
By Proposition \ref{whittaker1}, we have
\[
\iota_{\w'_c}(\pi^\vee)=\iota_{\w'_c}(\pi^\delta)
=\iota_{\w'_{-c}}(\pi)=\iota_{\w'_c}(\pi)\otimes\eta_{-1}, 
\]
as desired.
\end{proof}
\par

We say that $\phi\in\cl\Phi(\Sp(W_{2n}))$ is discrete
if $I^-=J=\emptyset$ and $\dim(V_i)=1$ for each $i\in I^+$.
In this case, $C_\phi^+$ is a finite group and
$\Pi_\phi$ consists of irreducible (unitary) discrete series representations 
(i.e., square-integrable representations) of $\Sp(W_{2n})$.
If $\phi\in\cl\Phi(\Sp(W_{2n}))$ is tempered but not discrete, then 
we can write
\[
\phi=\phi_1+\phi_0+\phi_1^\vee,
\]
where
\begin{itemize}
\item
$\phi_1$ is a $k$-dimensional irreducible representation of $\WD_F$ for some positive integer $k$,
\item
$\phi_0\in\cl\Phi_\temp(\Sp(W_{2n_0}))$ if $n_0=n-k\geq1$, and $\phi_0=\1$ if $n_0=n-k=0$.
\end{itemize}
In this case, there is a natural embedding $A_{\phi_0}^+\hookrightarrow A_\phi^+$,
where we put $A_{\phi_0}^+=\1$ if $n_0=0$.
We can decompose
\[
W_{2n}=Y_k+W_{2n_0}+Y_k^*
\]
and define a parabolic subgroup $Q=Q_k=M_QU_Q$ of $\Sp(W_{2n})$ as in \S \ref{parabolic}.
Hence,
\[
M_Q\cong \GL(Y_k)\times \Sp(W_{2n_0}).
\]
Let $\tau$ be the irreducible (unitary) discrete series representation of $\GL(Y_k)$
associated to $\phi_1$.
For $\pi_0\in\Pi_{\phi_0}$, the induced representation 
$\Ind_{Q}^{\Sp(W_{2n})}(\tau\otimes\pi_0)$ 
decomposes to a direct sum of irreducible representations, and
\[
\{\pi\in\Irr(\Sp(W_{2n}))\ |\ 
\pi\subset \Ind_{Q}^{\Sp(W_{2n})}(\tau\otimes\pi_0)\}
=
\{\pi\in\Pi_\phi\ |\ \iota_{\w'_c}(\pi)|A_{\phi_0}^+=\iota_{\w'_c}(\pi_0)\}
\]
for any $c\in F^\times$.
Here, if $n_0=0$, then we ignore $\pi_0$ and interpret the right hand side to be $\Pi_\phi$.
By using a normalized intertwining operator, we can determine $\iota_{\w_c'}(\pi)$ completely.
More precisely, see \S \ref{sec.intertwining} below.
\par

Let $\cl\Phi_\disc(\Sp(W_{2n}))$ (\resp $\cl\Phi_\gen(\Sp(W_{2n}))$) be the subset of 
$\cl\Phi(\Sp(W_{2n}))$ consisting of discrete representations (\resp generic representations).
Then we have a sequence
\[
\cl\Phi_\disc(\Sp(W_{2n}))\subset \cl\Phi_\temp(\Sp(W_{2n}))\subset 
\cl\Phi_\gen(\Sp(W_{2n}))\subset \cl\Phi(\Sp(W_{2n})).
\]

%\subsection{Type $D_n$}
\subsection{Type $D_n$}\label{Dn}
Let $n\geq1$ and $G=\SO(V_{2n})$ be a quasi-split special orthogonal group.
We put $E=F(\sqrt{\disc(V_{2n})})$.
Then $\widehat{G}=\SO_{2n}(\C)$ and $\Lgp{G}=\SO_{2n}(\C)\rtimes W_F$.
The action of $W_F$ on $\SO_{2n}(\C)$ factors through $W_F\rightarrow \Gal(E/F)$.
For $\gamma\in W_F$ and $\bphi\in\Phi(\SO(V_{2n}))$, 
we define ${}^\gamma\bphi\in\Phi(\SO(V_{2n}))$ 
by ${}^\gamma\bphi(x)=\gamma\cdot\bphi(x)\cdot\gamma^{-1}$.
An $L$-parameter $\bphi$ of $G$ gives a self-dual representation 
$\phi\colon\WD_F\rightarrow \GL(N)$
with sign $+1$, $\dim_\C(N)=2n$ and $\det(N)=\chi_{V_{2n}}$.
The map $\bphi\mapsto \phi$ gives a surjective map
\[
\Phi(\SO(V_{2n}))\rightarrow
\cl\Phi(\SO(V_{2n}))\coloneqq
\{\phi\colon\WD_F\rightarrow \orth(N)\ |\ \dim_\C(N)=2n, \det(N)=\chi_{V_{2n}}\}/\cong.
\]
Note that ${}^\gamma\bphi$ and $\bphi$ give the same self-dual representation.
By \cite[Theorem 8.1]{GGP}, for $\phi\in\cl\Phi(\SO(V_{2n}))$, 
the inverse image of $\phi$ under this map has one or two elements.
We say that $\phi\in\cl\Phi(\SO(V_{2n}))$ is $\epinv$ if
the inverse image of $\phi$ under this map is a singleton.
Let $\cl\Phi_\temp(\SO(V_{2n}))$ be the subset of $\cl\Phi(\SO(V_{2n}))$ of 
tempered representations.
\par

We denote the centralizer of $\im(\phi)$ in $\orth(N)$ by $C_\phi$
and its component group by $A_\phi=C_\phi/C_\phi^\circ$.
The centralizer of $\im(\phi)$ in $\SO(N)$ is denoted by $C_\phi^+$.
We put $A_\phi^+=\im(C_\phi^+\rightarrow A_\phi)$.
Then $\pi_0(S_{\bphi})\cong A_\phi^+$ if $\bphi\mapsto \phi$.
If 
\[
N\cong \bigoplus_{i\in I^+}(V_i\otimes N_i)
\oplus\bigoplus_{i\in I^-}(W_i\otimes M_i)
\oplus\bigoplus_{j\in J}(U_j\otimes (P_j+P_j^\vee))
\]
where 
\begin{itemize}
\item $M_i$ (\resp $N_i$) is an irreducible self-dual representation of $\WD_F$ 
with sign $-1$ (\resp $+1$),
\item $P_j$ is an irreducible representation of $\WD_F$ which is not self-dual,
\end{itemize}
such that $M_i$, $N_i$ and $P_j$ are pairwise inequivalent, and $V_i$, $W_i$ and $U_j$ are multiplicity spaces,
then the bilinear form on $N$ gives a non-degenerate bilinear form on $V_i$ (\resp $W_i$) of sign $+1$ (\resp $-1$).
Note that $\phi$ is $\epinv$ if and only if $\dim_\C(N_i)$ is odd for some $i\in I^+$
(see \cite[Theorem 8.1]{GGP}).
Moreover, we have
\[
C_\phi\cong\prod_{i\in I^+}\orth(V_i)\times\prod_{i\in I^-}\Sp(W_i)\times\prod_{j\in J} \GL(U_j).
\]
For a semi-simple element $a\in C_\phi$, we denote by $N^a$
the $(-1)$-eigenspace of $a$ on $N$.
Then $N^a$ gives an orthogonal representation of $\WD_F$.
If $a_i\in\orth(V_i)\setminus\SO(V_i)$, then $N^{a_i}=V_i^{a_i}\otimes N_i$.
We have
\[
A_\phi=\bigoplus_{i\in I^+}(\Z/2\Z)a_i\cong(\Z/2\Z)^m,
\]
where $m=\#I^+$.
We see that
$A_\phi^+$ is the kernel of the character
\[
A_\phi\ni a_i\mapsto (-1)^{\dim_\C(N_i)}.
\]
Hence we have 
\[
A_\phi^+\cong
\left\{
\begin{aligned}
&(\Z/2\Z)^{m-1}       \iif \text{$\phi$ is $\epinv$},\\
&(\Z/2\Z)^{m} \other.
\end{aligned}
\right.
\]
\par

It is known that 
\[
\#H^1(F,\SO(V_{2n}))=
\left\{
\begin{aligned}
&1\iif \dim(V_{2n})=2n=2\text{ and }\disc(V_{2n})=1,\\
&2\other
\end{aligned}
\right.
\] 
(which also follows from Kottwitz's isomorphism).
Hence if $n=1$ and $\disc(V_{2n})=1$, then there are no non-trivial pure inner forms of 
$G=\SO(V_{2n})$.
Otherwise, there are exactly two pure inner forms of $G=\SO(V_{2n})$.
They are $\SO(V_{2n})$ and $\SO(V_{2n}')$, where
$\dim(V_{2n})=\dim(V_{2n}')=2n$ and $\disc(V_{2n})=\disc(V_{2n}')$.
We see that
\begin{align*}
&\SO(V_{2n})\text{ is }
\left\{
\begin{aligned}
&\text{split} \iif \disc(V_{2n})=1,\\
&\text{quasi-split but not split} \iif \disc(V_{2n})\not=1,
\end{aligned}
\right.\\
&\SO(V_{2n}')\text{ is }
\left\{
\begin{aligned}
&\text{not quasi-split} \iif \disc(V_{2n}')=1,\\
&\text{quasi-split but not split} \iif \disc(V_{2n}')\not=1.
\end{aligned}
\right.\\
\end{align*}
\par

By \cite[Theorem 1.5.1 (b)]{A},
for $\phi\in\cl\Phi_\temp(\SO(V_{2n}))$, 
we obtain an $L$-packet $\Pi_\phi^A$, which is a finite subset of 
$\Irr_\temp(\SO(V_{2n}))/\sim_\ep$,
and a canonical bijection 
$\iota_{\w_c}^A\colon \Pi_\phi^A\rightarrow \pi_0(C_\phi^+/\{\pm1\})\widehat{\ }$
for each Whittaker datum $\w_c$ of $\SO(V_{2n})$.
We assume that there is a Vogan $L$-packet 
$\Pi_\phi\subset \sqcup_{V_{2n}^\bullet}\Irr_\temp(\SO(V_{2n}^\bullet))/\sim_\ep$ such that
$\Pi_\phi\cap \Irr_\temp(\SO(V_{2n}))/\sim_\ep=\Pi_\phi^A$
together with an extension 
$\iota_{\w_c}\colon \Pi_\phi\rightarrow \widehat{A_\phi^+}$ of $\iota_{\w_c}^A$.
We extend $L$-packets for general $\phi\in\cl\Phi(\SO(V_{2n}))$ as follows.
In general, we can write
\[
\phi=\phi_1\oplus\dots\oplus\phi_r\oplus\phi_0
\oplus\phi_r^\vee\oplus\dots\oplus\phi_1^\vee,
\]
where
\begin{itemize}
\item
$\phi_i$ is a $k_i$-dimensional representation of $\WD_F$ for $i=1,\dots,r$,
which is of the form
$\phi_i=\phi_i'\otimes|\cdot|_F^{s_i}$
for some tempered representation $\phi_i'$ of $\WD_F$ and some real number $s_i$ such that
\[
s_1>\dots>s_r>0, \quad n=n_0+(k_1+\cdots+k_r);
\]
\item
$\phi_0\in\cl\Phi_\temp(\SO(V_{2n_0}))$ if $n_0\geq1$, and $\phi_0=0$ if $n_0=0$.
\end{itemize}
First, we assume that $n_0>0$.
We write
\[
V_{2n}^\bullet=X_1\oplus\dots\oplus X_r\oplus V_{2n_0}^\bullet\oplus X_r^*\oplus\dots\oplus X_1^*,
\]
where
$X_i$ and $X_i^*$ are $k_i$-dimensional totally isotropic subspaces of $V_{2n}^\bullet$
which are mutually orthogonal, and
such that $X_i\oplus X_i^*$ is non-degenerate 
and orthogonal to $V_{2n_0}^\bullet$.
Let $P^\circ$ be the parabolic subgroup of $\SO(V_{2n}^\bullet)$ stabilizing the flag
\[
X_1\subset X_1\oplus X_2\subset\dots\subset X_1\oplus\dots\oplus X_r
\]
and $M_P^\circ$ be the Levi component of $P^\circ$ stabilizing the flag
\[
X_1^*\subset X_1^*\oplus X_2^*\subset\dots\subset X_1^*\oplus\dots\oplus X_r^*,
\]
so that
\[
M_P^\circ\cong \GL(X_1)\times\dots\times \GL(X_r)\times\SO(V_{2n_0}^\bullet).
\]
Then the $L$-packet $\Pi_\phi$ consists of
the $\orth(V_{2n}^\bullet)$-orbits of
the unique irreducible quotients $\sigma$ of the standard modules
\[
\Ind_{P^\circ}^{\SO(V_{2n}^\bullet)}(\tau_1\otimes\dots\otimes\tau_r\otimes\sigma_0),
\]
where $\tau_i$ is the irreducible essentially tempered representation of $\GL(X_i)$
corresponding to $\phi_i$ for $i=1,\dots,r$, and
$[\sigma_0]$ runs over elements in $\Pi_{\phi_0}$.
Since $V_{2n_0}\not=0$ and
\[
\Ind_{P^\circ}^{\SO(V_{2n}^\bullet)}(\tau_1\otimes\dots\otimes\tau_r\otimes\sigma_0)^\ep
\cong
\Ind_{P^\circ}^{\SO(V_{2n}^\bullet)}(\tau_1\otimes\dots\otimes\tau_r\otimes\sigma_0^\ep)
\]
for $\ep\in\orth(V_{2n_0})$,
the $L$-packet $\Pi_\phi$ is well-defined.
The natural map $A_{\phi_0}^+\rightarrow A_\phi^+$ is an isomorphism,
and we define $\iota_{\w_c}(\pi)\in\widehat{A_\phi^+}$ by
\[
\iota_{\w_c}(\pi_0)
=\iota_{\w_c}(\pi)|A_{\phi_0}^+.
\]
If $n_0=0$, then $\det(\phi)=\1$, so that $\phi\in\cl\Phi(\SO(V_{2n}))$ where $\SO(V_{2n})$ is split.
In this case, 
the $L$-packet $\Pi_{\phi}$ is a singleton and the unique element of $\Pi_\phi$ is
the $\orth(V_{2n})$-orbit consisting of the unique irreducible quotients $\sigma$ and $\sigma^\ep$ 
of the standard modules
\[
\Ind_{P^\circ}^{\SO(V_{2n})}(\tau_1\otimes\dots\otimes\tau_r)
\quad\text{and}\quad
\Ind_{P^\circ}^{\SO(V_{2n})}(\tau_1\otimes\dots\otimes\tau_r)^\ep.
\]
\par

Recall that the Vogan $L$-packet $\Pi_\phi$ is a subset of 
$\sqcup_{V_{2n}^\bullet}\Irr(\SO(V_{2n}^\bullet))/\sim_\ep$.
The following proposition is proven in \S \ref{Prasad}.
\begin{prop}\label{order}
For $\phi\in\cl\Phi(\SO(V_{2n}))$, we have the following:
\begin{itemize}
\item
$\phi$ is $\epinv$ if and only if $\Pi_\phi$ consists of orbits of order one;
\item
$\phi$ is not $\epinv$ if and only if $\Pi_\phi$ consists of orbits of order two.
\end{itemize}
\end{prop}
In particular, $\phi$ is $\epinv$ if and only if 
$\Pi_{\phi}$ consists of orbits of an $\ep$-invariant element.
This is the reason why we say that $\phi$ is $\epinv$.
We summarize these results in Table \ref{tab2} below.
Here, $[\sigma]$ is an element in $\Pi_\phi$ with $\sigma\in\Irr(\SO(V_{2n}^\bullet))$, and
$\cl\sigma$ is an irreducible constituent of 
$\Ind_{\SO(V_{2n}^\bullet)}^{\orth(V_{2n}^\bullet)}(\sigma)$.
\begin{table}[htb]
\begin{center}
\caption{}\label{tab2}
\begin{tabular}{|c|c|c|} \hline
   $\phi$ &$\epinv$&not $\epinv$\\ \hline\hline
     $\dim_\C(N_i)\colon$odd  &contained& not contained\\ \hline
     $A_\phi$   &   $A_\phi^+\not=A_\phi$     & $A_\phi^+=A_\phi$\\\hline
     order of $[\sigma]$  &   one     & two\\\hline
     $\sigma^\ep$  &$\sigma\cong\sigma^\ep$  &$\sigma\not\cong\sigma^\ep$\\ \hline
     $\Ind_{\SO(V_{2n}^\bullet)}^{\orth(V_{2n}^\bullet)}(\sigma)$ 
     & $\cl\sigma\oplus(\cl\sigma\otimes\det)$  & irreducible  \\ \hline
    $\cl{\sigma}\otimes\det$ & $\cl{\sigma}\not\cong\cl{\sigma}\otimes\det$ & $\cl{\sigma}\cong\cl{\sigma}\otimes\det$\\\hline
     $\cl{\sigma}|\SO(V_{2n}^\bullet)$ &irreducible& $\sigma\oplus\sigma^\ep$\\\hline
    \end{tabular}
\end{center}
\end{table}
\par

If $\disc(V_{2n})=1$, then
the set of Whittaker data of $\SO(V_{2n})$ is parametrized by $F^\times/F^{\times2}$.
If $\disc(V_{2n})\not=1$, then
both $\SO(V_{2n})$ and $\SO(V_{2n}')$ are quasi-split, and the set
\[
\bigsqcup_{V_{2n}^\bullet}\{\text{Whittaker data of $\SO(V_{2n}^\bullet)$}\}
\]
is parametrized by $F^\times/F^{\times2}$.
Note that $\cl\Phi(\SO(V_{2n}))=\cl\Phi(\SO(V_{2n}'))$.
An element in this set 
gives a (Vogan) $L$-packet of $\SO(V_{2n})$ and 
one of $\SO(V_{2n}')$.
They are both subsets of $\sqcup_{V_{2n}^\bullet}\Irr(\SO(V_{2n}^\bullet))/\sim_{\ep}$.
By Appendix \ref{A.4}, 
these $L$-packets coincide.
Hence, in each cases, 
there exists a bijection
\[
\iota_{\w_c}\colon \Pi_\phi\rightarrow \widehat{A_\phi^+}
\]
for any $\phi\in \cl\Phi(\SO(V_{2n}))$ and $c\in F^\times/F^{\times2}$.
For $c\in F^\times$ and a semi-simple element $a$ in $C_\phi^+$, we put
$\eta_{c}(a)\coloneqq\det(N^a)(c)$.
By \cite[\S 4]{GGP}, this gives a character $\eta_{c}$ of $A_\phi^+$.
Then we have the following proposition.
\begin{prop}\label{whittaker2}
For $c_1,c_2\in F^\times$, the diagram
\[
\begin{CD}
\Pi_\phi @>\iota_{\w_{c_1}}>> \widehat{A_\phi^+}\\
@|@VV\cdot\otimes\eta_{c_2/c_1} V\\
\Pi_\phi @>\iota_{\w_{c_2}}>> \widehat{A_\phi^+}
\end{CD}
\]
is commutative.
\end{prop}
The case when both $\w_{c_1}$ and $\w_{c_2}$ are Whittaker data of $\SO(V_{2n})$
 is Kaletha's result (\cite[Theorem 3.3]{Ka1}).
The proof of this proposition is given in Appendix \ref{A.5}.
\par

We say that $\phi\in\cl\Phi(\SO(V_{2n}))$ is discrete
if $I^-=J=\emptyset$ and $\dim(V_i)=1$ for each $i\in I^+$.
In this case, $C_\phi^+$ is a finite group and
$\Pi_\phi$ consists of orbits of irreducible (unitary) discrete series representations 
(i.e., square-integrable representations) of $\SO(V_{2n})$.
If $\phi\in\cl\Phi(\SO(V_{2n}))$ is tempered but not discrete, then 
we can write
\[
\phi=\phi_1+\phi_0+\phi_1^\vee,
\]
where
\begin{itemize}
\item
$\phi_1$ is a $k$-dimensional irreducible representation of $\WD_F$ for some positive integer $k$,
\item
$\phi_0\in\cl\Phi_\temp(\SO(V_{2n_0}))$ if $n_0=n-k\geq1$, and $\phi_0=0$ if $n_0=n-k=0$.
\end{itemize}
First, we assume that $n_0>0$.
In this case, there is a natural embedding $A_{\phi_0}^+\hookrightarrow A_\phi^+$.
Let $[\sigma_0]\in\Pi_{\phi_0}$ with $\sigma_0\in\Irr(\SO(V_{2n_0}^\bullet))$. 
We can decompose
\[
V_{2n}^\bullet=X_k+V_{2n_0}^\bullet+X_k^*
\]
and define a parabolic subgroup $P^\circ=P_k^\circ=M_P^\circ U_P$ of $\SO(V_{2n}^\bullet)$ 
as in \S \ref{parabolic}.
Hence,
\[
M_P^\circ\cong \GL(X_k)\times \SO(V_{2n_0}^\bullet).
\]
Let $\tau$ be the irreducible (unitary) discrete series representation of $\GL(X_k)$
associated to $\phi_1$.
Then the induced representation 
$\Ind_{P^\circ}^{\SO(V_{2n}^\bullet)}(\tau\otimes\sigma_0)$ 
decomposes to a direct sum of irreducible representations, and
\[
\{[\sigma]\ |\ \sigma\in\Irr(\SO(V_{2n}^\bullet)),
\sigma\subset \Ind_{P^\circ}^{\SO(V_{2n}^\bullet)}(\tau\otimes\sigma_0)\}
=
\{[\sigma]\in\Pi_\phi\ |\ \iota_{\w_c}([\sigma])|A_{\phi_0}^+=\iota_{\w_c}([\sigma_0])\}.
\]
Note that the left hand side is independent of
the choice of a representative $\sigma_0$ of $[\sigma_0]$.
By using a normalized intertwining operator, we can determine $\iota_{\w_c}([\sigma])$ completely.
More precisely, see \S \ref{sec.intertwining} below.
If $n_0=0$, then $\det(\phi)=\1$, so that $\phi\in\cl\Phi(\SO(V_{2n}))$ where $\SO(V_{2n})$ is split.
Moreover 
the induced representation $\Ind_{P^\circ}^{\SO(V_{2n})}(\tau)$
decomposes to a direct sum of irreducible representations, and
\[
\Pi_\phi=\{[\sigma]\ |\ \sigma\in\Irr(\SO(V_{2n})),
\sigma\subset \Ind_{P^\circ}^{\SO(V_{2n})}(\tau)\}.
\]
\par

Let $\cl\Phi_\disc(\SO(V_{2n}))$ (\resp $\cl\Phi_\gen(\SO(V_{2n}))$) be the subset of 
$\cl\Phi(\SO(V_{2n}))$ consisting of discrete representations (\resp generic representations).
Then we have a sequence
\[
\cl\Phi_\disc(\SO(V_{2n}))\subset \cl\Phi_\temp(\SO(V_{2n}))\subset 
\cl\Phi_\gen(\SO(V_{2n}))\subset \cl\Phi(\SO(V_{2n})).
\]

%\subsection{Metaplectic groups}\label{Mp}
\subsection{Metaplectic groups}\label{Mp}
Let $\Mp(W_{2n})$ be the metaplectic group, i.e.,
the non-trivial central extension of $\Sp(W_{2n})$:
\[
\begin{CD}
1@>>>\{\pm1\}@>>>\Mp(W_{2n})@>>>\Sp(W_{2n})@>>>1.
\end{CD}
\]
We denote the set of equivalence classes of irreducible genuine representations of 
$\Mp(W_{2n})$ by $\Irr(\Mp(W_{2n}))$.
\par
We recall some results of Gan--Savin (Theorem 1.1 and Corollary 1.2 in \cite{GS1}).
\begin{thm}\label{disc1}
Corresponding to the choice of a non-trivial additive character $\psi$ of $F$, 
there is a natural bijection given by the theta correspondence
\begin{align*}
\Irr(\Mp(W_{2n}))
\rightarrow
\bigsqcup_{V^{(1)}_{2n+1}}\Irr(\SO(V^{(1)}_{2n+1})),
\end{align*}
where the union is taken over all the isomorphism classes of orthogonal spaces $V_{2n+1}^{(1)}$
over $F$ with $\dim(V_{2n+1}^{(1)})=2n+1$ and $\disc(V_{2n+1}^{(1)})=1$.
\par
More precisely, for $\cl{\pi}\in\Irr(\Mp(W_{2n}))$, there is a unique $V_{2n+1}^{(1)}$ 
as above such that 
$\theta_{\psi,V_{2n+1}^{(1)},W_{2n}}(\cl{\pi})$ is nonzero, in which case the image of $\cl{\pi}$ is
$\theta(\cl{\pi})=\theta_{\psi,V_{2n+1}^{(1)},W_{2n}}(\cl{\pi})|{\SO(V_{2n+1}^{(1)})}\in\Irr(\SO(V_{2n+1}^{(1)}))$.
\end{thm}
\begin{cor}
Suppose that the local Langlands conjecture holds for $\SO(V_{2n+1})$ with 
$\dim(V_{2n+1})=2n+1$ and $\disc(V_{2n+1})=1$.
Then one has a surjection $($depending on $\psi)$ 
\begin{align*}
\LL_\psi\colon
\Irr(\Mp(W_{2n}))\rightarrow\cl\Phi(\SO(V_{2n+1})).
\end{align*}
For $\phi\in\cl\Phi(\SO(V_{2n+1}))$, 
we denote by $\Pi_\phi$ the inverse image of $\phi$ under this map,
and call $\Pi_\phi$ the $L$-packet of $\phi$.
Moreover, the composition of $\iota_\w$ and theta lifts gives a bijection $($depending on $\psi)$ of 
\[
\iota_\psi\colon\Pi_\phi\rightarrow \widehat{A_\phi}.
\]
\end{cor}
\par

We put $\Phi(\Mp(W_{2n}))\coloneqq\Phi(\SO(V_{2n+1}))$
and we call $\bphi\in \Phi(\Mp(W_{2n}))$ an $L$-parameter of $\Mp(W_{2n})$.
This is also an $L$-parameter of $\SO(V_{2n+1})$.
However we use
\[
\cl\Phi(\Mp(W_{2n})) \coloneqq \cl\Phi(\SO(V_{2n+1}))
=\{\WD_F\rightarrow \Sp(M)\ |\ \dim_\C(M)=2n\}/\cong
\] 
in this paper.
Note that there is a bijection $\Phi(\Mp(W_{2n}))\rightarrow\cl\Phi(\Mp(W_{2n}))$.
We regard $\LL_\psi$ as a map
\[
\LL_\psi\colon
\Irr(\Mp(W_{2n}))\rightarrow\cl\Phi(\Mp(W_{2n})).
\]
By \cite[Theorem 12.1]{GS1}, we have the following theorem.
\begin{thm}\label{change}
For $\cl{\pi}\in\Irr(\Mp(W_{2n}))$ and $c\in F^\times$, we put
($\LL_\psi(\cl{\pi}),\iota_\psi(\cl{\pi}))=(\phi,\eta)$ and 
$(\LL_{\psi_c}(\cl{\pi}),\iota_{\psi_c}(\cl{\pi}))=(\phi_c,\eta_c)$,
where $\phi\colon\WD_F\rightarrow \Sp(M)$ is a symplectic representation of $\WD_F$.
Then the following hold.
\begin{enumerate}
\item
$\phi_c=\phi\otimes\chi_c$, where $\chi_c$ is the quadratic character associated to $c\in F^\times/F^{\times2}$.
In particular, we have a canonical identification $A_\phi=A_{\phi_c}$.
\item
The characters $\eta$ and $\eta_c$ are related by
\[
\eta_c(a)/\eta(a)=
\varepsilon(M^a)\varepsilon(M^a\otimes\chi_c)\cdot \chi_c(-1)^{\dim_\C(M^a)/2}\in\{\pm1\}
\]
for $a\in A_\phi=A_{\phi_c}$.
\end{enumerate}
\end{thm}

By the above theorem, we have a parametrization of $\Irr(\Mp(W_{2n}))$
using $\Irr(\SO(V_{2n+1}))$ with $\disc(V_{2n+1})=c$ as follows:
\begin{cor}\label{Mp-param}
Suppose that the local Langlands conjecture holds for $\SO(V_{2n+1})$ with 
$\dim(V_{2n+1})=2n+1$ and $\disc(V_{2n+1})=1$.
We fix $c\in F^\times$.
\begin{enumerate}
\item
Corresponding to the choice of a non-trivial additive character $\psi$ of $F$, 
there is a natural bijection
\begin{align*}
\theta_\psi\colon\Irr(\Mp(W_{2n}))
\rightarrow
\bigsqcup_{V^{(c)}_{2n+1}}\Irr(\SO(V_{2n+1}^{(c)}))
\end{align*}
where the union is taken over all the isomorphism classes of orthogonal spaces $V_{2n+1}^{(c)}$
over $F$ with $\dim(V_{2n+1}^{(c)})=2n+1$ and $\disc(V_{2n+1}^{(c)})=c$.
\par
More precisely, for $\cl{\pi}\in\Irr(\Mp(W_{2n}))$, there is a unique $V_{2n+1}^{(c)}$ as above such that 
$\theta_{\psi,V_{2n+1}^{(c)},W_{2n}}(\cl{\pi})$ is nonzero, in which case the image of $\cl{\pi}$ is
$\theta_{\psi}(\cl{\pi})=\theta_{\psi,V_{2n+1}^{(c)},W_{2n}}(\cl{\pi})|{\SO(V_{2n+1}^{(c)})}\in\Irr(\SO(V_{2n+1}^{(c)}))$.
\item
The diagram
\[
\begin{CD}
\Irr(\Mp(W_{2n})) @>\theta_\psi>> \sqcup_{V^{(c)}_{2n+1}}\Irr(\SO(V_{2n+1}^{(c)}))\\
@V\LL_\psi VV @VVV\\
\cl\Phi(\Mp(W_{2n}))@>{\cdot\otimes\chi_c}>> \cl\Phi(\Mp(W_{2n}))=\cl\Phi(\SO(V_{2n+1}))
\end{CD}
\]
is commutative, where the bottom arrow is given by $\phi\mapsto \phi_c=\phi\otimes\chi_c$.
Moreover, for $\phi\in\cl\Phi(\Mp(W_{2n}))$ with $\Pi_\phi\subset\Irr(\Mp(W_{2n}))$
and $\Pi_{\phi_c}\subset \sqcup_{V^{(c)}_{2n+1}}\Irr(\SO(V_{2n+1}^{(c)}))$, 
the diagram
\[
\begin{CD}
\Pi_\phi @>\theta_\psi >> \Pi_{\phi_c}\\
@V\iota_\psi VV @VV\iota_{\w}V\\
\widehat{A_\phi} @>>> \widehat{A_{\phi}}=\widehat{A_{\phi_c}}
\end{CD}
\]
is commutative, where the bottom arrow is given by
$\eta\mapsto\eta_c$ as in the above theorem.
\end{enumerate}
\end{cor}
\begin{proof}
Let $(V_{2n+1}^{(c)},\pair{\cdot,\cdot}_c)$ be an orthogonal space with $\disc(V_{2n+1}^{(c)})=c$.
We consider a new orthogonal space $(V_{2n+1}^{(1)},\pair{\cdot,\cdot}_1)$ defined by
$V_{2n+1}^{(c)}=V_{2n+1}^{(1)}$ as vector spaces, and 
$\pair{\cdot,\cdot}_1=c^{-1}\pair{\cdot,\cdot}_c$.
Then we have $\disc(V_{2n+1}^{(1)})=c^{-1}\cdot\disc(V_{2n+1}^{(c)})=1$ in $F^\times/F^{\times2}$.
Moreover by explicit formulas of the Weil representations, we have
\[
\omega_{\psi,V_{2n+1}^{(c)},W_{2n}}=\omega_{\psi_c,V_{2n+1}^{(1)},W_{2n}}.
\]
So the first assertion follows from Theorem \ref{disc1}.
In addition, the bottom arrows in the diagrams of the second assertion
are $\LL_{\psi_c}\circ\LL_{\psi}^{-1}$ and $\iota_{\psi_c}\circ\iota_\psi^{-1}$, respectively.
Hence the second assertion follows from Theorem \ref{change}.
\end{proof}

To parametrize $\Irr(\Mp(W_{2n}))$, we use $\LL_{\psi}$ and $\iota_\psi$.
We say that $\phi\in\cl\Phi(\Mp(W_{2n}))$ is generic (\resp tempered)
if so is $\phi$ as an element in $\cl\Phi(\SO(V_{2n+1}))$.
We put $\cl\Phi_\gen(\Mp(W_{2n}))=\cl\Phi_\gen(\SO(V_{2n+1}))$ and
$\cl\Phi_\temp(\Mp(W_{2n}))=\cl\Phi_\temp(\SO(V_{2n+1}))$.
\par

As in \S \ref{section spaces}, we can define a parabolic subgroup $\cl{Q}$ of $\Mp(W_{2n})$
with the Levi subgroup isomorphic to
$\cl{\GL}_{n_1}(F)\times_{\{\pm1\}}\cdots\times_{\{\pm1\}}\cl{\GL}_{n_r}(F)
\times_{\{\pm1\}}\Mp(W_{2n_0})$
with $n_1+\dots+n_r+n_0=n$,
and a Borel subgroup $\cl{B}'=\cl{T}'U'$ of $\Mp(W_{2n})$.
Here, $\cl\GL_n(F)$ is a double cover of $\GL_n(F)$, 
$B'=T'U'$ is the Borel subgroup of $\Sp(W_{2n})$ defined in \S \ref{section spaces}
and $\cl{B'}$ is the inverse image of $B'$ in $\Mp(W_{2n})$.
More precisely, see \cite[\S 2]{GS1}.
Note that the covering splits uniquely over the unipotent radical $U'$ of $B'$.
Let $\mu_{c}'$ be the generic character of $U'$ defined in \S \ref{section spaces}.
We say that $\cl{\pi}\in\Irr(\Mp(W_{2n}))$ is $\w'_c$-generic if
$\Hom_{U'}(\cl{\pi},\mu_{c}')\not=0$.
The following proposition is a part of Theorem 1.3 (and Theorem 8.1) in \cite{GS1}.
\begin{prop}\label{O->Mp}
Let $V=V_{2n+1}$ be an orthogonal space over $F$ with $\dim(V)=2n+1$ and $\disc(V)=1$,
and $W=W_{2n}$ be a symplectic space over $F$ with $\dim(W)=2n$.
For $\sigma\in\Irr(\SO(V))$, by Theorem $\ref{disc1}$, 
we can take a unique extension $\cl{\sigma}\in\Irr(\orth(V))$ of $\sigma$ such that
$\Theta_{\psi,V,W}(\cl{\sigma})$ is nonzero, so that 
$\theta_{\psi,V,W}(\cl{\sigma})\in\Irr(\Mp(W))$. \begin{enumerate}
\item
The representation $\sigma$ is tempered if and only if 
$\Theta_{\psi,V,W}(\cl{\sigma})$ is irreducible and tempered.
In particular, in this case
we have $\Theta_{\psi,V,W}(\cl{\sigma})=\theta_{\psi,V,W}(\cl{\sigma})$.
\item
If $\cl{\sigma}$ is the unique irreducible quotient of 
\[
\Ind_{P}^{\orth(V)}
(\tau_1|\det|_F^{s_1}\otimes\cdots\otimes\tau_r|\det|_F^{s_r}\otimes\cl{\sigma}_0),
\]
where $P$ is a parabolic subgroup of $\orth(V)$ with the Levi subgroup isomorphic to
$\GL_{n_1}(F)\times\cdots\times\GL_{n_r}(F)\times\orth(V_0)$,
$\tau_i\in\Irr_\temp(\GL_{n_i}(F))$,
$\cl{\sigma}_0\in\Irr_\temp(\orth(V_0))$,
and $s_1>\cdots>s_r>0$.
Then we have
\[
\Ind_{\cl{Q}}^{\Mp(W)}
(\tau_1|\det|_F^{s_1}\otimes\cdots\otimes\tau_r|\det|_F^{s_r}\otimes
\Theta_{\psi,V_0,W_0}(\cl{\sigma}_0))
\twoheadrightarrow \Theta_{\psi,V,W}(\cl{\sigma}),
\]
where $\cl{Q}$ is the parabolic subgroup of $\Mp(W)$ with the Levi subgroup isomorphic to
$\cl{\GL}_{n_1}(F)\times_{\{\pm1\}}\cdots\times_{\{\pm1\}}\cl{\GL}_{n_r}(F)\times_{\{\pm1\}}\Mp(W_0)$ 
and $\dim(V_0)=\dim(W_0)+1$.
In particular, $\theta_{\psi,V,W}(\cl{\sigma})$ is 
the unique irreducible quotient of the standard module
$\Ind_{\cl{Q}}^{\Mp(W)}
(\tau_1|\det|_F^{s_1}\otimes\cdots\otimes\tau_r|\det|_F^{s_r}\otimes
\Theta_{\psi,V_0,W_0}(\cl{\sigma}_0))$.
\item
If $\sigma$ is $\w$-generic, then $\theta_{\psi,V,W}(\cl{\sigma})$ is $\w_1'$-generic.
\end{enumerate}
\end{prop}
\par

Let $\phi\in\cl\Phi_\gen(\Mp(W_{2n}))$ and $c\in F^\times$.
Then $\phi\otimes\chi_c\in\cl\Phi_\gen(\SO(V_{2n+1}))$ by (GPR) for $\SO(V_{2n+1})$.
For the generic representation $\sigma\in \Pi_{\phi\otimes\chi_c}$, 
by Proposition \ref{O->Mp}, 
the theta lift $\theta_{\psi_{c},V_{2n+1},W_{2n}}(\cl{\sigma})$
is $\w_c'$-generic (if it is nonzero), where 
$\cl{\sigma}$ is an extension of $\sigma$ to $\orth(V_{2n+1})$.
By Theorem \ref{change}, we have $\theta_{\psi_{c},V_{2n+1},W_{2n}}(\cl{\sigma})\in\Pi_{\phi}$.
Therefore, by Proposition \ref{O->Mp}, 
for $\phi\in\cl\Phi(\Mp(W_{2n}))$, 
we conclude that
\begin{itemize}
\item
if $\phi$ is generic, 
then $\Pi_\phi$ has a $\w_c'$-generic representation of $\Mp(W_{2n})$ for each $c\in F^\times$;
\item
if $\phi$ is tempered, then
$\Pi_\phi$ consists of tempered representations of $\Mp(W_{2n})$.
\end{itemize}
Note that even if $\phi$ is not generic,
$\Pi_\phi$ may have a generic representation of $\Mp(W_{2n})$.
\begin{example}
We set $n=1$.
The Weil representation $\omega_\psi$ of $\Mp(W_2)=\cl\SL_2(F)$ decomposes 
into a sum of two irreducible representations:
\[
\omega_\psi=\omega_\psi^e\oplus\omega_\psi^o, 
\] 
where $\omega_\psi^e$ (\resp $\omega_\psi^o$) is called even (\resp odd) Weil representations.
More precisely, see \cite{GS2}.
We consider the even Weil representation $\omega_\psi^e$.
Note that $\theta_{\psi,V_3,W_2}(\omega_\psi^e)|\SO(V_3)$
is the trivial representation $\1_{\SO(V_3)}$ of the split group $\SO(V_3)$.
Hence the element in $\cl\Phi(\Mp(W_{2}))$ associated to $\omega_\psi^e$ is
$\phi^e=|\cdot|_F^{1/2}\oplus|\cdot|_F^{-1/2}$.
Since $\Ad \circ \phi^e = |\cdot|_F^1 \oplus \1 \oplus |\cdot|_F^{-1}$, 
we see that $\phi^e$ is not generic.
However, by calculating the (twisted) Jacquet modules of $\omega_\psi^e$,
we see that $\omega_\psi^e$ is $\w_1'$-generic 
but not $\w_c'$-generic for any $c\in F^\times\setminus F^{\times2}$.
We conclude that
\begin{itemize}
\item
$\Pi_{\phi^e}\subset \Irr(\SO(V_3))$ has no generic representations of $\SO(V_3)$;
\item
$\Pi_{\phi^e}\subset \Irr(\Mp(W_2))$ has a $\w_1'$-generic representation of $\Mp(W_2)$
but does not have $\w_c'$-generic ones
for any $c\in F^\times\setminus F^{\times2}$.
\end{itemize}
\end{example}

An analogue of a conjecture of Gross--Prasad and Rallis for metaplectic groups
is given as follows.

\begin{thm}\label{GPRforMp}
Suppose that the local Langlands conjecture holds for $\SO(V_{2n+1})$ with 
$\dim(V_{2n+1})=2n+1$ and $\disc(V_{2n+1})=1$.
Let $\phi \in \Phi(\Mp(W_{2n}))$.
Then the $L$-packet $\Pi_\phi = \LL_\psi^{-1}(\phi)$ has a $\w_1'$-generic representation
if and only if the quotient of the $L$-functions
\[
\frac{L(s, \phi, \Ad)}{L(s-\half{1}, \phi)}
\]
is regular at $s=1$.
\end{thm}

\begin{rem}
By Theorem \ref{GPRforMp} together with Theorem \ref{change}, 
we see that 
the $L$-packet $\Pi_\phi = \LL_\psi^{-1}(\phi)$ has a $\w_c'$-generic representation
if and only if the quotient of the $L$-functions
\[
\frac{L(s, \phi, \Ad)}{L(s-\half{1}, \phi \otimes \chi_c)}
\]
is regular at $s=1$.
\end{rem}

The proof of Theorem \ref{GPRforMp} is given in Appendix \ref{appB}.
\par

%In this paper, we assume the following:
To prove the Gan--Gross--Prasad conjecture for symplectic-metaplectic case, 
we also need the following theorem.
\begin{thm}%[IS] 
\label{IS}
Suppose that the local Langlands conjecture holds for $\SO(V_{2n+1})$ with 
$\dim(V_{2n+1})=2n+1$ and $\disc(V_{2n+1})=1$.
Let
\[
\cl{\pi} = \Ind_{\cl{Q}}^{\Mp(W)}
(\tau_1|\det|_F^{s_1}\otimes\cdots\otimes\tau_r|\det|_F^{s_r}\otimes \cl{\pi}_0)
\]
be a standard module, where
\begin{itemize}
\item
$\cl{Q}$ is the parabolic subgroup of $\Mp(W)$ with the Levi subgroup isomorphic to
$\cl{\GL}_{n_1}(F)\times_{\{\pm1\}}\cdots\times_{\{\pm1\}}\cl{\GL}_{n_r}(F)
\times_{\{\pm1\}}\Mp(W_0)$; 
\item
$\tau_i$ is an irreducible tempered representation of $\cl{\GL}_{n_i}(F)$;
\item
$s_i$ is a real number such that $s_1 > \dots > s_r > 0$; 
\item
$\cl{\pi}_0$ is an irreducible tempered representation of $\Mp(W_0)$.
\end{itemize}
We denote the image of the unique irreducible quotient of $\cl{\pi}$ under $\LL_\phi$ by $\phi$.
If $\phi$ is generic, i.e., the adjoint $L$-function $L(s, \phi, \Ad)$ is regular at $s=1$, 
then the standard module $\cl{\pi}$ itself is irreducible.
\end{thm}

The proof of Theorem \ref{IS} is given in Appendix \ref{appB}.
Note that the analogous property for symplectic and special orthogonal groups
was proven by \cite[p. 40 Th\'eor\`eme $(\mathrm{i})$]{MW}, and
Heiermann \cite{H} showed this property 
for quasi-split connected reductive groups.
We use Theorem \ref{IS} %this assumption 
as follows:
\begin{lem}\label{IS-irr}
%Assume $(IS)$.
If $\sigma\in\Irr(\SO(V_{2n+1}))$ belongs to a generic $L$-packet and 
$\cl{\sigma}$ is an extension of $\sigma$ to $\orth(V_{2n+1})$
such that $\Theta_{\psi,V_{2n+1},W_{2n}}(\cl{\sigma})$ is nonzero, 
then the big theta lift $\Theta_{\psi,V_{2n+1},W_{2n}}(\cl{\sigma})$ is irreducible.
\end{lem}
\begin{proof}
We denote the $L$-parameter of $\sigma$ by $\phi$, i.e., $\sigma \in \Pi_\phi$.
By an assumption, $\phi$ is generic.
Choose $c \in F^\times$ such that $\disc(V_{2n+1}) = c \bmod F^{\times2}$.
Note that $\phi \otimes \chi_c$ is also generic by (GPR) for $\SO(V_{2n+1})$.
By Corollary \ref{Mp-param}, we see that $\theta_{\psi,V_{2n+1},W_{2n}}(\cl{\sigma})$ 
belongs to $\Pi_{\phi \otimes \chi_c}$, which is a generic $L$-packet of $\Mp(W_{2n})$.
By Theorem \ref{IS}, %the assumption (IS), 
the standard module in Proposition \ref{O->Mp} (2) is irreducible.
Hence so is the big theta lift $\Theta_{\psi,V_{2n+1},W_{2n}}(\cl{\sigma})$, 
which is a quotient of the standard module.
\end{proof}
\par

Let $(V,\pair{\cdot,\cdot})$ be an orthogonal space with $\dim(V)=2n+1$ and $\disc(V)=c$.
We define a new orthogonal space $-V$ by $(V,-\pair{\cdot,\cdot})$.
Then by \cite[Lemma 3.2]{GS1}, for $\cl{\sigma}\in\Irr(\orth(V))$ with
$\theta_{\psi,V,W}(\cl{\sigma})\not=0$, we have
\[
\theta_{\psi,V,W}(\cl{\sigma})^\vee\cong\theta_{\psi,-V,W}(\cl{\sigma}^\vee)
\cong\theta_{\psi_{-1},V,W}(\cl{\sigma}).
\]
We put $\cl{\pi}=\theta_{\psi,V,W}(\cl{\sigma})$. 
Let $\phi_\sigma\in\cl\Phi(\SO(V))$ be a representation of $\WD_F$ associated to 
$\sigma = \cl{\sigma}|\SO(V)$.
We put $\eta_\sigma=\iota_\w(\sigma)$, 
$\phi=\LL_\psi(\cl{\pi}^\vee)$ and $\eta=\iota_\psi(\cl{\pi}^\vee)$.
Then by Corollary \ref{Mp-param}, we have
\[
\phi_{\sigma}=\phi\otimes\chi_{-c}
\quad\text{and}\quad
\eta_{\sigma}=\eta_{-c},
\]
where $\eta_{-c}$ is the character defined from $\phi$ and $\eta$ as in Theorem \ref{change}.

%\section{Some conjectures}\label{conjectures}
\section{Some conjectures}\label{conjectures}
In this section, we explicate the statements of the local (Gan--)Gross--Prasad conjecture for 
the orthogonal case and the symplectic-metaplectic case as well as
 Prasad's conjectures on the local theta correspondence for $(\Sp(W_{2n}),\orth(V_{2m}))$ with
\[
|2m-(2n+1)|=1.
\]

%\subsection{The local Gross--Prasad conjecture}\label{GGPforSO}
\subsection{The local Gross--Prasad conjecture}\label{GGPforSO}
In this subsection, we state the local Gross--Prasad conjecture 
(local Gan--Gross--Prasad conjecture for the orthogonal case) proven by Waldspurger 
\cite{Wa2, Wa3, Wa4, Wa5} and M\oe glin--Waldspurger \cite{MW}.
\par

Let $V$ be an orthogonal space over $F$ and 
$V'$ be a non-degenerate subspace of $V$ with codimension one.
Assume that both $\SO(V)$ and $\SO(V')$ are quasi-split.
Let $V^{\even}$ and $V^{\odd}$ be in $\{V,V'\}$ such that
$\dim(V^{\even})\in2\Z$ and $\dim(V^{\odd})\not\in2\Z$.
\par

Let $\phi_M\colon\WD_F\rightarrow \Sp(M)$ and $\phi_N\colon\WD_F\rightarrow \orth(N)$
be elements in $\cl\Phi(\SO(V^{\odd}))$ and $\cl\Phi(\SO(V^{\even}))$,
respectively. 
Following \cite[\S 6]{GGP}, 
for semi-simple elements $a\in C_M$ and $b\in C_N^+$, we put
\begin{align*}
\chi_N(a)&=\varepsilon(M^a\otimes N)\det(M^a)(-1)^{\dim_\C(N)/2}\det(N)(-1)^{\dim_\C(M^a)/2},\\
\chi_M(b)&=\varepsilon(M\otimes N^b)\det(M)(-1)^{\dim_\C(N^b)/2}\det(N^b)(-1)^{\dim_\C(M)/2}.
\end{align*}
By \cite[Theorem 6.2]{GGP}, $\chi_N$ and $\chi_M$ define 
characters on $A_M$ and on $A_N^+$, respectively.
\par

We say that a pure inner form $G_1=\SO(V_1)\times\SO(V_1')$ of $G$ is relevant 
if $V_1'$ is a non-degenerate subspace of $V_1$ and $V_1/V_1'\cong V/V'$ as orthogonal spaces.
Then there is a natural embedding $\SO(V_1')\hookrightarrow \SO(V_1)$. 
Hence we have a diagonal embedding
\[
\Delta\colon \SO(V_1')\hookrightarrow \SO(V_1')\times\SO(V_1).
\]
\par

The Gan--Gross--Prasad conjecture for the orthogonal case is as follows:
\begin{conj}[B]
Let $\phi_M\colon\WD_F\rightarrow \Sp(M)$ and $\phi_N\colon\WD_F\rightarrow \orth(N)$
be in $\cl\Phi_\gen(\SO(V^{\odd}))$ and in $\cl\Phi_\gen(\SO(V^{\even}))$, respectively. 
We take $\sigma_M\in\Pi_{\phi_M}$ and $[\sigma_N]\in\Pi_{\phi_N}$ such that
$\sigma=\sigma_M\boxtimes\sigma_N$ is a representation of 
a relevant pure inner form $G_1=\SO(V_1^\odd)\times\SO(V_1^\even)$ of $G$.
Then one has 
\[
\Hom_{\Delta\SO(V'_1)}(\sigma,\C)\not=0
\Longleftrightarrow
\iota_\w(\sigma_M)\times\iota_{\w_c}([\sigma_N])=\chi_N\times\chi_M,
\]
where $c=-\disc(V^{\odd})/\disc(V^{\even})$.
\end{conj}
Let $\ep\in\orth(V_1')\setminus\SO(V_1')$.
We can extend $\sigma_M$ to $\cl\sigma_M\in\Irr(\orth(V_1^\odd))$.
If $f\in \Hom_{\Delta\SO(V_1')}(\sigma_M\boxtimes\sigma_N,\C)$, then we have
$f\circ(\cl\sigma_M(\ep)\otimes\id)\in\Hom_{\Delta\SO(V_1')}(\sigma_M\boxtimes\sigma_N^\ep,\C)$.
Hence, (B) is independent of the choice of a representative of $[\sigma_N]$.

%\subsection{The local Gan--Gross--Prasad conjecture for the symplectic-metaplectic case}
\subsection{The local Gan--Gross--Prasad conjecture for the symplectic-metaplectic case}
We fix a non-trivial additive character $\psi$ of $F$.
Let $W$ be a symplectic space over $F$, and
$V_1$ be an orthogonal space over $F$ with $\dim(V_1)=1$ and $\disc(V_1)=1$.
Hence there exists $v\in V_1$ such that $\pair{v,v}_{V_1}=2$.
Then the space $V_1\otimes W$ has a symplectic form
\[
\pair{\cdot, \cdot}_{V_1}\otimes\pair{\cdot, \cdot}_{W}.
\]
Let $\omega_{\psi}$ be the Weil representation of $\Mp(V_1\otimes W)=\Mp(W)$ 
given by the unique irreducible representation of 
the Heisenberg group $\H(V_1\otimes W)$ associated to the symplectic space 
$(V_1\otimes W,\pair{\cdot, \cdot}_{V_1}\otimes\pair{\cdot, \cdot}_{W})=(W,2\pair{\cdot,\cdot}_{W})$
with the central character $\psi$.
\par

Let $\phi_M\colon\WD_F\rightarrow \Sp(M)$ and $\phi_N\colon\WD_F\rightarrow \SO(N)$
be elements in $\cl\Phi(\Mp(W))$ and in $\cl\Phi(\Sp(W))$,
respectively. 
We put $N_1=N\oplus\C$, which is an orthogonal representation of $\WD_F$ with even dimension.
We may regard $A_{N}^+$ as a subgroup of $A_{N_1}^+$.
Note that $(A_{N_1}^+\colon A_N^+)\leq 2$.
Following \cite[\S 6]{GGP}, for semi-simple elements $a\in C_M$ and $b\in C_{N_1}^+$, we put
\begin{align*}
\chi_{N_1}(a)&=
\varepsilon(M^a\otimes N_1)\det(M^a)(-1)^{\dim_\C(N_1)/2}\det(N_1)(-1)^{\dim_\C(M^a)/2},\\
\chi_M(b)&=
\varepsilon(M\otimes N_1^b)\det(M)(-1)^{\dim_\C(N_1^b)/2}\det(N_1^b)(-1)^{\dim_\C(M)/2}.
\end{align*}
By \cite[Theorem 6.2]{GGP}, $\chi_{N_1}$ and $\chi_M$ 
define characters on $A_M$ and on $A_{N_1}^+$, respectively.
\par

The identity map and the projection map give a diagonal map
\[
\Delta\colon \Mp(W)\rightarrow \Mp(W)\times\Sp(W).
\]
The local Gan--Gross--Prasad conjecture for the symplectic-metaplectic case is as follows:
\begin{conj}[FJ]
Let $\phi_M\colon\WD_F\rightarrow \Sp(M)$ and $\phi_N\colon\WD_F\rightarrow \SO(N)$
be in $\cl\Phi_\gen(\Mp(W))$ and in $\cl\Phi_\gen(\Sp(W))$, respectively. 
We denote by $\Pi_{\phi_M}\subset\Irr(\Mp(W))$ the inverse image of $\phi_M$ under the map $\LL_\psi$.
Then for $\cl\pi\in\Pi_{\phi_M}$ and $\pi\in\Pi_{\phi_N}$, 
one predicts
\[
\Hom_{\Delta\Mp(W)}((\cl\pi\boxtimes\pi)\otimes\overline{\omega_\psi},\C)\not=0
\Longleftrightarrow
\iota_\psi(\cl\pi)\times\iota_{\w'_1}(\pi)=\chi_{N_1}\times\chi_M|_{A_\phi^+}.
\]
\end{conj}
There is a conjecture for general codimension cases (Conjectures 17.1 and 17.3 in \cite{GGP}).
However, by \cite[Theorem 19.1]{GGP}, the general codimension cases 
are reduced to the basic case (FJ).

%\subsection{The Prasad conjectures}\label{Prasad}
\subsection{Prasad's conjectures}\label{Prasad}
We consider the theta correspondences for 
$(\Sp(W_{2n}),\orth(V_{2n+2}))$ and $(\orth(V_{2n}),\Sp(W_{2n}))$. 
Let $V^+=V_{2m}^+$ be an orthogonal space with $\dim(V_{2m}^+)=2m$ and type $(d,c)$.
We denote by $V^-=V_{2m}^-$ an orthogonal space 
such that $\dim(V_{2m}^-)=\dim(V_{2m}^+)=2m$ and $\disc(V_{2m}^-)=\disc(V_{2m}^+)$
but $V_{2m}^-\not\cong V_{2m}^+$.
Note that $\SO(V_{2m}^+)$ is quasi-split, and
$\SO(V_{2m}^-)$ is the non-trivial pure inner form of $\SO(V_{2m}^+)$
(if $V_{2m}^-$ exists).
We put $\chi_V=\chi_{V_{2m}}$.
\par

First, we let $m=n+1$.
The following proposition is a special case of \cite[Theorem C.5]{GI1}.
\begin{prop}\label{param}
Let $\phi\in\cl\Phi(\Sp(W_{2n}))$ and put
\[
\phi'=(\phi\otimes\chi_{V})\oplus \1\in\cl\Phi(\SO(V_{2n+2}^+)).
\]
\begin{enumerate}
\item
If $\phi$ does not contain $\chi_V$, then
\begin{itemize}
\item
$\theta_{\psi,V_{2n+2}^\bullet,W_{2n}}(\pi)$ is nonzero for any $\pi\in\Pi_\phi$ and $\bullet = \pm$;
\item
$\sigma\coloneqq\theta_{\psi,V_{2n+2}^\bullet,W_{2n}}(\pi)|\SO(V_{2n+2}^\bullet)$ is irreducible 
and $[\sigma]\in\Pi_{\phi'}$;
\item
the theta correspondence gives a bijection
\[
\theta_{\psi,V_{2n+2}^\bullet,W_{2n}}\colon\Pi_\phi\rightarrow
\Pi_{\phi'}\cap\Irr(\SO(V_{2n+2}^\bullet)).
\]
\end{itemize}
\item
If $\phi$ contains $\chi_V$, then
\begin{itemize}
\item
exactly one of $\theta_{\psi,V_{2n+2}^+,W_{2n}}(\pi)$ and $\theta_{\psi,V_{2n+2}^-,W_{2n}}(\pi)$ 
is nonzero for any $\pi\in\Pi_\phi$;
\item
if $\theta_{\psi,V_{2n+2}^\bullet,W_{2n}}(\pi)\not=0$, then 
$\sigma\coloneqq\theta_{\psi,V_{2n+2}^\bullet,W_{2n}}(\pi)|\SO(V_{2n+2}^\bullet)$ 
is irreducible and
$[\sigma]\in\Pi_{\phi'}$;
\item
the theta correspondence gives a bijection
\[
\theta_{\psi,V_{2n+2}^\bullet,W_{2n}}\colon\Pi_\phi\rightarrow\Pi_{\phi'}.
\]
\end{itemize}
\end{enumerate}
\end{prop}
There is a canonical injective map
\[
A_{\phi}^+\rightarrow A_{\phi'}^+.
\]
It is bijective if and only if $\phi$ contains $\chi_V$.
This map induces a surjection of the character groups
\[
\widehat{A_{\phi}^+}\leftarrow \widehat{A_{\phi'}^+}.
\]
\par

The first conjecture of Prasad (P1) predicts the behavior of characters associated to 
$\pi$ and $\theta_{\psi,V_{2n+2}^\bullet,W_{2n}}(\pi)$.
\begin{conj}[P1]\label{P1}
Let $\phi\in\cl\Phi(\Sp(W_{2n}))$
and put $\phi'=(\phi\otimes\chi_V)\oplus\1\in\cl\Phi(\SO(V_{2n+2}^+))$.
For $\pi\in\Pi_\phi$, we take 
$\sigma=\theta_{\psi,V_{2n+2}^\bullet,W_{2n}}(\pi)|\SO(V_{2n+2}^\bullet)$
such that $\sigma\not=0$, so that $[\sigma]\in\Pi_{\phi'}$.
Then one predicts
\[
\iota_{\w_{c_0}}([\sigma])|A_{\phi}^+=\iota_{\w'_{c_0}}(\pi)
\]
for $c_0\in F^\times$.
Namely, the diagram
\[
\begin{CD}
\Pi_\phi @>\theta_{\psi,V_{2n+2}^\bullet,W_{2n}}>>\Pi_{\phi'}\\
@V\iota_{\w'_{c_0}}VV                              @VV\iota_{\w_{c_0}}V\\
\widehat{A_{\phi}^+}@<<< \widehat{A_{\phi'}^+}
\end{CD}
\]
is commutative.
\end{conj}
Note that the maps $\theta_{\psi,V_{2n+2}^\bullet,W_{2n}}$ and $\iota_{\w'_{c_0}}$
depend on the choice of $\psi$, but the maps $\iota_{\w_{c_0}}$ and 
$\widehat{A_{\phi}^+}\leftarrow \widehat{A_{\phi'}^+}$ do not.
\par

Next, we let $m=n$.
The following proposition is a special case of \cite[Theorem C.5]{GI1}.
\begin{prop}\label{param2}
Let $\phi'\in\cl\Phi(\SO(V_{2n}^+))$
and put
\[
\phi=(\phi'\otimes\chi_V)\oplus\chi_V\in\cl\Phi(\Sp(W_{2n})).
\]
\begin{enumerate}
\item
If $\phi'$ does not contain $\1$, then
\begin{itemize}
\item
$\theta_{\psi,V_{2n}^\bullet,W_{2n}}(\cl{\sigma})$ is nonzero for any $[\sigma]\in\Pi_{\phi'}$ 
and any irreducible constituent 
$\cl{\sigma}$ of $\Ind_{\SO(V_{2n}^\bullet)}^{\orth(V_{2n}^\bullet)}(\sigma)$;
\item
$\theta_{\psi,V_{2n}^\bullet,W_{2n}}(\cl{\sigma})$ is irreducible and belongs to $\Pi_{\phi}$;
\item
the theta correspondence gives a bijection
\[
\theta_{\psi,V_{2n}^\bullet,W_{2n}}\colon \bigsqcup_{[\sigma]\in\Pi_{\phi'}}
\left\{\cl{\sigma}\subset\Ind_{\SO(V_{2n}^\bullet)}^{\orth(V_{2n}^\bullet)}(\sigma)\right\} 
\longleftrightarrow \Pi_\phi.
\]
\end{itemize}
\item
If $\phi'$ contains $\1$, then 
\begin{itemize}
\item
for any $[\sigma]\in\Pi_{\phi'}$, there exists a unique extension $\cl{\sigma}$ of $\sigma$ 
to $\orth(V_{2n}^\bullet)$ 
such that
$\theta_{\psi,V_{2n}^\bullet,W_{2n}}(\cl{\sigma})$ is nonzero;
\item
if $\theta_{\psi,V_{2n}^\bullet,W_{2n}}(\cl{\sigma})\not=0$, 
then $\theta_{\psi,V_{2n}^\bullet,W_{2n}}(\cl{\sigma})$ is irreducible and
belongs to $\Pi_{\phi}$;
\item
the theta correspondence gives a bijection
\[
\theta_{\psi,V_{2n}^\bullet,W_{2n}}\colon \Pi_{\phi'}\longleftrightarrow \Pi_\phi.
\]
\end{itemize}
\end{enumerate}
\end{prop}
Note that for $[\sigma]\in\Pi_{\phi'}$, 
the induction $\Ind_{\SO(V_{2n}^\bullet)}^{\orth(V_{2n}^\bullet)}(\sigma)$
does not depend on the choice of a representative $\sigma$
by Lemma \ref{well-defined} (\ref{well1}).
Also, if $\phi'$ contains $\1$, then $\phi'$ is $\epinv$ so that
the order of $[\sigma]$ is one and 
there are exactly two extensions of $\sigma$ to $\orth(V_{2n}^\bullet)$ 
(see Table \ref{tab2} in \S \ref{Dn}).
\par

There is a canonical injective map
\[
A_{\phi'}^+\rightarrow A_\phi^+.
\]
It is not bijective if and only if
$\phi'$ is $\epinv$ and does not contain $\1$.
This map induces a surjection of the character groups
\[
\widehat{A_{\phi'}^+}\leftarrow \widehat{A_{\phi}^+}.
\]
\par

We put
\[
\cl{\Pi}_{\phi'}=
\bigsqcup_{[\sigma]\in\Pi_{\phi'}}\left\{\cl{\sigma}\subset\Ind_{\SO(V_{2n}^\bullet)}^{\orth(V_{2n}^\bullet)}(\sigma)\ 
| \ 
\theta_{\psi,V_{2n}^\bullet,W_{2n}}(\cl{\sigma})\not=0\right\}
\subset\bigsqcup_{V_{2n}^\bullet}\Irr(\orth(V_{2n}^\bullet)).
\]
If $\cl{\sigma}$ is an irreducible constituent of $\Ind_{\SO(V_{2n}^\bullet)}^{\orth(V_{2n}^\bullet)}(\sigma)$,
then the image of $\cl\sigma$ under the map
\[
\Irr(\orth(V))\rightarrow \Irr(\orth(V))/\sim_{\det}\xrightarrow{1:1} \Irr(\SO(V))/\sim_\ep
\]
is $[\sigma]$.
Here, the last map is given by the restriction. 
See \S \ref{SOvsO}.
We define a map $\iota_{\w_{c_0}}\colon\cl\Pi_{\phi'}\rightarrow \widehat{A_{\phi'}^+}$ by
$\iota_{\w_{c_0}}(\cl{\sigma})\coloneqq\iota_{\w_{c_0}}([\sigma])$ for $c_0\in F^\times$.
\par

The second conjecture of Prasad (P2) predicts the behavior of characters associated to 
$\cl\sigma$ and $\theta_{\psi,V_{2n}^\bullet,W_{2n}}(\cl{\sigma})$.
\begin{conj}[P2]\label{P2}
Let $\phi'\in\cl\Phi(\SO(V_{2n}^+))$
and put $\phi=(\phi\otimes\chi_V)\oplus\chi_V\in\cl\Phi(\Sp(W_{2n}))$.
For $[\sigma]\in\Pi_{\phi'}$, 
we take an irreducible constituent 
$\cl{\sigma}\subset \Ind_{\SO(V_{2n}^\bullet)}^{\orth(V_{2n}^\bullet)}(\sigma)$
such that $\pi=\theta_{\psi,V_{2n}^\bullet,W_{2n}}(\cl{\sigma})$
is nonzero, so that $\pi\in\Pi_{\phi}$.
Then one predicts
\[
\iota_{\w'_{c_0}}(\pi)|A_{\phi'}^+=\iota_{\w_{c_0}}(\cl\sigma)=\iota_{\w_{c_0}}([\sigma])
\]
for $c_0\in F^\times$.
Namely, the diagram
\[
\begin{CD}
\cl{\Pi}_{\phi'} @>\theta_{\psi,V_{2n}^\bullet,W_{2n}}>>\Pi_{\phi}\\
@V\iota_{\w_{c_0}}VV                              @VV\iota_{\w'_{c_0}}V\\
\widehat{A_{\phi'}^+}@<<< \widehat{A_{\phi}^+}
\end{CD}
\]
is commutative.
\end{conj}
Note that the maps $\theta_{\psi,V_{2n}^\bullet,W_{2n}}$ and $\iota_{\w'_{c_0}}$
depend on the choice of $\psi$, but the maps $\iota_{\w_{c_0}}$ and 
$\widehat{A_{\phi}^+}\leftarrow \widehat{A_{\phi'}^+}$ do not.
\begin{rem}
Conjectures $(P1)$ and $(P2)$ may be slightly different from ones in \cite{P1}.
The difference comes from the choice of the splitting
$\Sp(W_{2n}) \times \orth(V_{2m}) \rightarrow \Mp(W_{2n} \otimes V_{2m})$.
We use the splitting in \cite{Ku2}.
See also explicit formulas for the Weil representations below
$($Lemmas $\ref{Weil1}$, $\ref{Weil2}$ and $\ref{Weil3})$.
\end{rem}
\par

In this paper, we show (FJ), (P1) and (P2) under some assumptions.

\par
Propositions \ref{param} and \ref{param2} imply Proposition \ref{order}.
\begin{proof}[Proof of Proposition $\ref{order}$]
Let $\phi' \in \cl\Phi(\SO(V_{2n}^+))$.
If $\phi'$ contains $\1$, by Proposition \ref{param}, any element in $\Pi_{\phi'}$ is 
the restriction of an irreducible representation of $\orth(V_{2n}^\bullet)$.
Hence, $\phi'$ is $\epinv$ and $\Pi_{\phi'}$ consists of orbits of an $\ep$-invariant element.
\par

Suppose that $\phi'$ does not contain $\1$.
In this case, 
\[
\cl\Pi_{\phi'} = 
\left\{\left.
\cl\sigma \in \bigsqcup_{V_{2n}^\bullet}\Irr(\orth(V_{2n}^\bullet))\ \right|\ 
\cl\sigma \subset \Ind_{\SO(V_{2n}^\bullet)}^{\orth(V_{2n}^\bullet)}(\sigma)
\text{ for some $[\sigma] \in \Pi_{\phi'}$}
\right\}.
\]
It suffices to show that 
$\#\cl\Pi_{\phi'} = r \cdot \#\Pi_{\phi'}$, where
\[
r= \left\{
\begin{aligned}
&2	\iif \text{$\phi'$ is $\epinv$},\\
&1	\other.
\end{aligned}
\right.
\]
Put $\phi = (\phi' \otimes \chi_V) \oplus \chi_V \in \cl\Phi(\Sp(W_{2n}))$ 
as in Proposition \ref{param2}.
Then Proposition \ref{param2} says that the theta lift gives a bijection
$\cl\Pi_{\phi'} \rightarrow \Pi_{\phi}$.
Hence $\#\cl\Pi_{\phi'} = \#\Pi_{\phi}$.
On the other hand, we have $\# A_{\phi}^+ = r \cdot \# A_{\phi'}^+$ (see \S \ref{Cn} and \S \ref{Dn}).
Therefore we have
\[
\#\cl\Pi_{\phi'} = \#\Pi_{\phi} = \# A_{\phi}^+ = r \cdot \# A_{\phi'}^+ = r \cdot \#\Pi_{\phi'}, 
\]
as desired. This completes the proof.
\end{proof}

%\section{Reductions to (P1) for tempered $L$-parameters}
\section{Reductions to (P1) for tempered $L$-parameters}
In this section, we reduce (P1), (P2) and Theorem \ref{main}
to (P1) for tempered representations.
Through this section, we assume (P1) for tempered representations.
\par

%\subsection{Some lemmas}
\subsection{Some lemmas}
Let $\Sp(W_{2n})$ and $\SO(V_{2n+2})$ be quasi-split 
symplectic and special orthogonal groups, respectively.
We put $\chi_V=\chi_{V_{2n+2}}$.
To prove Theorem \ref{main}, we need the following two lemmas.
\begin{lem}\label{L(1,Ad)}
Let $\phi\in\cl\Phi_\gen(\Sp(W_{2n}))$ 
and put 
$\phi'=(\phi\otimes\chi_V)\oplus\1\in\cl\Phi(\SO(V_{2n+2}))$.
Then $\phi'$ is generic if and only if the local twisted $L$-function
$L(s,\phi\otimes\chi_V)$ is regular at $s=1$.
\end{lem}
\begin{proof}
Since $\Ad \circ \phi' = (\Ad \circ \phi) \oplus (\phi \otimes \chi_V)$, we have
\[
L(s,\phi',\Ad) = L(s,\phi,\Ad) \cdot L(s,\phi\otimes\chi_V).
\]
The assertion follows from this equation and (GPR) for $\Sp(W_{2n})$ and $\SO(V_{2n+2})$.
\end{proof}

\begin{lem}\label{big theta}
Let $\phi\in\cl\Phi(\Sp(W_{2n}))$ and $\pi\in\Pi_\phi$.
Assume that $\phi'=(\phi\otimes\chi_V)\oplus\1\in\cl\Phi(\SO(V_{2n+2}))$ is generic.
Then the big theta lift $\Theta_{\psi,V_{2n+2}^\bullet,W_{2n}}(\pi)|\SO(V_{2n+2}^\bullet)$
is irreducible $($if it is nonzero$)$.
Hence, 
\[
\Theta_{\psi,V_{2n+2}^\bullet,W_{2n}}(\pi)|\SO(V_{2n+2}^\bullet)
=\theta_{\psi,V_{2n+2}^\bullet,W_{2n}}(\pi)|\SO(V_{2n+2}^\bullet).
\]
\end{lem}
\begin{proof}
If $\phi$ is tempered, the assertion is in Proposition \ref{Theta}.
In general, we may write
\[
\phi=\phi_1\oplus\dots\oplus\phi_r\oplus\phi_0\oplus\phi_r^\vee\oplus\dots\oplus\phi_1^\vee,
\]
as in \S \ref{Cn}.
Note that the canonical map $A_{\phi_0}^+\rightarrow A_\phi^+$ is an isomorphism.
Let $\pi\in\Pi_\phi$.
Then $\pi$ is the unique irreducible quotient of the standard module
\[
\Ind_Q^{\Sp(W_{2n})}(\tau_1\chi_V|\det|_F^{s_1}\otimes\dots\otimes\tau_r\chi_V|\det|_F^{s_r}
\otimes\pi_0),
\]
where $\tau_i \chi_V$ is the tempered representation of $\GL_{k_i}(F)$ 
associated to $\phi'_i$,
and $\pi_0\in\Pi_{\phi_0}$ with $\iota_{\w'_{c_0}}(\pi)|A_{\phi_0}^+=\iota_{\w'_{c_0}}(\pi_0)$.
Here, if $n_0=0$, then we ignore $\pi_0$.
By \cite[Proposition C.4 $(\mathrm{ii})$]{GI1}, 
we have a surjection
\[
\Ind_P^{\orth(V_{2n+2}^\bullet)}(\tau_1|\det|_F^{s_1}\otimes\dots\otimes\tau_r|\det|_F^{s_r}
\otimes\Theta_{\psi,V_{2n_0+2}^\bullet,W_{2n_0}}(\pi_0))
\twoheadrightarrow \Theta_{\psi,V_{2n+2}^\bullet,W_{2n}}(\pi).
\]
Here, if $n_0=0$, then we interpret $\Theta_{\psi,V_{2n_0+2},W_{2n_0}}(\pi_0)$
to be the trivial representation of $\orth(V_2)$.
In particular, if $\Theta_{\psi,V_{2n+2}^\bullet,W_{2n}}(\pi)$ is nonzero, 
then so is $\Theta_{\psi,V_{2n_0+2}^\bullet,W_{2n_0}}(\pi_0)$.
By a similar argument to Lemma \ref{well-defined} (\ref{well3}), we have
\begin{align*}
&\Ind_P^{\orth(V_{2n+2}^\bullet)}(\tau_1|\det|_F^{s_1}\otimes\dots\otimes\tau_r|\det|_F^{s_r}
\otimes\Theta_{\psi,V_{2n_0+2}^\bullet,W_{2n_0}}(\pi_0))|\SO(V_{2n+2}^\bullet)\\
&\cong
\Ind_{P^\circ}^{\SO(V_{2n+2}^\bullet)}(\tau_1|\det|_F^{s_1}\otimes\dots\otimes\tau_r|\det|_F^{s_r}
\otimes\Theta_{\psi,V_{2n_0+2}^\bullet,W_{2n_0}}(\pi_0)|\SO(V_{2n_0+2}^\bullet)).
\end{align*}
Since $\phi'$ is generic, 
by \cite[p. 40 Th\'eor\`eme $(\mathrm{i})$]{MW},
this standard module is irreducible.
Therefore the quotient 
$\Theta_{\psi,V_{2n+2}^\bullet,W_{2n}}(\pi)|\SO(V_{2n+2}^\bullet)$ is also irreducible.
\end{proof}

%\subsection{Existence of (FJ)}\label{existence}
\subsection{Existence of (FJ)}\label{existence}
In this subsection, we prove the following proposition.
\begin{prop}\label{exist}
Let 
\begin{align*}
\phi_M\colon\WD_F\rightarrow\Sp(M)
\quad\text{and}\quad
\phi_N\colon\WD_F\rightarrow\SO(N)
\end{align*}
be in $\cl\Phi_\gen(\Mp(W_{2n}))$ and in $\cl\Phi_\gen(\Sp(W_{2n}))$, respectively.
Assume $(P1)$ %, $(IS)$ (Assumption \ref{IS}) 
and that
\begin{itemize}
\item
$L(s,\phi_N\otimes\chi_d)$ is regular at $s=1$ for some $d\in F^\times$.
\end{itemize}
Then there exist $\cl{\pi}\in\Pi_{\phi_M}=\LL_\psi^{-1}(\phi_M)$ and $\pi\in\Pi_{\phi_N}$
such that
\[
\Hom_{\Delta\Mp(W_{2n})}((\cl{\pi}\boxtimes\pi)\otimes\overline{\omega_\psi},\C)\not=0.
\]
\end{prop}
\begin{proof}
We consider the following see-saw diagram:
\[
\xymatrix{
   \Mp(W_{2n})\times_{\{\pm1\}}\Mp(W_{2n})   \ar@{-}[d]    &  &  \orth(V_{2n+2}) \ar@{-}[d]     \\
     \Sp(W_{2n}) \ar@{-}[urr]      &  &  \orth(V_{2n+1})\times\orth(V_1)\ar@{-}[ull]   \\
}
\]
with $\disc(V_1)=-1$.
\par

Note that $\dim_\C(M)=2n$ and $\dim_\C(N)=2n+1$.
We put
\[
(\phi_\sigma,N_\sigma)=((\phi_N\otimes\chi_{d})\oplus\1,(N\otimes\chi_d)\oplus\C)
\quad\text{and}\quad
(\phi_\tau,M_\tau)=(\phi_M\otimes\chi_{d},M\otimes\chi_d).
\]
Then $\phi_\sigma$ is an $\epinv$ element in $\cl\Phi(\SO(V_{2n+2}))$ with $\disc(V_{2n+2})=d$,
and $\phi_\tau\in\cl\Phi(\SO(V_{2n+1}))$ with $\disc(V_{2n+1})=d$.
Moreover, by Lemma \ref{L(1,Ad)} and 
(GPR) for $\SO(V_{2n+1})$, we see that
$\phi_\sigma$ and $\phi_\tau$ are generic.
By (B), there are a pair of orthogonal spaces $V_{2n+1}\subset V_{2n+2}$,
and $[\sigma]\in\Pi_{\phi_\sigma}\cap\Irr(\SO(V_{2n+2}))/\sim_\ep$, 
$\tau\in\Pi_{\phi_\tau}\cap\Irr(\SO(V_{2n+1}))$ such that
\[
\Hom_{\Delta\SO(V_{2n+1})}(\sigma\otimes\tau,\C)\not=0.
\]
If we denote the orthogonal complement of $V_{2n+1}$ in $V_{2n+2}$ by $V_1$, then we have
\[
\disc(V_1)=-\disc(V_{2n+2})/\disc(V_{2n+1})=-1.
\]
\par

By Proposition \ref{param} and (the proof of) Lemma \ref{big theta}, 
we can find $\pi\in\Pi_{\phi_N}$ such that
\[
\Theta_{\psi,V_{2n+2},W_{2n}}(\pi)|\SO(V_{2n+2})=\sigma.
\]
We claim that there exists an extension $\tau_\bullet$ of $\tau$ to $\orth(V_{2n+1})$ such that
\[
\Hom_{\Delta\orth(V_{2n+1})}(\Theta_{\psi,V_{2n+2},W_{2n}}(\pi)\otimes\tau_\bullet,\C)\not=0.
\]
Let $\varepsilon=-1$ be the non-trivial element in the center of $\orth(V_{2n+1})$, 
so that $\varepsilon\not\in\SO(V_{2n+1})$.
Then there exist two extensions $\tau_{+}$ and $\tau_-$ of $\tau$ to $\orth(V_{2n+1})$ such that
$\tau_{\pm}(\ep)=\pm1$.
For $f\in\Hom_{\Delta\SO(V_{2n+1})}(\sigma\otimes\tau,\C)$ with $f\not=0$, we put
\[
f_{\pm}=f\circ(\Theta_{\psi,V_{2n+2},W_{2n}}(\pi)(\varepsilon)\otimes\tau_{\pm}(\varepsilon)).
\]
We see that 
\[
f_{\pm}\in\Hom_{\Delta\SO(V_{2n+1})}(\sigma\otimes\tau,\C).
\]
Moreover we have $(f_{\pm})_{\pm}=f$ and $f_{-}=-f_{+}$.
Since $\dim_\C\Hom_{\Delta\SO(V_{2n+1})}(\sigma\otimes\tau,\C)=1$ 
by \cite{AGRS} and \cite{AGRS2}, 
we have $f_{\pm}=a_\pm\cdot f$ for some $a_\pm\in\{\pm1\}$.
If $f_\bullet=f$, then 
$f\in\Hom_{\Delta\orth(V_{2n+1})}(\Theta_{\psi,V_{2n+2},W_{2n}}(\pi)\otimes\tau_\bullet,\C)$.
Hence this $\Hom$ space is nonzero.
\par

By the see-saw identity, we see that
\begin{align*}
\Hom_{\Sp(W_{2n})}(\Theta_{\psi,V_{2n+1},W_{2n}}(\tau_\bullet)\otimes\omega_{-\psi},\pi)
&\cong\Hom_{\orth(V_{2n+1})}(\Theta_{\psi,V_{2n+2},W_{2n}}(\pi),\tau_\bullet)\\
&\cong\Hom_{\Delta\orth(V_{2n+1})}(\Theta_{\psi,V_{2n+2},W_{2n}}(\pi)\otimes\tau_\bullet,\C)
\not=0,
\end{align*}
since $\tau_\bullet^\vee\cong\tau_\bullet$.
\par

Putting $\cl{\pi}\coloneqq\Theta_{\psi,V_{2n+1},W_{2n}}(\tau_\bullet)$,
we see that
\[
\Hom_{\Sp(W_{2n})}(\pi^\vee\otimes\cl{\pi}\otimes\overline{\omega_\psi},\C)\not=0.
\]
Note that $\pi\in \Pi_{\phi_N}$ implies that $\pi^\vee\in \Pi_{\phi_N}$.
By Lemma \ref{IS-irr}, 
we see that $\cl{\pi}$ is irreducible. 
Moreover, by Corollary \ref{Mp-param}, we have $\cl{\pi}\in \Pi_{\phi_M}$.
This completes the proof.
\end{proof}

%\subsection{Uniqueness of (FJ)}
\subsection{Uniqueness of (FJ)}
In this subsection, we prove the following proposition.
\begin{prop}
Let $\phi_M$ and $\phi_N$ be as in Proposition $\ref{exist}$.
Assume $(P1)$ and that
$L(s,\phi_N\otimes\chi_d)$ is regular at $s=1$ for some $d\in F^\times$.
Let $\cl{\pi}\in\Pi_{\phi_M}=\LL^{-1}_\psi(\phi_M)$ and $\pi\in\Pi_{\phi_N}$.
If they satisfy
\[
\Hom_{\Delta\Mp(W_{2n})}((\cl{\pi}\boxtimes\pi)\otimes\overline{\omega_\psi},\C)\not=0,
\]
then we have
\[
\eta_{\cl{\pi}}\times \eta_\pi=\chi_{N_1}\times\chi_M|A_\phi^+,
\]
where we put $\eta_{\cl{\pi}}=\iota_\psi(\cl{\pi})$ and $\eta_\pi=\iota_{\w'_{1}}(\pi)$.
\end{prop}
\begin{proof}
Note that
\begin{align*}
\Hom_{\Delta\Mp(W_{2n})}((\cl{\pi}\boxtimes\pi)\otimes\overline{\omega_{\psi}},\C)\not=0
&\Longleftrightarrow
\Hom_{\Sp(W_{2n})}(\cl{\pi}\otimes\omega_{-\psi},\pi^\vee)\not=0.
\end{align*}
\par

There exists a unique orthogonal space $V_{2n+1}$ such that
$\dim(V_{2n+1})=2n+1$, $\disc(V_{2n+1})=d$ and 
\[
\theta_{\psi,V_{2n+1},W_{2n}}(\tau)=\cl{\pi}
\]
for some $\tau\in\Irr(\orth(V_{2n+1}))$.
Let $\phi_\tau\colon\WD_F\rightarrow \Sp(M_\tau)$ be the element in $\cl\Phi(\SO(V_{2n+1}))$
associated to $\tau|\SO(V_{2n+1})$ and
$\eta_\tau=\iota_\w(\tau|\SO(V_{2n+1}))\in\widehat{A_{\phi_\tau}}$.
Then by Corollary \ref{Mp-param}, we have
\[
M_\tau=M\otimes\chi_{d},\quad
\eta_\tau(a)=
\eta_{\cl{\pi}}(a)\varepsilon(M^a)\varepsilon(M^a\otimes\chi_{d})\chi_{d}(-1)^{\dim_\C(M^a)/2}.
\]
Let $V_{2n+2}=V_{2n+1}\oplus V_1$ with $\dim(V_1)=1$ and $\disc(V_1)=-1$.
Then we have $\dim(V_{2n+2})=2n+2$ and $\disc(V_{2n+2})=-\disc(V_{2n+1})\cdot\disc(V_1)=d$.
By the see-saw argument, we have
\begin{align*}
0&\not=\Hom_{\Sp(W_{2n})}(\cl{\pi}\otimes\omega_{-\psi},\pi^\vee)\\
&\cong\Hom_{\orth(V_{2n+1})}(\Theta_{\psi,V_{2n+2},W_{2n}}(\pi^\vee),\tau)\\
&\cong\Hom_{\Delta\orth(V_{2n+1})}(\Theta_{\psi,V_{2n+2},W_{2n}}(\pi^\vee)\otimes\tau,\C)\\
&\subset\Hom_{\Delta\SO(V_{2n+1})}(\Theta_{\psi,V_{2n+2},W_{2n}}(\pi^\vee)\otimes\tau,\C)
\end{align*}
since $\tau^\vee\cong\tau$.
We write $\cl{\sigma}=\Theta_{\psi,V_{2n+2},W_{2n}}(\pi^\vee)$.
Since $\Hom_{\Delta\orth(V_{2n+1})}(\cl{\sigma}\otimes\tau,\C)\not=0$, 
we have $\cl{\sigma}\not=0$, and so that 
$\cl{\sigma}\in\Irr(\orth(V_{2n+2}))$ by Lemmas \ref{L(1,Ad)} and \ref{big theta}.
Let $\phi_\sigma\colon\WD_F\rightarrow \orth(N_\sigma)$ be the element in $\cl\Phi(\SO(V_{2n+2}))$ 
associated to $[\sigma]$
and put $\eta_\sigma=\iota_{\w_{1}}([\sigma])\in\widehat{A_{\phi_\sigma}^+}$,
where $\sigma\coloneqq\cl{\sigma}|\SO(V_{2n+2})\in\Irr(\SO(V_{2n+2}))$.
Then by (P1), we have
\[
N_\sigma=(N\otimes\chi_{d})\oplus\C
\quad\text{and}\quad
\eta_\sigma|A_{\phi_N}^+=\eta_{\pi^\vee},
\]
where we put $\eta_{\pi^\vee}=\iota_{\w'_1}(\pi^\vee)$.
\par

By (GPR) for $\SO(V_{2n+1})$ and Lemma \ref{L(1,Ad)}, we see that 
$\phi_\tau\in\cl\Phi_\gen(\SO(V_{2n+1}))$ and $\phi_\sigma\in\cl\Phi_\gen(\SO(V_{2n+2}))$.
Since $\Hom_{\Delta\SO(V_{2n+1})}(\sigma\otimes\tau,\C)\not=0$
and $-\disc(V_{2n+1})/\disc(V_{2n+2})=-1$, 
by (B), we see that
\begin{align*}
\eta_\tau(a)\cdot\iota_{\w_{-1}}([\sigma])(b)
&=\varepsilon(M_\tau^a\otimes N_\sigma)\det(M_\tau^a)(-1)^{\dim_\C(N_\sigma)/2}\det(N_\sigma)(-1)^{\dim_\C(M_\tau^a)/2}\\
&\times\varepsilon(M_\tau\otimes N_\sigma^b)\det(M_\tau)(-1)^{\dim_\C(N_\sigma^b)/2}\det(N_\sigma^b)(-1)^{\dim_\C(M_\tau)/2}
\end{align*}
for $a\in A_{\phi_\tau}$ and $b\in A_{\phi_\sigma}^+$.
\par
Now, for $a\in A_{\phi_M}=A_{\phi_\tau}$, we have
\begin{align*}
\varepsilon(M_\tau^a\otimes N_\sigma)
&=\varepsilon((M^a\otimes\chi_{d})\otimes ((N\otimes\chi_{d})\oplus\C))
=\varepsilon(M^a\otimes N)\varepsilon(M^a\otimes\chi_{d})\\
&=\varepsilon(M^a\otimes N_1)\varepsilon(M^a)\varepsilon(M^a\otimes\chi_{d}),
\\
\det(M_\tau^a)&=\det(M^a)=1,\quad\dim_\C(M_\tau^a)=\dim_\C(M^a).
\end{align*}
For $b\in A_{\phi_N}^+\subset A_{\phi_\sigma}^+$, 
we have $N_\sigma^b=N^b\otimes\chi_{d}$ and $N_1^b=N^b$.
Hence we have
\begin{align*}
\varepsilon(M_\tau\otimes N_\sigma^b)
&=\varepsilon((M\otimes\chi_{d})\otimes (N^b\otimes\chi_{d}))
=\varepsilon(M\otimes N_1^b),\\
\det(N_\sigma^b)&=\det(N^b)=\det(N_1^b),\quad\dim_\C(N_\sigma^b)=\dim_\C(N_1^b).
\end{align*}
Moreover we have
\begin{align*}
\det(N_\sigma)=\det(N)\chi_{d}=\det(N_1)\chi_{-d},\quad
\dim_\C(N_\sigma)=\dim_\C(N)+1=\dim_\C(N_1).
\end{align*}
Finally, by Proposition \ref{whittaker2} and Corollary \ref{Sp-contragredient}, 
we have
\[
\iota_{\w_{-1}}([\sigma])(b)=\eta_\sigma(b)\det(N_\sigma^b)(-1)
=\iota_{\w_1'}(\pi^\vee)(b)\det(N^b)(-1)
=\eta_{\pi}(b).
\]
Therefore we have
\begin{align*}
&(\eta_{\cl{\pi}}(a)\varepsilon(M^a)\varepsilon(M^a\otimes\chi_{d})\chi_{d}(-1)^{\dim_\C(M^a)/2})
\cdot\eta_\pi(b)\\
&=(\varepsilon(M^a\otimes N_1)\varepsilon(M^a)\varepsilon(M^a\otimes\chi_{d}))
\det(M^a)(-1)^{\dim_\C(N_1)/2}(\det(N_1)\chi_{d})(-1)^{\dim_\C(M^a)/2}\\
&\times\varepsilon(M\otimes N_1^b)\det(M)(-1)^{\dim_\C(N_1^b)/2}\det(N_1^b)(-1)^{\dim_\C(M)/2}.
\end{align*}
This gives the desired equation for $\eta_{\cl{\pi}}(a)\eta_{\pi}(b)$.
\end{proof}

%\subsection{(P1) and (P2) for tempered cases $\Rightarrow $ those for general cases}
\subsection{(P1) and (P2) for tempered cases $\Rightarrow $ those for general cases}
Obviously, 
(P1) (\resp (P2)) is true for $\phi\in\cl\Phi(\Sp(W_{2n}))$ 
(\resp $\phi'\in\cl\Phi(\SO(V_{2n}))$) such that $A_{\phi}^+=1$ (\resp $A_{\phi'}^+=1$).
Therefore, (P1) (\resp (P2)) for $\phi\in\cl\Phi(\Sp(W_{2n}))\setminus\cl\Phi_\temp(\Sp(W_{2n}))$
(\resp $\phi'\in\cl\Phi(\SO(V_{2n}))\setminus\cl\Phi_\temp(\SO(V_{2n}))$)
follows from the tempered case and
the compatibility of the local Langlands correspondence, 
the Langlands quotients and the local theta correspondence.
See \S \ref{Cn}, \S \ref{Dn} and \cite[Proposition C.4 $(\mathrm{ii})$]{GI1}.

%\subsection{Proof of (P2)}\label{proof of P2}
\subsection{Proof of (P2) for tempered cases}\label{proof of P2}
In this subsection, we show that (P1) implies (P2).
\par

Let $V=V_{2n}$ and $W=W_{2n}$.
We put $d=\disc(V)$ so that $\chi_V=\chi_d$.
Let $\phi'\in\cl\Phi_\temp(\SO(V))$ and put
\[
\phi=(\phi'\otimes\chi_d)\oplus\chi_d\in\cl\Phi_\temp(\Sp(W)).
\]
Let $[\sigma]\in\Pi_{\phi'}$.
Take an irreducible constituent $\cl{\sigma}$ of $\Ind_{\SO(V^\bullet)}^{\orth(V^\bullet)}(\sigma)$
such that the theta lift
$\pi=\Theta_{\psi,V^\bullet,W}(\cl{\sigma})$
is nonzero, so that $\pi\in\Pi_{\phi}$.
Fix $c_0\in F^\times$.
We have to show that
\[
\iota_{\w'_{c_0}}(\pi)|A_{\phi'}^+=\iota_{\w_{c_0}}([\sigma]).
\]
\par

We define an orthogonal space $V_1^\bullet$ by
$V_1^\bullet=Fe\oplus V^\bullet\oplus Fe^*$ equipped with
the pairing $\pair{\cdot,\cdot}_{V_1^\bullet}$ which
is an extension of $\pair{\cdot,\cdot}_{V^\bullet}$ and satisfies
$\pair{e,e}_{V_1^\bullet}=\pair{e^*,e^*}_{V_1^\bullet}=0$ and $\pair{e,e^*}_{V_1^\bullet}=1$.
Consider the theta correspondence for $(\Sp(W),\orth(V_1^\bullet))$.
Let $\omega=\omega_{\psi,V^\bullet,W}$ and $\omega_1=\omega_{\psi,V_1^\bullet,W}$ 
be Weil representations of 
$\Sp(W)\times\orth(V^\bullet)$ and $\Sp(W)\times\orth(V_1^\bullet)$, respectively.
Let $P=M_PU_P$ be the maximal parabolic subgroup of $\orth(V_1^\bullet)$ 
stabilizing $Fe$, where $M_P$ is the Levi component of $P$ stabilizing $Fe^*$.
Note that $M_P\cong \GL(1,F)\times \orth(V^\bullet)$.
By an explicit formula of the Weil representation (Lemma \ref{Weil3} below),
we see that
there exists a surjective $\Sp(W)\times P$-homomorphism
\[
\omega_1\rightarrow \omega\boxtimes |\cdot|_F^n.
\]
Here, $U_P$ acts on $\omega\boxtimes |\cdot|_F^n$ trivially.
In particular, 
if $\pi\in\Irr(\Sp(W))$ participates in the theta correspondence with 
$\orth(V^\bullet)$, then 
$\pi$ does in the theta correspondence with $\orth(V_1^\bullet)$.
See also \cite[p. 67]{MVW} and \cite[$\mathrm{III}$.~4]{Ku1}.
\par

We put $\cl{\sigma}_1=\Theta_{\psi,V_1^\bullet,W}(\pi)$ to be the theta lift to $\orth(V_1^\bullet)$
and $\sigma_1=\cl{\sigma}_1|\SO(V_1^\bullet)$.
As we have seen, this is nonzero,
so that $\sigma_1$ is irreducible by Propositions \ref{Theta} and \ref{param}.
There exists an exact sequence of $\Sp(W)\times\orth(V_1^\bullet)$-modules:
\[
\begin{CD}
1@>>>\Sch_1[\pi]@>>>\omega_1 @>>>\pi\boxtimes \cl{\sigma}_1@>>>1,
\end{CD}
\]
where the kernel $\Sch_1[\pi]$ is given by
\[
\Sch_1[\pi]=\bigcap_{f\in\Hom_{\Sp(W)}(\omega_1,\pi)}\ker(f).
\]
On the other hand, there exists a surjective 
$\Sp(W)\times P$-homomorphism
\[
\omega_1\rightarrow \omega\boxtimes |\cdot|_F^n\rightarrow 
\pi\boxtimes\cl{\sigma}\boxtimes|\cdot|_F^n.
\]
Since this map kills $\Sch_1[\pi]$, the diagram
\[
\xymatrix{
1  \ar@{->}[d]  \\
\Sch_1[\pi] \ar@{->}[d] \ar@{->}[rr]&& 0 \ar@{->}[d]\\
\omega_1  \ar@{->}[d] \ar@{->>}[r] & \omega\boxtimes |\cdot|_F^n \ar@{->>}[r] &     
\pi\boxtimes\cl{\sigma}\boxtimes|\cdot|_F^n  \\
\pi\boxtimes \cl{\sigma}_1  \ar@{->}[d] \ar@{-->}[urr]\\
1
}
\]
gives a surjective $\Sp(W)\times P$-homomorphism
\[
\pi\boxtimes (\cl{\sigma}_1|P)\rightarrow \pi\boxtimes\cl{\sigma}\boxtimes|\cdot|_F^n.
\]
In particular, we get a nonzero $P$-homomorphism
\[
\cl{\sigma}_1|P \rightarrow\cl{\sigma}\boxtimes|\cdot|_F^n.
\]
This implies that 
\[
\Hom_{\SO(V_1^\bullet)}(\sigma_1,
\Ind_{P^\circ}^{\SO(V_1^\bullet)}(\sigma\boxtimes\1))
\cong \Hom_{P^\circ}(\sigma_1|P^\circ, \sigma \boxtimes |\cdot|_F^n)
\]
is nonzero.
Note that the induction $\Ind_{P^\circ}^{\SO(V_1^\bullet)}(\sigma\boxtimes\1)$
does not depend on the choice of a representative of $[\sigma]$
by Lemma \ref{well-defined} (\ref{well2}).
Hence we have
\[
\iota_{\w_{c_0}}([\sigma_1])|A_{\phi'}^+=\iota_{\w_{c_0}}([\sigma]).
\]
Since
$\iota_{\w_{c_0}}([\sigma_1])|A_{\phi}^+=\iota_{\w'_{c_0}}(\pi)$
by (P1) and $A_{\phi'}^+\subset A_\phi^+$, we have
\[
\iota_{\w'_{c_0}}(\pi)|A_{\phi'}^+=(\iota_{\w_{c_0}}([\sigma_1])|A_{\phi}^+)|A_{\phi'}^+
=\iota_{\w_{c_0}}([\sigma_1])|A_{\phi'}^+=\iota_{\w_{c_0}}([\sigma]),
\]
as desired.

%\section{Preparations for the proof of Theorem \ref{last}}
\section{Preparations for the proof of (P1) for tempered $L$-parameters}
Let $V=V_{2m}$ and $W=W_{2n}$.
To prove (P1) for tempered $L$-parameters, we need to introduce more notation.

%\subsection{Haar measures}\label{measure}
\subsection{Haar measures}\label{measure}
Let $k$ be a positive integer.
As in \S \ref{parabolic}, we put
\[
V'=X+V+X^*,\quad
W'=Y+W+Y^*
\]
with
$X=X_k$, $X^*=X_k^*\subset V'$ and $Y=Y_k$, $Y^*=Y_k^*\subset W'$.
Hence $\dim(V')=2m'\coloneqq2(m+k)$ and $\dim(W)=2n'\coloneqq2(n+k)$.
Let $P=P_k=M_PU_P\subset \orth(V')$ and $Q=Q_k=M_QU_Q\subset\Sp(W')$
be the parabolic subgroups defined in \S \ref{parabolic}.
Hence $M_P\cong \GL(X)\times\orth(V)$ and $M_Q\cong\GL(Y)\times\Sp(W)$.
We need to choose Haar measures on various groups.
In particular, we shall define Haar measures on $U_P$ and $U_Q$ as follows.
\par

Recall the symplectic form $\pair{\cdot,\cdot}=\pair{\cdot,\cdot}_{V'}\otimes \pair{\cdot,\cdot}_{W'}$ 
on $V'\otimes W'$ over $F$, and 
the maps $I_X\colon X^*\rightarrow X$ and $I_Y\colon Y^*\rightarrow Y$. 
We consider the following spaces and pairings:
\begin{itemize}
\item
$(x,y)\mapsto \psi(\pair{x,I_Y^{-1}y})$ for $x,y\in V'\otimes Y$;
\item
$(x,y)\mapsto \psi(\pair{x,I_Yy})$ for $x,y\in V'\otimes Y^*$;
\item
$(x,y)\mapsto \psi(\pair{x,I_Yy})$ for $x,y\in V\otimes Y^*$;
\item
$(x,y)\mapsto \psi(\pair{I_X^{-1}x,y})$ for $x,y\in X\otimes W$;
\item
$(x,y)\mapsto \psi(\pair{I_Xx,y})$ for $x,y\in X^*\otimes W$;
\item
$(x,y)\mapsto \psi(\pair{I_X^{-1}x,I_Yy})$ for $x,y\in X\otimes Y^*$;
\item
$(x,y)\mapsto \psi(\pair{I_Xx,I_Y^{-1}y})$ for $x,y\in X^*\otimes Y$;
\item
$(x,y)\mapsto \psi(\pair{I_Xx,I_Yy})$ for $x,y\in X^*\otimes Y^*$.
\end{itemize}
On these spaces, we take the self-dual Haar measures with respect to these pairings.
Put
\[
e^{**}=v_1^*\otimes w_1^*+\dots+v_k^*\otimes w_k^*\in X^*\otimes Y^*.
\]
\begin{itemize}
\item
We transfer the Haar measure on $V\otimes Y^*$ to $\Hom(V,X)$ via the isomorphism
$b\mapsto b^*e^{**}$ for $b\in\Hom(V,X)$.
\item
We transfer the Haar measure on $X^*\otimes W$ to $\Hom(W,Y)$ via the isomorphism
$b\mapsto b^*e^{**}$ for $b\in\Hom(W,Y)$.
\end{itemize}
Furthermore:
\begin{itemize}
\item
We transfer the Haar measure on $X\otimes Y^*$ to $\Hom(X^*,X)$ via the isomorphism
$c\mapsto ce^{**}$ for $c\in\Hom(X^*,X)$.
This Haar measure on $\Hom(X^*,X)$ is self-dual with respect to the pairing
$(c_1,c_2)\mapsto \psi(\pair{I_X^{-1}c_1e^{**},I_Yc_2e^{**}})$.
\item
We take the Haar measure 
\[
|2|_F^{-\frac{k(k-1)}{4}}dc
\]
on $\Sym(X^*,X)$, where $dc$ is the self-dual Haar measure with respect to the pairing
$(c_1,c_2)\mapsto \psi(\pair{I_X^{-1}c_1e^{**},I_Yc_2e^{**}})$.
\item
We transfer the Haar measure on $X^*\otimes Y$ to $\Hom(Y^*,Y)$ via the isomorphism
$c\mapsto ce^{**}$ for $c\in\Hom(Y^*,Y)$.
This Haar measure on $\Hom(Y^*,Y)$ is self-dual with respect to the pairing
$(c_1,c_2)\mapsto \psi(\pair{I_Xc_1e^{**},I_Y^{-1}c_2e^{**}})$.
\item
We take the Haar measure 
\[
|2|_F^{-\frac{k(k-1)}{4}}dc
\]
on $\Sym(Y^*,Y)$, where $dc$ is the self-dual Haar measure with respect to the pairing
$(c_1,c_2)\mapsto \psi(\pair{I_Xc_1e^{**},I_Y^{-1}c_2e^{**}})$.
\end{itemize}
Then:
\begin{itemize}
\item
We take the Haar measure $du=dbdc$ on $U_P$ for $u=u_P(b)u_P(c)=u_P(c)u_P(b)$
with $b\in\Hom(V,X)$ and $c\in \Sym(X^*,X)$.
\item
Similarly, we define the Haar measure on $U_Q$.
\end{itemize}
We note the following Fourier inversion formula:
\begin{lem}\label{inversion}
For $\varphi\in\Sch(X\otimes Y^*)$, we have
\[\int_{\Sym(Y^*,Y)}\left(
\int_{\Hom(X^*,X)}\varphi(xe^{**})\psi(\pair{xe^{**},c'e^{**}})dx
\right)dc'
=\int_{\Sym(X^*,X)}\varphi(ce^{**})dc.
\]
\end{lem}
\begin{proof}
The proof is similar to that of \cite[Lemma 7.1]{GI2}.
\end{proof}

%\subsection{Weil representations}\label{Weil rep}
\subsection{Weil representations}\label{Weil rep}
We recall some explicit formulas for the Weil representations.
Recall that $V'$ and $W'$ have decompositions
\begin{align*}
V'=X+V+X^*,\quad
W'=Y+W+Y^*
\end{align*}
with $2m=\dim(V)$, $2n=\dim(W)$, $2m'=\dim(V')$, 
$2n'=\dim(W')$ and $k=\dim(X)=\dim(Y)$.
We decompose $W=Y_n+Y_n^*$ as in \S \ref{Symplectic}. 
We have fixed a non-trivial additive character $\psi$ of $F$.
\par

Let  $\H(W)=W\oplus F$ be the associated Heisenberg group, i.e., 
the multiplication law is given by
\[
(w,t)\cdot(w',t')=\left(w+w',t+t'+{1\over2} \pair{w,w'}_W\right)
\]
for $w,w'\in W$ and $t,t'\in F$.
Let $\rho$ be the Heisenberg representation of $\H(W)$ on $\Sch(Y_n^*)$ 
with the central character $\psi$.
Namely, 
\[
\rho(y+y',t)\varphi(y_1')=\psi(t+\pair{y_1',y}_W+{1\over2}\pair{y',y}_W)\varphi(y_1'+y')
\]
for $\varphi\in\Sch(Y_n^*)$, $y\in Y_n$, $y',y_1'\in Y_n^*$ and $t\in F$.
\par

For simplicity, we write:
\begin{itemize}
\item
$\omega$ for the Weil representation $\omega_{\psi,V,W}$ of 
$\Sp(W)\times\orth(V)$ on a space $\Sch$;
\item
$\omega'$ for the Weil representation $\omega_{\psi,V',W}$ of $\Sp(W)\times\orth(V')$
on a space $\Sch'$;
\item
$\omega''$ for the Weil representation $\omega_{\psi,V',W'}$ of $\Sp(W')\times\orth(V')$ 
on a space $\Sch''$.
\end{itemize}
We take the Schr\"odinger model 
\[
\Sch=\Sch(V\otimes Y_n^*)
\]
of $\omega$.
We take a mixed model 
\[
\Sch'=\Sch \otimes \Sch(X^*\otimes W)=\Sch(V\otimes Y_n^*)\otimes\Sch(X^*\otimes W)
\]
of $\omega'$, where we regard $\Sch'$ as a space of functions on $X^*\otimes W$ 
with values in $\Sch$.
Similarly, we take a mixed model
\[
\Sch'' = \Sch(V'\otimes Y^*)\otimes \Sch'
=\Sch(V'\otimes Y^*)\otimes \Sch(V\otimes Y_n^*)\otimes\Sch(X^*\otimes W)
\]
of $\omega''$, where we regard $\Sch''$ as a space of functions on $V'\otimes Y^*$
with values in $\Sch'$.
Also, we write:
\begin{itemize}
\item
$\rho$ for the Heisenberg representation of $\H(V\otimes W)$ on $\Sch$
with the central character $\psi$;
\item
$\rho'$ for the Heisenberg representation of $\H(V'\otimes W)$ on $\Sch'$
with the central character $\psi$.
\end{itemize}
\par

\begin{lem}\label{Weil1}
Let $Q_n=M_{Q_n}U_{Q_n}$ be the Siegel parabolic subgroup of $\Sp(W)$ stabilizing $Y_n$, 
and define $m(a')=m_{Q_n}(a')\in M_{Q_n}$ and $u(c')=u_{Q_n}(c')\in U_{Q_n}$ 
for $a'\in \GL(Y_n)$ and $c'\in\Sym(Y_n^*,Y_n)$ as in \S $\ref{parabolic}$.
We put $I=I_{Q_n}\in\Hom(Y_n^*,Y_n)$ and $w=w_{Q_n}\in\Sp(W)$ as in \S $\ref{parabolic}$.
Then for $\varphi\in\Sch=\Sch(V\otimes Y_n^*)$ and $x\in V\otimes Y_n^*$, we have
\begin{align*}
[\omega(1,h)\varphi](x)&=\varphi(h^{-1}x),
&h\in\orth(V),\\
[\omega(m(a'),1)\varphi](x)&=\chi_{V}(\det(a'))|\det(a')|_F^{m}\varphi(a'^*x),
&a'\in\GL(Y_n),\\
[\omega(u(c'),1)\varphi](x)&=\psi({1\over2}\pair{c'x,x})\varphi(x),
&c'\in\Sym(Y_n^*,Y_n),\\
[\omega(w,1)\varphi](x)
&=\gamma_{V}^{-n}\int_{V\otimes Y_n}\varphi(I^{-1}z)\psi(-\pair{z,x})dz.
\end{align*}
Here, $dz$ is the self-dual measure on $V\otimes Y_n$ with respect to the pairing 
$(x,y)\mapsto\psi(\pair{y,I_0^{-1}x})$,
and $\gamma_{V}$ is a $4$-th root of unity satisfying $\gamma_{V}^2=\chi_{V}(-1)$.
\end{lem}
\begin{proof}
This formula is the Schr\"odinger model.
\end{proof}

\begin{lem}\label{Weil2}
The mixed model $\Sch(V'\otimes Y^*)\otimes \Sch(V'\otimes Y_n^*)$
of $\omega''$ is given as follows:
For $\varphi=\varphi_1\otimes\varphi'\in \Sch(V'\otimes Y^*)\otimes \Sch(V'\otimes Y_n^*)$
and $(x_1,x')\in (V'\otimes Y^*)\times(V'\otimes Y_n^*)$, 
we have
\begin{align*}
[\omega(g,h')\varphi](x_1,x')
&=\varphi_1(h'^{-1}x_1)\cdot [\omega'(g,h')\varphi'](x'),
&(g,h')\in\Sp(W)\times \orth(V'),\\
[\omega(m_Q(a'),1)\varphi](x_1,x')
&=\chi_V(\det(a'))|\det(a')|_F^{m'}\varphi_1(a'^*x_1)\cdot\varphi'(x'),
&a'\in\GL(Y),\\
[\omega(u_Q(b'),1)\varphi](x_1,x')
&=\varphi_1(x_1)\cdot[\rho'(b'^*x_1,0)\varphi'](x'),
&b'\in\Hom(W,Y),\\
[\omega(u_Q(c'),1)\varphi](x_1,x')
&=\psi({1\over2}\pair{c'x_1,x_1})\varphi_1(x_1)\cdot\varphi'(x'),
&c'\in\Sym(Y^*,Y),\\
[\omega(w_Q,1)\varphi](x_1,x')
&=\gamma_V^{-k}\int_{V'\otimes Y}\varphi(I_Y^{-1}z)\psi(-\pair{z,x_1})dz\cdot\varphi'(x').
\end{align*}
\end{lem}
\begin{proof}
This formula is obtained by the canonical isomorphism
\[
\Sch(V'\otimes (Y^* \oplus Y_n^*))\rightarrow \Sch(V'\otimes Y^*)\otimes \Sch(V'\otimes Y_n^*).
\]
\end{proof}

\begin{lem}\label{Weil3}
For $\varphi'=\varphi_2\otimes\varphi_3\in \Sch'=\Sch(V\otimes Y_n^*)\otimes\Sch(X^*\otimes W)$,
we have
\begin{align*}
[\omega'(g,h)\varphi'](x_2,x_3)
&=[\omega(g,h)\varphi_2](x_2)\cdot\varphi_3(g^{-1}x_3),
&(g,h)\in\Sp(W)\times \orth(V),\\
[\omega'(1,m_P(a))\varphi'](x_2,x_3)
&=|\det(a)|_F^{n}\varphi_2(x_2)\cdot\varphi_3(a^*x_3),
&a\in\GL(X),\\
[\omega'(1,u_P(b))\varphi'](x_2,x_3)
&=[\rho(b^*x_3,0)\varphi_2](x_2)\cdot \varphi_3(x_3),
&b\in\Hom(V,X),\\
[\omega'(1,u_P(c))\varphi'](x_2,x_3)
&=\psi({1\over2}\pair{cx_3,x_3})\varphi_2(x_2)\cdot \varphi_3(x_3),
&c\in\Sym(X^*,X),\\
[\omega'(1,w_P)\varphi'](x_2,x_3)
&=\varphi_2(x_2)\cdot\int_{X\otimes W}\varphi_3(I_X^{-1}z)\psi(-\pair{z,x_3})dz.
\end{align*}
Moreover we have
\[
[\rho'(v+v_0+v^*,0)\varphi'](x_2,x_3)
=\psi(\pair{x_3,v}+{1\over 2}\pair{v^*,v})
[\rho(v_0,0)\varphi_2](x_2)\cdot\varphi_3(x_3+v^*)
\]
for $v\in X\otimes W$, $v_0\in V\otimes W$ and $v^*\in X^*\otimes W$.
\end{lem}
\begin{proof}
These formulas are given by the partial Fourier transform
\[
\Sch(V'\otimes Y_n^*)\rightarrow \Sch(V\otimes Y_n^*)\otimes \Sch(X^*\otimes W)
,\quad\varphi\mapsto \hat{\varphi}
\]
defined by
\[
\hat{\varphi}(x,y)=\int_{X\otimes Y_n^*}\varphi(
\begin{pmatrix}
z\\ x\\ y_2
\end{pmatrix}
)\psi(-\pair{z,y_1})dz
\]
for $x\in V\otimes Y_n^*$ and $y=y_1+y_2$ with
$y_1\in X^*\otimes Y_n$ and $y_2\in X^*\otimes Y_n^*$.
Here, $dz$ is the self-dual measure on $X\otimes Y_n^*$ with respect to the pairing 
$(x,y)\mapsto\psi(\pair{I_X^{-1}x,Iy})$.
\end{proof}

In particular, the map $\Sch'\ni\varphi\mapsto\varphi(\cdot,0)\in\Sch$ gives
a surjective $\Sp(W)\times P$-homomorphism
\[
\omega'\rightarrow \omega\otimes |\det|_F^{n}, 
\]
where $P=M_PU_P$ is the parabolic subgroup defined in \S \ref{parabolic}
such that $M_P \cong \GL(X) \times \orth(V)$.
Here, $U_P$ acts on $\omega\boxtimes |\det|_F^n$ trivially.
This fact was used in \S \ref{proof of P2}.
\par

By the above lemmas, we get a formula for
the mixed model 
$\Sch''=\Sch(V'\otimes Y^*)\otimes \Sch(V\otimes Y_n^*)\otimes\Sch(X^*\otimes W)$ of $\omega''$.
\par

%\subsection{Normalized intertwining operators}\label{sec.intertwining}
\subsection{Normalized intertwining operators}\label{sec.intertwining}
In this subsection, we define normalized intertwining operators
which are used to describe the local Langlands correspondence.
\par

We assume that $V$ is type $(d,c)$, so that $\SO(V)$ is quasi-split.
As in \S \ref{parabolic}, 
we define parabolic subgroups $P=P_k=M_PU_P$ of $\orth(V')$
and $Q=Q_k=M_QU_Q$ of $\Sp(W')$ such that
\[
M_P\cong\GL(X)\times \orth(V),
\quad
M_Q\cong\GL(Y)\times\Sp(W).
\]
We put $P^\circ= P\cap \SO(V')$ and $M_P^\circ=M_P\cap\SO(V')\cong \GL(X)\times\SO(V)$.
Assume that $k$ is even and $\dim(V)=2m\geq 4$.
We identify $\GL(X)$ (\resp $\GL(Y)$) with $\GL_k(F)$ using the basis
$\{v_1,\dots,v_k\}$ (\resp $\{w_1,\dots,w_k\}$).
Hence we can define an isomorphism $i\colon\GL(Y)\rightarrow \GL(X)$ via these identifications.
\par

Let $\tau$ be an irreducible tempered representation of $\GL_k(F)$ on
a space $\VV_\tau$ with a central character $\omega_\tau$.
We may regard $\tau$ as a representation of $\GL(X)$ or $\GL(Y)$ via the above identifications.
For any $s\in\C$, we realize the representation
$\tau_s \coloneqq \tau\otimes |\det|_F^s$ on $\VV_\tau$ by setting 
$\tau_s(a)v \coloneqq |\det(a)|_F^s\tau(a)v$ for $v\in \VV_\tau$ and $a\in \GL_k(F)$.
Let $\sigma$ (\resp $\pi$) be an irreducible tempered representation of $\SO(V)$
(\resp $\Sp(W)$) on a space $\VV_{\sigma}$ (\resp $\VV_{\pi}$).
Assume that $\sigma$ is $\epinv$, i.e., there exists $\cl{\sigma}\in\Irr(\orth(V))$
such that $\cl{\sigma}|\SO(V)=\sigma$.
We may assume that $\cl{\sigma}$ is realized on $\VV_{\sigma}$.
We consider the induced representations
\[
\Ind_{P}^{\orth(V')}(\tau_s\otimes \cl{\sigma})
\quad \text{and}\quad
\Ind_{Q}^{\Sp(W')}(\tau_s\otimes \pi)
\]
of $\orth(V')$ and $\Sp(W')$, respectively.
They are realized on the spaces of smooth functions
$\Phi_s\colon \orth(V')\rightarrow \VV_\tau\otimes \VV_{\sigma}$
and $\Phi'_s\colon \Sp(W')\rightarrow \VV_\tau\otimes \VV_{\pi}$
such that
\begin{align*}
\Phi_s(u_Pm_P(a)hh')&=|\det(a)|_F^{s+\rho_P}\tau(a)\cl{\sigma}(h)\Phi_s(h')\\
\text{and}\quad
\Phi'_s(u_Qm_Q(a')gg')&=|\det(a')|_F^{s+\rho_Q}\tau(a')\pi(g)\Phi'_s(g')
\end{align*}
for any $u_P\in U_P$, $a\in\GL(X)$, $h\in\orth(V)$, $h'\in\orth(V')$, and
$u_Q\in U_Q$, $a'\in\GL(Y)$, $g\in\Sp(W)$ and $g'\in\Sp(W')$, respectively.
By Lemma \ref{well-defined} (\ref{well3}), we have
\[
\Ind_{P}^{\orth(V')}(\tau_s\otimes \cl{\sigma})|\SO(V')
\cong
\Ind_{P^\circ}^{\SO(V')}(\tau_s\otimes {\sigma}).
\]
\par

Let $A_P$ (\resp $A_Q$) be the split component of the center of $M_P^\circ$ (\resp $M_Q$) and
$W(M_P^\circ)=\Norm(A_P,\SO(V'))/M_P^\circ$ (\resp $W(M_Q)=\Norm(A_Q,\Sp(W'))/M_Q$) 
be the relative Weyl group for $M_P^\circ$ (\resp $M_Q$).
Note that $W(M_P^\circ)\cong W(M_Q)\cong\Z/2\Z$.
We denote by $w$ (\resp $w'$) the non-trivial element in $W(M_P^\circ)$ (\resp $W(M_Q)$).
\par

The definition of the normalized intertwining operators is very subtle 
because one has to choose the following data appropriately:
\begin{itemize}
\item
a representative $\cl{w}$ of $w$ (\resp $\cl{w}'$ of $w'$);
\item
a Haar measure on $U_P$ (\resp $U_Q$) to define the unnormalized intertwining operator;
\item
a normalizing factor $\gamma(w,\tau_s\otimes\sigma)$ (\resp $\gamma(w',\tau_s\otimes\pi)$);
\item
an intertwining isomorphism $\AA_w$ (\resp $\AA_{w'}$).
\end{itemize}
To do these, we need an $F$-splitting of $\SO(V')$ (\resp $\Sp(W')$).
\par

Let $(B,T)$ (\resp $(B',T')$) be the pair of the Borel subgroup and the maximal torus of
$\SO(V')$ (\resp $\Sp(W')$) as in \S \ref{Quadratic} (\resp \S \ref{Symplectic}).
They are $F$-rational.
The Lie algebras of $\SO(V')$ and $\Sp(W')$ are given by
\begin{align*}
\Lie(\SO(V'))&=\{x\in \End(V')\ |\ \pair{xv,v'}_{V'}+\pair{v,xv'}_{V'}=0
\quad\text{for any $v,v'\in V'$}\}
\end{align*}
and
\[
\Lie(\Sp(W'))=\{y\in \End(W')\ |\ \pair{yw,w'}_{W'}+\pair{w,yw'}_{W'}=0
\quad\text{for any $w,w'\in W'$}\}, 
\]
respectively.
Let $v_i$, $v_i^*$, $e$ and $e'$  for $k+1\leq i\leq m'-1=m+k-1$
(\resp $w_j$, $w_j^*$ for $k+1\leq j\leq n'=n+k$)
be the elements in $V$ (\resp in $W$)
as in \S \ref{Quadratic}.
Hence 
\[
\pair{e,e}_V=2c,\quad \pair{e',e'}_V=-2cd,\quad \pair{e,e'}_V=0.
\]
For $t\in T$ and $t'\in T'$, we define $t_i,t'_j\in F^\times$ and $a(t), b(t)\in F$ by
\begin{align*}
t_i=\pair{tv_i,v_i^*}_{V'},\quad
t'_j=\pair{t'w_j,w_j^*}_{W'}
\end{align*}
and
\[
\left\{
\begin{aligned}
te&=a(t)\cdot e+b(t)\cdot e',\\
te'&=b(t)d\cdot e+a(t)\cdot e'.
\end{aligned}
\right.
\]
Note that $a(t)^2-b(t)^2d=1$.
Then the simple roots of $T$ in $B\subset \SO(V')$ 
and the ones of $T'$ in $B'\subset \Sp(W')$ are given by
\begin{align*}
\alpha_i(t)&=t_i/t_{i+1},
\quad
\alpha_\pm(t)=t_{m'-1}(a(t)\pm b(t)\sqrt{d})\\
\text{and}\quad
\alpha'_j(t')&=t'_j/t'_{j+1},
\quad
\alpha'_{n'}(t')=(t'_{n'})^2
\end{align*}
for $1\leq i\leq m'-2$ and $1\leq j\leq n'-1$, respectively.
We choose root vectors $x_{\alpha_\bullet}\in \Lie(\SO(V'))$ satisfying
\[
x_{\alpha_i}v_{i'}=
\left\{
\begin{aligned}
&v_i	\iif i'=i+1,\\
&0	\other,
\end{aligned}
\right.
\quad
x_{\alpha_i}v_{i'}^*=
\left\{
\begin{aligned}
&-v_{i+1}^*	\iif i'=i,\\
&0	\other,
\end{aligned}
\right.
\quad
x_{\alpha_i}e=x_{\alpha_i}e'=0
\]
for $1\leq i\leq m'-2$ and
\[
x_{\alpha_\pm}v_{i'}=0,
\quad
x_{\alpha_\pm}v_{i'}^*=
\left\{
\begin{aligned}
&-\frac{e\pm \sqrt{d}^{-1}e'}{2c}	\iif i'=m'-1,\\
&0	\other,
\end{aligned}
\right.
\quad
x_{\alpha_\pm}e=v_{m'-1},\quad
x_{\alpha_\pm}e'=\mp\sqrt{d}v_{m'-1}.
\]
Also, we choose root vectors $y_{\alpha'_\bullet}\in\Lie(\Sp(W'))$ satisfying
\[
y_{\alpha_j'}w_{j'}=
\left\{
\begin{aligned}
&w_j	\iif j'=j+1,\\
&0	\other,
\end{aligned}
\right.
\quad
y_{\alpha'_j}w_{j'}^*=
\left\{
\begin{aligned}
&-w_{j+1}^*	\iif j'=j,\\
&0	\other,
\end{aligned}
\right.
\]
for $1\leq j\leq n'-1$ and
\[
y_{\alpha'_{n'}}w_{j'}=0,
\quad
y_{\alpha'_{n'}}w_{j'}^*=
\left\{
\begin{aligned}
&w_{n'}	\iif j'=n',\\
&0	\other.
\end{aligned}
\right.
\]
We define splittings $\spl_{\SO(V')}$ of $\SO(V')$ and $\spl_{\Sp(W')}$ of $\Sp(W')$ by
\[
\spl_{\SO(V')}=(B,T,\{x_{\alpha_\bullet}\}_{\alpha_\bullet})
\quad\text{and}\quad
\spl_{\Sp(W')}=(B',T',\{y_{\alpha'_\bullet}\}_{\alpha'_\bullet}), 
\]
respectively.
These are $F$-splittings.
\par

Let $\alpha^\vee_\bullet$ be the simple coroot corresponding to $\alpha_\bullet$.
This is a homomorphism
$
\alpha^\vee_\bullet\colon \GL(1)\rightarrow T
$
of algebraic groups, and gives a map 
\[
d\alpha^\vee_\bullet\colon \Lie(\GL(1))=\G_a\rightarrow \Lie(T)
\]
of Lie algebras.
We put $H_{\alpha_{\bullet}}=d\alpha^\vee_\bullet(1)\in\Lie(T)$.
We take a root vector $x_{-\alpha_\bullet}$ of the negative root $-\alpha_\bullet$ such that
\[
[x_{\alpha_\bullet},x_{-\alpha_\bullet}]=H_\bullet.
\]
We denote by $W(T,\SO(V'))$ the Weyl group of $T$ in $\SO(V')$.
Let $w_{\alpha_\bullet}\in W(T,\SO(V'))$ be the reflection with respect to $\alpha_\bullet$.
Following the procedure of \cite[\S 2.1]{LS}, 
we take the representative $\cl{w}_{\alpha_\bullet}\in\SO(V')$ of $w_{\alpha_\bullet}$ defined by
\[
\cl{w}_{\alpha_\bullet}=\exp(x_{\alpha_\bullet})\exp(-x_{-\alpha_\bullet})\exp(x_{\alpha_\bullet}).
\]
We put $w_{\alpha_{m'-1}}\coloneqq w_{\alpha_+}w_{\alpha_-}=w_{\alpha_-}w_{\alpha_+}$ and
$\cl{w}_{\alpha_{m'-1}}\coloneqq 
\cl{w}_{\alpha_+}\cl{w}_{\alpha_-}=\cl{w}_{\alpha_-}\cl{w}_{\alpha_+}$.
Similarly, we take the representative $\cl{w}'_{\alpha'_\bullet}$ of 
the reflection $w'_{\alpha'_\bullet}\in W(T',\Sp(W'))$
with respect to $\alpha'_\bullet$.
Moreover, we define $w_{p,q}$ for $1\leq p,q\leq m'-1$ by
\[
w_{p,q}=w_{\alpha_p}w_{\alpha_{p+1}}\cdots 
w_{\alpha_{m'-2}}w_{\alpha_{m'-1}}w_{\alpha_{m'-2}}\cdots w_{\alpha_{q+1}}w_{\alpha_q}.
\]
Similarly, we define $w'_{p,q}$, $\cl{w}_{p,q}$ and $\cl{w}'_{p,q}$.
\par

Let $w_T\in W(T,\SO(V'))$ and $w'_{T'}\in W(T',\Sp(W'))$ 
be the representatives of $w\in W(M_P^\circ)$ and $w'\in W(M_Q)$
which stabilize the simple roots of $(B\cap M_P^\circ, T)$ and $(B'\cap M_Q, T')$, respectively.
It is easily seen that 
\begin{align*}
w_T=w_{k,k}\cdot w_{k-1,k}\cdots w_{1,k}
\quad\text{and}\quad
w'_{T'}=w'_{k,k}\cdot w'_{k-1,k}\cdots w'_{1,k},
\end{align*}
and this is a reduced decomposition of $w_T$ and $w'_{T'}$, respectively.
We take representatives $\cl{w}\in\SO(V')$ of $w$ 
and $\cl{w}'\in\Sp(W')$ of $w'$ defined by
\begin{align*}
\cl{w}=\cl{w}_{k,k}\cdot \cl{w}_{k-1,k}\cdots \cl{w}_{1,k}
\quad\text{and}\quad
\cl{w}'=\cl{w}'_{k,k}\cdot \cl{w}'_{k-1,k}\cdots \cl{w}'_{1,k},
\end{align*}
respectively.
Then we see that
\[
\cl{w}=w_P\cdot m_P(a) \cdot (-\1_V)^k
\quad \text{and}\quad 
\cl{w}'=w_Q \cdot m_Q(a') \cdot (-\1_W)^k,
\]
where $a\in\GL_k(F)\cong\GL(X)$ and $a' \in \GL_k(F)\cong\GL(Y)$ are given by
\[
a=
\begin{pmatrix}
&&(-1)^{m'-k+1}c\\
&\iddots&\\
(-1)^{m'}c&&
\end{pmatrix}
\quad\text{and}\quad
a'=
\begin{pmatrix}
&&(-1)^{n'-k+1}\\
&\iddots&\\
(-1)^{n'}&&
\end{pmatrix},
\]
respectively.
\par

We have defined an $F$-splitting $\spl_{\SO(V')}$ of $\SO(V')$ 
(\resp $\spl_{\Sp(W')}$ of $\Sp(W')$).
This splitting determines a Chevalley basis of $\Lie(\SO(V'))$ (\resp $\Lie(\Sp(W'))$), 
and hence an invariant $F$-valued differential form of highest degree on 
$U_P$ (\resp $U_Q$). 
Its absolute value, together with the self-dual Haar measure on $F$ with respect to $\psi$, 
gives a Haar measure $du_P'$ on $U_P$ (\resp $du_Q'$ on $U_Q$).
This is the measure that we take in the definition of the unnormalized intertwining operator.
\begin{lem}
These measures satisfy
\[
du_P'=\left|c\right|_F^{k\rho_P}du_P
\quad\text{and}\quad
du_Q'=du_Q,
\]
where $du_P$ and $du_Q$ are the measures defined in \S $\ref{measure}$.
\end{lem}
\begin{proof}
We show this lemma only for $du_P'$.
One can prove this lemma for $du_Q'$ in a similar way.
\par

Let $B=TU$ be the Borel subgroup of $\SO(V')$ defined in \S \ref{Quadratic}.
For $t\in T$, we defined $t_i\in F^\times$ and $a(t),b(t)\in F$ as above.
We put
\[
\alpha_{i,j}(t)=t_i/t_j,\quad
\alpha_{i,\pm}(t)=t_i(a(t)\pm b(t)\sqrt{d}), \quad
\beta_{i,j}(t)=t_it_j.
\]
Then the set of roots of $T$ in $\Lie(U)$ is given by
\[
\{\alpha_{i,j},\beta_{i,j}\ |\ 1\leq i<j\leq m'-1\}
\cup
\{\alpha_{i,\pm}\ |\ 1\leq i\leq m'-1\}.
\]
Note that $\alpha_{i,i+1}=\alpha_i$ and $\alpha_{m'-1,\pm}=\alpha_\pm$
are simple roots, 
so that we have defined the root vectors $x_{\alpha_{i,i+1}}=x_{\alpha_i}$ and 
$x_{\alpha_{m'-1,\pm}}=x_{\alpha_\pm}$.
Moreover we have
\[
\alpha_{i,j}=\alpha_{i}+\alpha_{i+1,j},\quad
\alpha_{i,\pm}=\alpha_{i,m'-1}+\alpha_\pm,\quad
\beta_{i,j}=\alpha_{i,+}+\alpha_{j,-}.
\]
We define root vectors $x_{\alpha_\bullet}, x_{\beta_\bullet}\in\Lie(\SO(V'))$ by 
\[
x_{\alpha_{i,j}}=[x_{\alpha_{i}},x_{\alpha_{i+1,j}}],\quad
x_{\alpha_{i,\pm}}=[x_{\alpha_{i,m'-1}},x_{\alpha_\pm}],\quad
x_{\beta_{i,j}}=[x_{\alpha_{i,+}},x_{\alpha_{j,-}}].
\]
Then the basis 
\[
\{x_{\alpha_{i,j}}\ |\ 1\leq i\leq k< j \leq m'-1\}
\cup
\{x_{\alpha_{i,\pm}}\ |\ 1\leq i\leq k\}
\cup
\{x_{\beta_{i,j}}\ |\ 1\leq i\leq k,\ i<j\leq m'-1\}
\]
of $\Lie(U_P)$
is a part of a Chevalley basis of $\Lie(\SO(V'))$.
Let $\{dx_{\alpha_\bullet},dx_{\beta_\bullet}\}$ be the dual basis of the linear dual of $\Lie(U_P)$,
and put
\[
\omega=(\bigwedge dx_{\alpha_{\bullet}})\wedge (\bigwedge dx_{\beta_{\bullet}}).
\]
This is an $F$-valued differential form of highest degree on $U_P$ 
(defined up to a multiplication of $\pm1$).
\par

On the other hand, we can identify $U_P$ with $F^l$ as a topological spaces by the map
\[
U_P\rightarrow F^{l}, u\mapsto 
((\pair{uv_j,v_i^*}_{V'})_{1\leq i\leq k<j\leq m'-1},\ 
(\pair{ue,v_i^*}_{V'},\pair{ue',v_i^*}_{V'})_{1\leq i\leq k},\ 
(\pair{uv_j^*,v_i^*}_{V'})_{1\leq i\leq k,i<j\leq m'-1}),
\]
where 
\[
l=\dim(U_P)=2km'-\frac{3}{2}k^2-\frac{1}{2}k.
\]
This gives an $F$-valued differential form $\omega_0$ of highest degree on $U_P$.
Then we have
\[
\omega=\pm (2\sqrt{d})^{-k}c^{\sum_{i=1}^k(m'-1-i)}\omega_0.
\]
\par

Let $dx$ be the self-dual Haar measure on $F$ with respect to $\psi$.
We put $U_P(\oo_F)$ to be the subset of $U_P$ corresponding to $\oo_F^l$ via the above identification
$U_P\cong F^l$.
Then the measure $du_P'$ on $U_P$ is defined by
\[
\vol(U_P(\oo_F),du_P')=|\pm (2\sqrt{d})^{-k}c^{\sum_{i=1}^k(m'-1-i)}|_F\vol(\oo_F,dx)^l.
\]
On the other hand, the measure $du_P$ defined in \S \ref{measure} satisfies
\[
\vol(U_P(\oo_F),du_P)=|(2c\sqrt{d})^{-k}|_F\vol(\oo_F,dx)^l.
\]
Since 
\[
k+\sum_{i=1}^k(m'-1-i)=k\left(m'-\frac{k+1}{2}\right)=k\rho_P,
\]
we have
\[
du_P'=|c|_F^{k+\sum_{i=1}^k(m-1-i)}du_P=|c|_F^{k\rho_P}du_P,
\]
as desired.
\end{proof}
\par

We define unnormalized intertwining operators
\begin{align*}
\MM(\cl{w},\tau_s\otimes \cl{\sigma})&\colon
\Ind_{P}^{\orth(V')}(\tau_s\otimes \cl{\sigma})\rightarrow
\Ind_{P}^{\orth(V')}(w(\tau_s\otimes \cl{\sigma})),\\
\MM(\cl{w}',\tau_s\otimes \pi)&\colon
\Ind_{Q}^{\Sp(W')}(\tau_s\otimes \pi)\rightarrow\Ind_{Q}^{\Sp(W')}(w'(\tau_s\otimes \pi))
\end{align*}
by (the meromorphic continuations of) the integrals
\begin{align*}
\MM(\cl{w},\tau_s\otimes \cl{\sigma})\Phi_s(h')&=
\left|c\right|_F^{k\rho_P}
\int_{U_P}\Phi_s(\cl{w}^{-1}u_Ph')du_P,\\
\MM(\cl{w}',\tau_s\otimes \pi)\Phi'_s(g')&=
\int_{U_Q}\Phi'_s(\cl{w}'^{-1}u_Qg')du_Q,
\end{align*}
where $w(\tau_s\otimes \cl{\sigma})$ (\resp $w'(\tau_s\otimes \pi)$) 
is the representation of $M_P$ on $\VV_\tau\otimes\VV_{\sigma}$ 
(\resp $M_Q$ on $\VV_\tau\otimes\VV_{\pi}$)
given by 
$w(\tau_s\otimes \cl{\sigma})(m_P)=(\tau_s\otimes \cl{\sigma})(\cl{w}^{-1}m_P\cl{w})$ 
for $m_P\in M_P$
(\resp $w'(\tau_s\otimes \pi)(m_Q)=(\tau_s\otimes \pi)(\cl{w}'^{-1}m_Q\cl{w}')$ 
for $m_Q\in M_Q$).
Since $\sigma$ is $\epinv$, by Lemma \ref{well-defined} (\ref{well3}), 
the operator $\MM(\cl{w},\tau_s\otimes \cl{\sigma})$ gives
an operator 
\[
\MM(\cl{w},\tau_s\otimes \sigma)\colon
\Ind_{P^\circ}^{\SO(V')}(\tau_s\otimes \sigma)\rightarrow
\Ind_{P^\circ}^{\SO(V')}(w(\tau_s\otimes \sigma)).
\]
\par

Following \cite[\S 2.3]{A},
we use the normalizing factors $r(w,\tau_s\otimes \sigma)$ 
and $r(w',\tau_s\otimes \pi)$ defined as follows.
Let $\phi_\tau$, $\phi_{\sigma}$ and $\phi_{\pi}$ be the representations of $\WD_F$
corresponding to $\tau$, $[\sigma]$ and $\pi$, respectively.
Then we define $r(w,\tau_s\otimes \sigma)$ and $r(w',\tau_s\otimes \pi)$ by
\begin{align*}
&r(w,\tau_s\otimes \sigma)=\lam(E/F,\psi)^k
\frac{L(s,\phi_\tau\otimes \phi_{\sigma}^\vee)}
{\ep(s,\phi_\tau\otimes \phi_{\sigma}^\vee,\psi)L(1+s,\phi_\tau\otimes \phi_{\sigma}^\vee)}
\frac{L(2s, \wedge^2 \circ \phi_\tau)}
{\ep(2s, \wedge^2 \circ \phi_\tau,\psi)L(1+2s, \wedge^2 \circ \phi_\tau)},\\
&r(w',\tau_s\otimes \pi)
=\frac{L(s,\phi_\tau\otimes \phi_{\pi}^\vee)}
{\ep(s,\phi_\tau\otimes \phi_{\pi}^\vee,\psi)L(1+s,\phi_\tau\otimes \phi_{\pi}^\vee)}
\frac{L(2s, \wedge^2 \circ \phi_\tau)}
{\ep(2s, \wedge^2 \circ \phi_\tau,\psi)L(1+2s, \wedge^2 \circ \phi_\tau)},
\end{align*}
where $\wedge^2$ is the representation of $\GL_k(\C)$ 
on the space of skew-symmetric $(k,k)$-matrices,
and $\lam(E/F,\psi)$ is the Langlands $\lam$-factor associated to $E=F(\sqrt{d})$.
Note that $\lam(E/F,\psi)^2=\omega_{E/F}(-1)=\chi_V(-1)$.
Then the normalized intertwining operators
\begin{align*}
\RR(w,\tau_s\otimes \sigma)&\coloneqq
r(w,\tau_s\otimes \sigma)^{-1}\MM(\cl{w},\tau_s\otimes \sigma),\\
\RR(w',\tau_s\otimes \pi)&\coloneqq
r(w',\tau_s\otimes \pi)^{-1}\MM(\cl{w}',\tau_s\otimes \pi)
\end{align*}
are holomorphic at $s=0$ by \cite[Proposition 2.3.1]{A}.
\par

Now assume that $w(\tau\otimes \sigma)\cong\tau\otimes \sigma$
and $w'(\tau \otimes \pi) \cong \tau \otimes \pi$,
both of which are equivalent to $\tau^\vee\cong\tau$.
We take the unique isomorphism 
\[
\AA_w\colon \VV_\tau\otimes \VV_{\sigma}\rightarrow \VV_\tau\otimes\VV_{\sigma}
\]
such that:
\begin{itemize}
\item
$\AA_w\circ w(\tau\otimes \sigma)(m)=(\tau\otimes \sigma)(m)\circ\AA_w$ 
for any $m\in M_P$;
\item
$\AA_w=\AA'_w\otimes \1_{\VV_{\sigma}}$ with an isomorphism
$\AA'_w\colon \VV_\tau\rightarrow \VV_\tau$ such that
$\Lam\circ\AA'_w=\Lam$.
Here $\Lam\colon\VV_\tau\rightarrow \C$ is the unique (up to a scalar) 
Whittaker functional with respect to the Whittaker datum $(B_k,\psi_{U_k})$, 
where $B_k$ is the Borel subgroup consisting of upper triangular matrices in $\GL_k(F)$
and $\psi_{U_k}$ is the generic character of the unipotent radical $U_k$ of $B_k$
given by
$\psi_{U_k}(x)=\psi(x_{1,2}+\dots+x_{k-1,k})$.
\end{itemize}
Similarly, we take the unique isomorphism
\[
\AA_{w'}\colon \VV_\tau\otimes \VV_{\pi}\rightarrow \VV_\tau\otimes\VV_{\pi}.
\]
\par

We define self-intertwining operators
\begin{align*}
R(w,\tau\otimes\sigma)&\colon 
\Ind_{P^\circ}^{\SO(V')}(\tau_s\otimes \sigma)\rightarrow
\Ind_{P^\circ}^{\SO(V')}(\tau_s\otimes \sigma),\\
R(w',\tau\otimes\pi)&\colon 
\Ind_{Q}^{\Sp(W')}(\tau_s\otimes \pi)\rightarrow\Ind_{Q}^{\Sp(W')}(\tau_s\otimes \pi)
\end{align*}
by
\begin{align*}
R(w,\tau\otimes\sigma)\Phi_s(h')&=\AA_{w}(\RR(w,\tau\otimes\sigma)\Phi_s(h')),\\
R(w',\tau\otimes\pi)\Phi'_s(g')&=\AA_{w'}(\RR(w',\tau\otimes\pi)\Phi'_s(g'))
\end{align*}
for $\Phi_s\in \Ind_{P^\circ}^{\SO(V')}(\tau_s\otimes \sigma)$,
$\Phi'_s\in \Ind_{Q}^{\Sp(W')}(\tau_s\otimes \pi)$, $h'\in\SO(V')$ and $g'\in \Sp(W')$.
\par

Put $\phi_{\sigma'}=\phi_\tau+\phi_{\sigma}+\phi_\tau^\vee\in\cl\Phi_\temp(\SO(V'))$ 
(\resp $\phi_{\pi'}=\phi_\tau+\phi_{\pi}+\phi_\tau^\vee\in\cl\Phi_\temp(\Sp(W'))$).
We take $[\sigma']\in\Pi_{\phi_{\sigma'}}$ (\resp $\pi'\in\Pi_{\phi_{\pi'}}$)
such that $\sigma'\subset\Ind_{P^\circ}^{\SO(V')}(\tau\otimes \sigma)$
(\resp $\pi'\subset\Ind_{Q}^{\Sp(W')}(\tau\otimes \pi)$).
Let $\w_c$ (\resp $\w'_1$) be the Whittaker datum of $\SO(V')$ (\resp $\Sp(W')$) 
defined in \S \ref{Quadratic}.
Note that $\w_c$ (\resp $\w'_1$) is given by the splitting $\spl_{\SO(V')}$ (\resp $\spl_{\Sp(W')}$) 
and the non-trivial additive character $\psi$ we have fixed.
The intertwining operators and the local Langlands correspondence are
related as follows:
\begin{prop}\label{R v.s. S}
Suppose that $\phi_\tau$ is a tempered orthogonal representation of $\WD_F$ 
with even dimension $k$.
Let $\tau$ be the irreducible representation of $\GL_k(F)$ associated to $\phi_\tau$.
\begin{enumerate}
\item
Put $\phi_{\sigma'}=\phi_\tau+\phi_{\sigma}+\phi_\tau$.
Assume that 
\begin{itemize}
\item
$\SO(V)$ is quasi-split and $V$ is type $(d,c)$;
\item
$\dim(V)=2m\geq4$;
\item
$\phi_{\sigma}\in\cl\Phi_\temp(\SO(V))$ is $\epinv$.
\end{itemize}
Let $a \in C_{\phi_{\sigma'}}^+$ which is identity on $\phi_\sigma$ and exchanges 
each of two factors of $\phi_\tau$, so that $\phi_\tau=\phi_{\sigma'}^a$.
Then for $\sigma \in \Pi_{\phi_\sigma}$ and $\sigma' \in \Pi_{\phi_{\sigma'}}$ with
$\sigma'\subset\Ind_{P^\circ}^{\SO(V')}(\tau\otimes \sigma)$, we have
\[
R(w,\tau\otimes\sigma)|\sigma'=\iota_{\w_{c}}([\sigma'])(a).
\]
\item
Put $\phi_{\pi'}=\phi_\tau+\phi_{\pi}+\phi_\tau$.
Let $a' \in C_{\phi_{\pi'}}^+$ which is identity on $\phi_\pi$ and exchanges 
each of two factors of $\phi_\tau$, so that $\phi_\tau=\phi_{\pi'}^{a'}$.
Then for $\pi \in \Pi_{\phi_\pi}$ and $\pi' \in \Pi_{\phi_{\pi'}}$ with
$\pi'\subset\Ind_{Q}^{\Sp(W')}(\tau\otimes \pi)$, we have
\[
R(w',\tau\otimes\pi)|\pi'=\iota_{\w'_1}(\pi')(a').
\]
\end{enumerate}
\end{prop}
\begin{proof}
This follows from Theorem 2.2.1 and Theorem 2.4.1 in \cite{A}.
For the detail, see \cite[Theorem 2.4]{At}.
\end{proof}
\par

If $\SO(V)$ is not quasi-split, then $\SO(V')$ has no $F$-rational splittings.
However, the representative $\cl{w} = w_P \cdot m_P(a) \cdot (-\1_V)^k$ with
\[
a = 
\begin{pmatrix}
&&(-1)^{m'-k+1}c\\
&\iddots&\\
(-1)^{m'}c&&
\end{pmatrix}
\in \GL_k(F) \cong \GL(X)
\]
and
the measure $du_P' = |c|_F^{k\rho_P}du_P$ make sense even if $\SO(V)$ is not quasi-split.
Hence we can define the normalized intertwining operator for non-quasi-split $\SO(V')$
by the same formula.
In this paper, we assume the following: 
\begin{ass}[NQ]\label{NQ}
Proposition $\ref{R v.s. S}$ $(1)$ holds 
for some $c \in F^\times$ even if $\SO(V)$ is not quasi-split.
\end{ass}
We emphasize that $k$ is even.
Note that if Proposition \ref{R v.s. S} (1) holds for some $c \in F^\times$, 
then it holds for any $c \in F^\times$ by Proposition \ref{whittaker2}.

%\subsection{Zeta integrals of Godement--Jacquet}
\subsection{Zeta integrals of Godement--Jacquet}
In this subsection, we review the theory of local factors for $\GL(k)$ 
developed by Godement--Jacquet \cite{GJ}.
\par

Let $\tau$ be an irreducible smooth representation of $\GL_k(F)$ on a space $\VV_\tau$
with a central character $\omega_\tau$.
We write
\[
L(s,\tau)=L(s,\phi_\tau)
\quad\text{and}\quad
\ep(s,\tau,\psi)=\ep(s,\phi_\tau,\psi)
\]
for the standard $L$-factor and $\ep$-factor of $\tau$,
where $\phi_\tau$ is the $k$-dimensional representation of $\WD_F$ associated to $\tau$.
Then the standard $\gamma$-factor of $\tau$ is defined by
\[
\gamma(s,\tau,\psi)=\ep(s,\tau,\psi)\cdot
\frac{L(1-s,\tau^\vee)}{L(s,\tau)}.
\]
\par

For $s\in\C$, $\Phi\in\Sch(\M_k(F))$ and a matrix coefficient $f$ of $\tau$, we put
\[
Z(s,\Phi,f)=\int_{\GL_k(F)}\Phi(a)f(a)|\det(a)|_F^sda.
\]
This integral is absolutely convergent for $\re(s)\gg0$
and admits a meromorphic continuation to $\C$.
Moreover, the quotient
\[
\frac{Z\left(s+\frac{k-1}{2},\Phi,f\right)}{L(s,\tau)}
\]
is an entire function of $s$.
\begin{lem}
If $\tau$ is tempered, then $Z(s,\Phi,f)$
is absolutely convergent for $\re(s)>(k-1)/2$.
\end{lem}
\begin{proof}
Put $t=\re(s)>(k-1)/2$.
Fix a uniformizer $\varpi$ of $F$ and put
\[
t(a)=
\begin{pmatrix}
\varpi^{a_1}&&\\
&\ddots&\\
&&\varpi^{a_k}
\end{pmatrix}\in \GL_k(F)
\]
for $a=(a_1,\dots,a_k)\in \Z^k$.
Let $B$ be the Borel subgroup of $G=\GL_k(F)$ consisting of upper triangular matrices, and
put $K=\GL_k(\oo_F)$.
We denote the modulus character of $B$ by $\delta_B$.
Then we have
\[
\delta_B(t(a))=\prod_{i<j}\left|\frac{\varpi^{a_i}}{\varpi^{a_j}}\right|_F
=q^{-(k-1)a_1}q^{-(k-3)a_2}\cdots q^{(k-1)a_k}.
\]
Since $F$ is non-archimedean, the integral formula
\begin{align*}
&\int_{\GL_k(F)}\Big|\Phi(g)f(g)|\det(g)|_F^s\Big|da\\
&=\sum_{a_1\leq\dots\leq a_k}\mu(t(a))
\int_{K\times K}|\Phi(k_1t(a)k_2)|\cdot|f(k_1t(a)k_2)|\cdot|\det(k_1t(a)k_2)|_F^tdk_1dk_2
\end{align*}
holds for some $\mu(t(a))\geq0$.
Moreover, there exists a positive constant $A_0$ such that
$\mu(t(a))\leq A_0\cdot\delta_B(t(a))$ for $a=(a_1,\dots,a_k)\in\Z^k$ with $a_1\leq\dots\leq a_k$
(see, e.g., \cite[p. 149]{Sil}).
\par

We define the height function $\sigma(g)$ of $G$ by
\[
\sigma(g)=\max_{\substack{1\leq i\leq k\\1\leq j\leq k}}(\log|g_{i,j}|_F,\log|(g^{-1})_{i,j}|_F).
\]
Harish-Chandra's spherical function $\Xi(g)$ of $G$ is given by
\[
\Xi(g)=\int_{K}h_0(kg)dk,
\]
where $h_0\in\Ind_B^G(\1)$ is the function whose restriction to $K$ is identically equal to $1$.
Note that $\Xi$ is a matrix coefficient of the tempered representation $\Ind_B^G(\1)$,
and is bi-$K$-invariant. 
It is known that there exist positive constants $A_1$ and $A_2$ such that
\[
\Xi(t(a))\leq A_1\cdot\delta_B^{-1/2}(t(a))\cdot(1+\sigma(t(a)))^{A_2}
\]
for $a=(a_1,\dots,a_k)\in\Z^k$ with $a_1\leq\dots\leq a_k$
(see, e.g., \cite[p. 154, Theorem 4.2.1]{Sil}).
Since $\tau$ is tempered, any matrix coefficient $f$ of $\tau$ satisfies
\[
|f(g)|\leq B_1\cdot\Xi(g)\cdot(1+\sigma(g))^{B_2}
\]
for some positive constants $B_1$ and $B_2$.
\par

Since $\sigma(k_1gk_2)=\sigma(g)$ for $k_1,k_2\in K$ and $g\in G$,
we conclude that there are positive constants $C_1$ and $C_2$ such that
\begin{align*}
&\int_{\GL_k(F)}\Big|\Phi(g)f(g)|\det(g)|_F^s\Big|da\\
&=C_1\cdot\sum_{a_1\leq\dots\leq a_k}\delta_B^{1/2}(t(a))(1+\sigma(t(a)))^{C_2}|\det(t(a))|_F^t
\int_{K\times K}|\Phi(k_1t(a)k_2)|dk_1dk_2.
\end{align*}
We choose $\ep>0$ such that $(k-1)/2-t+C_2\ep<0$.
Note that $\Phi$ has a compact support in $\M_k(F)$,
and so does the function 
\[
\M_k(F)\ni x\mapsto\int_{K\times K}|\Phi(k_1xk_2)|dk_1dk_2.
\]
This implies that there are constants $M>0$ and $r>0$ such that
\begin{align*}
&\int_{\GL_k(F)}\Big|\Phi(g)f(g)|\det(g)|_F^s\Big|da\\
&=M\cdot\sum_{-r\leq a_1\leq\dots\leq a_k}\delta_B^{1/2}(t(a))(1+\sigma(t(a)))^{C_2}|\det(t(a))|_F^t\\
&=M\cdot\sum_{-r\leq a_1\leq\dots\leq a_k}q^{a_1(-(k-1)/2-t)}q^{a_2(-(k-3)/2-t)}\cdots q^{a_k((k-1)/2-t)}(1+\sigma(t(a)))^{C_2}.
\end{align*}
To see the convergence of this sum, 
we only consider the sum over $a_k\geq r\geq -a_1$.
Then we have
\[
\sigma(t(a))=\log(q^{a_k}).
\]
Moreover, we may assume that if $a_k\geq r$, then $1+\log(q^{a_k})\leq q^{\ep a_k}$.
In this case, the sum over $a_k\geq r\geq -a_1$ is bounded by
\[
M\cdot\sum_{a_k\geq r}(a_k+r+1)^{k-1}\cdot
q^{-r(-(k-1)/2-t)}\cdots q^{-r((k-3)/2-t)} \cdot
q^{a_k((k-1)/2-t+C_2\ep)}.
\]
This sum converges since $q^{(k-1)/2-t+C_2\ep}<1$.
\end{proof}
\par

Let $\hat{\Phi}\in\Sch(\M_k(F))$ be the Fourier transform of $\Phi$ defined by
\[
\hat{\Phi}(x)=\int_{\M_k(F)}\Phi(y)\psi(\tr(xy))dy,
\]
where $dy$ is the self-dual Haar measure on $\M_k(F)$ with respect to the pairing
$(x,y)\mapsto\psi(\tr(xy))$.
Let $\ch{f}$ be the matrix coefficient of $\tau^\vee$ given by
$\ch{f}(a)=f(a^{-1})$.
Then the local functional equation asserts that
\[
Z\left(-s+\frac{k+1}{2},\hat{\Phi},\ch{f}\right)
=\gamma(s,\tau,\psi)\cdot
Z\left(s+\frac{k-1}{2},\Phi,f\right).
\]

%\section{Proof of Theorem \ref{last}}\label{proof of last}
\section{Proof of (P1) for tempered $L$-parameters}\label{proof of last}
Now we can begin the proof of (P1) for tempered $L$-parameters. 
This will be proven by an explicit construction of an
equivariant map which realizes the theta correspondence.
Let $V=V_{2m}$ and $W=W_{2n}$.
In this section, we put $m=n+1$.
Assume that $n \geq 1$ so that $\dim(V)=2m\geq4$.

%\subsection{Construction of equivariant maps}
\subsection{Construction of equivariant maps}
We put
\begin{align*}
V'=X+V+X^*,\quad
W'=Y+W+Y^*
\end{align*}
with $\dim(X)=\dim(Y)=k$.
Assume that $k$ is even.
Using the basis $\{v_1,\dots,v_k\}$ of $X$ (\resp $\{w_1,\dots,w_k\}$ of $Y$), 
we identify $\GL(X)$ (\resp $\GL(Y)$) with $\GL_k(F)$.
Hence we can define an isomorphism $i\colon \GL(Y)\rightarrow \GL(X)$ via these identifications.
Put
\[
e=v_1\otimes w_1^*+\dots+ v_k\otimes w_k^*\in X\otimes Y^*,\quad 
e^*=v_1^*\otimes w_1+\dots+ v_k^*\otimes w_k\in X^*\otimes Y.
\]
Then we have $i(a)e=a^*e$ and $i(a)^*e^*=ae^*$ for $a\in\GL(Y)$.
\par

Recall that $W=Y_n+Y_n^*$.
For $\varphi\in\Sch''=\Sch(V'\otimes Y^*)\otimes \Sch(V\otimes Y_n^*)\otimes \Sch(X^*\otimes W)$, 
we define maps
\[
f(\varphi), \hat{f}(\varphi)\colon \Sp(W')\times \orth(V')\rightarrow \Sch(V\otimes Y_n^*)
\]
by
\begin{align*}
[f(\varphi)(g',h')](x_0)&=[\omega''(g',h')\varphi](
\begin{pmatrix}
e\\0\\0
\end{pmatrix}
,x_0,0
),\\
[\hat{f}(\varphi)(g',h')](x_0)&=
\int_{X\otimes Y^*}
[\omega(g',h')\varphi](
\begin{pmatrix}
z\\0\\0
\end{pmatrix}
,x_0,0
)
\psi(\pair{z,e^*})dz
\end{align*}
for $g' \in \Sp(W')$ and $h' \in \orth(V')$.
Here, we write an element in $V'\otimes Y^*$ as a block matrix
\[
\begin{pmatrix}
x_1\\x_2\\x_3
\end{pmatrix}
\]
with $x_1\in X\otimes Y^*$, $x_2\in V\otimes Y^*$ and $x_3\in X^*\otimes Y^*$.
\begin{lem}
For $f=f(\varphi)$ or $f=\hat{f}(\varphi)$, we have
\begin{align*}
f(u_Qg',u_Ph')&=f(g',h'),&u_P\in U_P, u_Q\in U_Q,\\
f(gg',hh')&=\omega(g,h)f(g',h'),&h\in\orth(V), g\in\Sp(W), \\
f(m_Q(a)g',m_P(i(a))h')&=\chi_V(\det(a))|\det(a)|_F^{\rho_P+\rho_Q}f(g',h'),&a\in\GL(Y).
\end{align*}
\end{lem}
\begin{proof}
This follows from Lemmas \ref{Weil1}, \ref{Weil2} and \ref{Weil3}.
\end{proof}
\par

Let $\tau$ be an irreducible (unitary) tempered representation of $\GL_k(F)$ on a space $\VV_\tau$.
We may regard $\tau$ as a representation of $\GL(X)$ or $\GL(Y)$ via the above identifications.
Let $\pi$ and $\cl{\sigma}$ be irreducible tempered representations of $\Sp(W)$ and $\orth(V)$ 
on spaces $\VV_{\pi}$ and $\VV_{\cl{\sigma}}$, respectively.
Fix nonzero invariant non-degenerate bilinear forms $\pair{\cdot,\cdot}$ on $\VV_\tau\times\VV_{\tau^\vee}$,
$\VV_{\pi}\times\VV_{\pi^\vee}$ and $\VV_{\cl{\sigma}}\times\VV_{\cl{\sigma}^\vee}$.
Let 
\[
\pair{\cdot,\cdot}\colon (\VV_\tau\otimes \VV_{\cl{\sigma}^\vee})\times\VV_{\tau^\vee}\rightarrow \VV_{\cl{\sigma}^\vee}
\]
be the induced map.
\par

Now we assume that 
\[
\cl{\sigma}=\Theta_{\psi,V,W}(\pi).
\]
We fix a nonzero $\Sp(W)\times\orth(V)$-equivariant map
\[
\TT\colon\omega\otimes \cl{\sigma}^\vee \rightarrow \pi.
\]
For $\varphi\in\Sch''=\Sch(V'\otimes Y^*)\otimes \Sch(V\otimes Y_n^*)\otimes \Sch(X^*\otimes W)$, 
$\Phi_s\in\Ind_P^{\orth(V')}(\tau_s\otimes\cl{\sigma}^\vee)$,
$g'\in\Sp(W')$, $\ch{v}\in\VV_{\tau^\vee}$, and $\ch{v}_0\in\VV_{\pi^\vee}$, 
we put
\begin{align*}
\pair{\TT_s(\varphi,\Phi_s)(g'), \ch{v}\otimes\ch{v}_0}
=L(s,\tau)^{-1}\cdot
\int_{U_P\orth(V)\bs\orth(V')}\pair{\TT(\hat{f}(\varphi)(g',h'),\pair{\Phi_s(h'),\ch{v}}),\ch{v}_0}dh'.
\end{align*}
Note that $\pair{\Phi_s(h'),\ch{v}}\in\VV_{\sigma^\vee}$.
\begin{prop}\label{equiv}
We have the following:
\begin{enumerate}
\item
The integral $\pair{\TT_s(\varphi,\Phi_s)(g'), \ch{v}\otimes\ch{v}_0}$
is absolutely convergent for $\re(s)>0$ 
and admits a holomorphic continuation to $\C$.
\item
For $\re(s)<1$, we have
\begin{align*}
\pair{\TT_s(\varphi,\Phi_s)(g'),\ch{v}\otimes\ch{v}_0}
=L(s,\tau)^{-1}\gamma(s,\tau,\psi)^{-1}\cdot
\int_{U_P\orth(V)\bs\orth(V')}\pair{\TT(f(\varphi)(g',h'),\pair{\Phi_s(h'),\ch{v}}),\ch{v}_0}dh'.
\end{align*}
\item
The map
\[
\TT_s\colon\omega''\otimes \Ind_{P}^{\orth(V')}(\tau_s\otimes\cl{\sigma}^\vee)\rightarrow \Ind_Q^{\Sp(W')}(\tau_s\chi_V\otimes \pi)
\]
is $\Sp(W')\times\orth(V')$-equivariant.
\item\label{equiv4}
For $\Phi\in\Ind_{P}^{\orth(V')}(\tau\otimes\cl{\sigma}^\vee)$ with $\Phi\ne0$, 
there exists $\varphi\in\omega''$ such that
\[
\TT_0(\varphi,\Phi)\not=0.
\]
\end{enumerate}
\end{prop}
\begin{proof}
The proof is similar to those of Lemmas 8.1, 8.2 and 8.3 in \cite{GI2}.
\end{proof}

%\subsection{Compatibilities with intertwining operators}
\subsection{Compatibilities with intertwining operators}
Now we shall explain a key property of the equivariant map we have constructed.
\par

We have assumed that $k=\dim(X)=\dim(Y)$ is even and $\dim(V)=2(n+1)\geq4$.
Let $w\in W(M_P)$ and $w'\in W(M_Q)$ be the non-trivial elements in the relative Weyl groups.
As in \S\ref{sec.intertwining}, we take the representatives 
$\cl{w}\in\SO(V')$ of $w$ and $\cl{w}'\in\Sp(W')$ of $w'$ defined by
\[
\cl{w}=w_P\cdot m_P(-c\cdot a)
\quad \text{and}\quad 
\cl{w}'=w_Q\cdot m_Q(a),
\]
where $a\in\GL_k(F)\cong\GL(X)\cong\GL(Y)$ is given by
\[
a=
\begin{pmatrix}
&&(-1)^{n+1}\\
&\iddots&\\
(-1)^{n+k}&&
\end{pmatrix}.
\]
We fix $\tau$, $\pi$ and $\cl{\sigma}=\Theta_{\psi,V,W}(\pi)\not=0$.
We shall write
\[
\MM(\cl{w},s)=\MM(\cl{w},\tau_s\otimes\cl{\sigma}^\vee)
\quad\text{and}\quad
\MM(\cl{w}',s)=\MM(\cl{w}',\tau_s\chi_V\otimes\pi)
\]
for the unnormalized intertwining operators, which are defined by the integrals
\begin{align*}
\MM(\cl{w},s)\Phi_s(h')&=
\left|c\right|_F^{k\rho_P}
\int_{U_P}\Phi_s(\cl{w}^{-1}u_Ph')du_P,\\
\MM(\cl{w}',s)\Phi'_s(g')&=
\int_{U_Q}\Phi'_s(\cl{w}'^{-1}u_Qg')du_Q
\end{align*}
for $\Phi_s\in\Ind_P^{\orth(V')}(\tau_s\otimes\cl{\sigma}^\vee)$
and $\Phi'_s\in\Ind_Q^{\Sp(W')}(\tau_s\chi_V\otimes\pi)$, respectively.
By the Howe duality, the diagram
\[
\begin{CD}
\omega''\otimes\Ind_P^{\orth(V')}(\tau_s\otimes\cl{\sigma}^\vee)
@>\TT_s>>\Ind_Q^{\Sp(W')}(\tau_s\chi_V\otimes\pi)\\
@V1\otimes\MM(\cl{w},s)VV @VV\MM(\cl{w}',s)V\\
\omega''\otimes\Ind_P^{\orth(V')}(w(\tau_s\otimes\cl{\sigma}^\vee))
@>\TT_{-s}>>\Ind_Q^{\Sp(W')}(w'(\tau_s\chi_V\otimes\pi))
\end{CD}
\]
commutes up to a scalar.
The following proposition determines this constant of proportionality explicitly.
\begin{prop}\label{proportionality}
For $\varphi\in\omega''$ and $\Phi_s\in\Ind_P^{\orth(V')}(\tau_s\otimes\cl{\sigma}^\vee)$, we have
\begin{align*}
\MM(\cl{w}',s)\TT_s(\varphi,\Phi_s)
&=\omega_\tau(c)\cdot\left|c\right|_F^{ks}\cdot
\gamma_V^{-k}\cdot L(s,\tau)^{-1}L(-s,\tau^\vee)\gamma(-s,\tau^\vee,\psi)
\cdot
\TT_{-s}(\varphi,\MM(\cl{w},s)\Phi_s).
\end{align*}
\end{prop}
\begin{proof}
The proof is similar to that of \cite[Proposition 8.4]{GI2}.
\end{proof}
\par

\begin{cor}\label{compatibility}
For $\varphi\in\omega''$ and $\Phi_s\in\Ind_P^{\orth(V')}(\tau_s\otimes \cl{\sigma}^\vee)$, we have
\[
\RR(w',\tau_s\chi_V\otimes\pi)\TT_s(\varphi,\Phi_s)
=\omega_\tau(c)\cdot\alpha(s)\cdot
\TT_{-s}(\varphi,\RR(w,\tau_s\otimes\cl{\sigma}^\vee)\Phi_s),
\]
where
\begin{align*}
\alpha(s)&=\left|c\right|_F^{ks}\cdot\frac{\ep(-s,\tau^\vee,\psi)}{\ep(s,\tau,\psi)}.
\end{align*}
In particular, if $\tau\cong\tau^\vee$, then $\alpha(0)=1$ and
\[
R(w',\tau\chi_V\otimes\pi)\TT_0(\varphi,\Phi)
=\omega_\tau(c)\cdot
\TT_{0}(\varphi,R(w,\tau\otimes\cl{\sigma}^\vee)\Phi)
\]
for $\Phi\in\Ind_P^{\orth(V')}(\tau\otimes \cl{\sigma}^\vee)$.
\end{cor}
\begin{proof}
Let $\phi_\tau$, $\phi_{\pi}$ and $\phi_{\sigma}$ be the representations of $\WD_F$
corresponding to $\tau$, $\pi$ and $[\sigma]$, respectively.
Here we put $\sigma\coloneqq\cl{\sigma}|\SO(V)$.
Note that $\sigma$ is irreducible.
By Proposition \ref{proportionality}, we have
\begin{align*}
\alpha(s)&=\left|c\right|_F^{ks}\cdot
\frac{r(w,\tau_s\otimes\sigma^\vee)}{r(w',\tau_s\chi_V\otimes\pi)}\cdot
\gamma_V^{-k}\cdot
\frac{L(-s,\tau^\vee)}{L(s,\tau)}
\cdot\gamma(-s,\tau^\vee,\psi).
\end{align*}
Since $\phi_{\sigma}=(\phi_{\pi}\otimes\chi_V)\oplus\1$,
we have
\[
\phi_\tau\otimes\phi_{\sigma}^\vee=
(\phi_\tau\otimes\chi_V)\otimes\phi_{\pi}^\vee
\oplus\phi_\tau.
\]
Moreover, we have
\[
\wedge^2 \circ(\phi_\tau\otimes\chi_V)
=\wedge^2 \circ\phi_\tau.
\]
Hence we have
\begin{align*}
\frac{r(w,\tau_s\otimes\sigma^\vee)}{r(w',\tau_s\chi_V\otimes\pi)}
=\lam(E/F,\psi)^k\cdot
\frac{L(s,\phi_\tau)}{\ep(s,\phi_\tau,\psi)L(1+s,\phi_\tau)}.
\end{align*}
Note that $\lam(E/F,\psi)^k\gamma_V^{-k}=1$
since $k$ is even and $\lam(E/F,\psi)^2=\gamma_V^2=\chi_V(-1)$.
Recall that
\[
\gamma(-s,\tau^\vee,\psi)=\frac{\ep(-s,\tau^\vee,\psi)L(1+s,\tau)}{L(-s,\tau^\vee)}.
\]
Therefore, we have
\begin{align*}
\alpha(s)&=\left|c\right|_F^{ks}\cdot
\frac{\ep(-s,\tau^\vee,\psi)}{\ep(s,\tau,\psi)}.
\end{align*}
\par

Assume that $\tau\cong\tau^\vee$.
By definition of $\cl{w}$ and $\cl{w}'$, the action of $\cl{w}$ on $\GL(X)\subset M_P$ 
coincides with that of $\cl{w}'$ on $\GL(Y)\subset M_Q$
via the identification $i\colon\GL(Y)\rightarrow \GL(X)$.
This implies that $\AA_w'=\AA_{w'}'$ as $\C$-isomorphisms on $\VV_\tau$.
Hence the last equation holds.
\end{proof}

%\subsection{Completion of the proof}
\subsection{Completion of the proof}
We set $W=W_{2n}$ and $V=V_{2n+2}$.
We assume that $V$ is type $(d,c)$ if $\SO(V)$ is quasi-split.
Fix $c_0\in F^\times$.
Let $\phi_{\pi}\in\cl\Phi_\temp(\Sp(W))$. 
For an arbitrary semi-simple element $a \in C_{\phi_\pi}^+$, we put 
\[
\phi_{\pi'} = \phi_{\pi}^a \oplus \phi_{\pi} \oplus \phi_{\pi}^a.
\]
Note that the canonical injection $A_{\phi_\pi}^+ \hookrightarrow A_{\phi_{\pi'}}^+$ is bijective.
Write $\phi_\pi^a = \phi_\tau \otimes \chi_V$.
Then $\phi_\tau$ is a tempered orthogonal representation of $\WD_F$.
The dimension $k$ of $\phi_\tau$ is even since $a \in C_{\phi_\pi}^+$.
We denote by $\tau\in\Irr_\temp(\GL_{k}(F))$ the representation corresponding to $\phi_\tau$.
Let $\pi\in\Pi_{\phi_\pi}$ be an irreducible tempered representation of $\Sp(W)$
with the associated character $\eta_{\pi}=\iota_{\w_{c_0}'}(\pi)\in\widehat{A_{\phi_\pi}^+}$.
Then $\pi' \coloneqq \Ind_Q^{\Sp(W')}(\tau \chi_V \otimes \pi)$ is irreducible
since $A_{\phi_{\pi'}}^+ = A_{\phi_{\pi}}^+$ (see \S \ref{Cn}).
Moreover, $\pi' \in \Pi_{\phi_{\pi'}}$ and 
the associated character $\eta_{\pi'}=\iota_{\w_{c_0}'}(\pi')\in\widehat{A_{\phi_{\pi'}}^+}$
satisfies that 
\[
\eta_{\pi'}|A_{\phi_\pi}^+ = \eta_{\pi}.
\]
\par

Now, we assume that
\[
\cl{\sigma}=\Theta_{\psi,V,W}(\pi)\not=0.
\]
If we put $\sigma \coloneqq \cl\sigma | \SO(V)$, then
$\sigma$ is irreducible and $[\sigma]\in\Pi_{\phi_{\sigma}}$ with 
\[
\phi_{\sigma}=(\phi_{\pi}\otimes\chi_V)\oplus\1\in\cl\Phi_\temp(\SO(V)).
\]
Put
\[
\phi_{\sigma'}=\phi_\tau \oplus \phi_{\sigma} \oplus \phi_\tau
=(\phi_{\pi'}\otimes\chi_V)\oplus\1\in\cl\Phi_\temp(\SO(V')).
\] 
Note that the canonical injection $A_{\phi_{\sigma}}^+ \hookrightarrow A_{\phi_{\sigma'}}^+$ 
is bijective.
Hence 
the induced representation $\sigma' = \Ind_{P^\circ}^{\SO(V')}(\tau \otimes \sigma)$ is irreducible, 
and so is $\cl{\sigma}' = \Ind_P^{\orth(V')}(\tau\otimes\cl{\sigma})$ (see \S \ref{Dn}).
Moreover we have $[\sigma']\in\Pi_{\phi_{\sigma'}}$, 
and the associated characters 
$\eta_{\sigma}=\iota_{\w_{c_0}}([\sigma])\in\widehat{A_{\phi_{\sigma}}^+}$
and $\eta_{\sigma'}=\iota_{\w_{c_0}}([\sigma'])\in\widehat{A_{\phi_{\sigma'}}^+}$ satisfy that
\[
\eta_{\sigma'}|A_{\phi_{\sigma}}^+=\eta_{\sigma}.
\]
We need to show that $\eta_{\sigma}(a')=\eta_{\pi}(a)$, where $a'$  
is the image of $a$ via the canonical map 
$C_{\phi_\pi}^+ \twoheadrightarrow A_{\phi_\pi}^+ \hookrightarrow A_{\phi_\sigma}^+$.
This equation is equivalent to $\eta_{\sigma'}(a')=\eta_{\pi'}(a)$ since 
the diagram
\[
\begin{CD}
A_{\phi_{\pi'}}^+ @>>> A_{\phi_{\sigma'}}^+\\
@A{1:1}AA @AA{1:1}A\\
A_{\phi_{\pi}}^+ @>>> A_{\phi_{\sigma}}^+
\end{CD}
\]
is commutative.
First, we have the following:
\begin{lem}\label{iota}
We have
\[
\iota_{\w_1'}(\pi')(a)=\omega_\tau(c)\cdot \iota_{\w_{c}}([\sigma'])(a').
\]
\end{lem}
\begin{proof}
Choose $\Phi\in\cl\sigma'^\vee=\Ind_P^{\orth(V')}(\tau\otimes\cl{\sigma}^\vee)$ with
$\Phi\not=0$.
By Proposition \ref{equiv} (\ref{equiv4}), there is $\varphi\in\omega''$ such that
$\TT_0(\varphi,\Phi)\not=0$.
Note that $\TT_0(\varphi,\Phi) \in \Ind_Q^{\Sp(W')}(\tau \chi_V \otimes \pi) = \pi'$.
By Corollary \ref{compatibility}, we have
\[
R(w',\tau\chi_V\otimes\pi)\TT_0(\varphi,\Phi)
=\omega_\tau(c)\cdot\TT_{0}(\varphi,R(w,\tau\otimes\cl{\sigma}^\vee)\Phi).
\]
Note that $\cl\sigma^\vee\cong\cl\sigma$ and $\cl\sigma'^\vee\cong\cl\sigma'$, and 
we may regard $\Phi$ as an element in $\sigma'=\cl\sigma'|\SO(V)$.
Since $R(w',\tau\chi_V\otimes\pi)|\pi'=\iota_{\w_1'}(\pi')(a)$
and $R(w,\tau\otimes\sigma)|\sigma'=\iota_{\w_c}([\sigma'])(a')$
by Proposition \ref{R v.s. S}, 
we have
\[
\iota_{\w_1'}(\pi')(a)\TT_0(\varphi,\Phi)
=\omega_\tau(c)\cdot\TT_{0}(\varphi,\iota_{\w_c}([\sigma'])(a')\Phi).
\]
This shows the desired equation.
\end{proof}
\par

However, by Propositions \ref{whittaker1} and \ref{whittaker2}, we know that
\[
\frac{\iota_{\w'_1}(\pi')(a)}{\iota_{\w'_{c_0}}(\pi')(a)}
=\eta_{c_0^{-1}}(a)=\det(\phi_\tau\otimes\chi_V)(c_0^{-1})
=\omega_\tau(c_0^{-1})
\]
and
\[
\frac{\iota_{\w_c}([\sigma'])(a')}{\iota_{\w_{c_0}}([\sigma'])(a')}
=\eta_{c/c_0}(a')=\det(\phi_\tau)(c/c_0)
=\omega_\tau(c/c_0).
\]
Since $\omega_\tau^2=\1$,
we have
\[
\eta_{\pi'}(a)=\omega_{\tau}(c_0)\cdot\iota_{\w'_{1}}(\pi')(a)
=\omega_{\tau}(c_0c)\cdot \iota_{\w_{c}}([\sigma'])(a')=\eta_{\sigma'}(a'),
\]
as desired.
Since $a \in C_{\phi_\pi}^+$ is arbitrary, we have $\eta_\sigma | A_{\phi_\pi}^+ = \eta_\pi$.
This completes the proof of (P1) for tempered $L$-parameters.

%\appendix
%\section{Transfer factors}\label{app}
\appendix
\section{Transfer factors}\label{app}
In this appendix, we recall the transfer factors and
explain the expectations in the local Langlands correspondence.
More precisely, see \cite{Ka2}, \cite{KS} and \cite{LS}.
Comparing the transfer factors of pure inner twists, 
we prove Propositions \ref{whittaker1} and \ref{whittaker2}.
\par

For simplicity, we let
$F$ be a non-archimedean local field of characteristic zero.

%\subsection{Endoscopic data}
\subsection{Endoscopic data}
Let $G$ be a quasi-split connected reductive group over $F$.
An endoscopic datum for $G$ is a tuple $(H, \H, s, \eta)$, 
where 
\begin{itemize}
\item
$H$ is a quasi-split connected reductive group defined over $F$;
\item
$\H$ is a split extension of $W_F$ by $\widehat{H}$;
\item
$s\in \widehat{G}$ is a semi-simple element;
\item
$\eta\colon \H\hookrightarrow \Lgp{G}$ is an $L$-embedding
\end{itemize}
such that
\begin{enumerate}
\item
the homomorphism $W_F\rightarrow \text{Out}(\widehat{H})$ given by $\H$
is identified with the homomorphism $W_F\rightarrow \text{Out}(H)$
provided by the rational structure of $H$ via the natural isomorphism 
$\text{Out}(\widehat{H})\cong\text{Out(H)}$;
\item
$\eta(\widehat{H})=\widehat{G}_s^\circ=\cent(s,\widehat{G})^\circ$;
\item
a certain condition of $s$ and $\eta$.
\end{enumerate}
An isomorphism from $(H,\H,s,\eta)$ to another such tuple $(H',\H',s',\eta')$
is an element $g\in \widehat{G}$ such that
\begin{enumerate}
\item
$g\eta(\H)g^{-1} = \eta'(\H')$;
\item
$gsg^{-1}\equiv s' \bmod Z(\widehat{G})$.
\end{enumerate}
Given an endoscopic datum $(H, \H, s, \eta)$ of $G$, 
we may replace it by an isomorphic one and assume that $s\in \eta(Z(\widehat{H})^\Gamma)$. 
\par

Let $\phi\colon \WD_F\rightarrow \Lgp{G}$ be an $L$-parameter of $G$.
We put $S_\phi=\text{Cent}(\text{Im}(\phi),\widehat{G})$.
For a semi-simple element $s\in S_\phi$, 
there exists an endoscopic datum $(H,\H,s,\eta)$ for $G$
such that
\begin{itemize}
\item
$\widehat{H}=\widehat{G}_s^\circ$;
\item
$\H=\widehat{H}\cdot\phi(W_F)$;
\item
$\eta\colon \H\hookrightarrow \Lgp{G}$ is the natural embedding.
\end{itemize}
Then we call $(H,\H,s,\eta)$
the endoscopic datum associated to $s\in S_\phi$.
Note that $s\in Z(\widehat{H})^\Gamma$ and
$\im(\phi)\subset \H$.
\par

It is not always true that $\H=\Lgp{H}$.
So we take a z-pair $\z=(H_\z,\eta_{\z})$ for $\e=(H,\H,s,\eta)$.
We recall that
$H_\z$ is an extension of $H$ by an induced torus and
$\eta_\z\colon \H\hookrightarrow \Lgp{H}_\z$
is an $L$-embedding such that the diagram
\[
\begin{CD}
\H@>\eta_\z>>\Lgp{H}_\z\\
@AAA@AAA\\
\widehat{H}@>>>\widehat{H}_\z
\end{CD}
\]
is commutative,
where the bottom arrow is the embedding $\widehat{H}\hookrightarrow \widehat{H}_\z$, 
which is the dual to the surjection $H_\z\rightarrow H$.
If $\e$ is the endoscopic datum associated to $s\in S_\phi$
for some tempered $L$-parameter $\phi$ of $G$,
then we obtain a tempered $L$-parameter 
$\phi_\z=\eta_\z\circ \phi$ of $H_\z$.

%\subsection{Transfer factors of pure inner twists}\label{A.2}
\subsection{Transfer factors of pure inner twists}\label{A.2}
Recall that a pure inner twist of $G$ is a triple
$(G',\psi,z)$, where
\begin{itemize}
\item
$G'$ is a connected reductive algebraic group over $F$;
\item
$\psi\colon G\rightarrow G'$ is an isomorphism over $\overline{F}$;
\item
$z\in Z^1(F,G)$ is a $1$-cocycle
\end{itemize}
such that $\psi^{-1}\circ \sigma(\psi)=\Int(z_\sigma)$ for $\sigma\in\Gamma=\Gal(\overline{F}/F)$.
Then we call $G'$ a pure inner form of $G$.
\par

Let $(G_1,\psi_1,z_1)$ and $(G_2,\psi_2,z_2)$
be two pure inner twists of $G$.
An isomorphism from $(G_1,\psi_1,z_1)$ to $(G_2,\psi_2,z_2)$ is 
a pair $(f,g)$, where
$f\colon G_1\rightarrow G_2$ is an isomorphism over $F$ and
$g\in G(\overline{F})$ such that
$z_{2,\sigma}=g z_{1,\sigma} \sigma(g^{-1})$ and the diagram
\[
\begin{CD}
 G@>\psi_1>>G_1\\
 @V\Int(g)VV   @VVfV\\
 G@>\psi_2>> G_2
 \end{CD} 
 \]
is commutative.
It is known that there exists a canonical bijection
\begin{align*}
\{\text{isomorphism classes of pure inner twists of $G$}\}
&\rightarrow H^1(F,G),\\ 
(G',\psi,z)&\mapsto [z].
\end{align*}
\par

Let $\e=(H,\H,s,\eta)$ be an endoscopic datum for $G$
and $\z=(H_\z,\eta_\z)$ be a z-pair for $\e$.
We may assume that $\eta^{-1}(s)\in Z(\widehat{H})^\Gamma$, and identify $s$ with $\eta^{-1}(s)$.
Then for an inner twist $(G',\psi)$ of $G$, 
these data give a relative transfer factor of $G'$,
which is a function
\[
\Delta[\e,\z,\psi]\colon 
H_{\z,G\text{-}\sr}(F)\times G'_\sr(F)\times H_{\z,G\text{-}\sr}(F)\times G'_\sr(F)
\rightarrow \C.
\]
See \cite[\S 3.7]{LS}.
We also write $\Delta[\e,\z]=\Delta[\e,\z,\id]$.
Here, $G'_\sr$ is the set of strongly regular semi-simple elements of $G'$.
We explain $H_{\z,G\text{-}\sr}$.
Fix pairs $(B,T)$ of $G$, $(\BB,\TT)$ of $\widehat{G}$,
$(B_H,T_H)$ of $H$ and $(\BB_H,\TT_H)$ of $\widehat{H}$,
where by a pair we mean a tuple of a Borel subgroup and a maximal torus contained in it.
Assume that $\eta(\BB_H)\subset \BB$ and $\eta(\TT_H)=\TT$.
Then we have an isomorphism
\[
\widehat{T}_H\cong \TT_H\xrightarrow{\eta}\TT\cong\widehat{T},
\]
and so that we get an isomorphism 
$\eta^*\colon T_H\rightarrow T$ over $\overline{F}$.
Now we assume that $T_H$ is $F$-rational.
Since $\sigma(\eta^*)\circ(\eta^*)^{-1}$ is given by an element in the Weyl group of $T$ in $G$,
by Steinberg's theorem, 
there exists $g\in G_{\rm sc}(\overline{F})$ such that
$\sigma(g)g^{-1}$ normalizes $T$ and induces $\sigma(\eta^*)\circ(\eta^*)^{-1}$.
Here, $G_{\rm sc}$ is the simply connected cover of the derived group of $G$.
Then 
\[
\xi\coloneqq \Int(g^{-1})\circ\eta^* \colon T_H\rightarrow g^{-1}Tg
\]
is an isomorphism over $F$.
Such an $F$-isomorphism $\xi\colon T_H\rightarrow T'\coloneqq g^{-1}Tg$ 
is called an admissible embedding of $T_H$ in $G$.
For $h\in T_H$, the element $\xi(h)$ is called an image of $h$,
and the elements $h$ and $\xi(h)$ are said to be related.
The set $H_{\z,G\text{-}\sr}$ consists of the preimages in $H_\z$ of those elements of $H$
that are related to elements of $G_\sr$.
\par

An absolute transfer factor of $G$ is a function
$\Delta[\e,\z]_{\rm abs}\colon H_{\z,G\text{-}\sr}\times G_\sr\rightarrow \C$ 
such that
\begin{itemize}
\item
$\Delta[\e,\z]_{\rm abs}(\gamma_\z,\delta) \not=0$ if and only if
$\gamma_\z$ and $\delta$ are related;
\item
for any two pairs $(\gamma_\z,\delta)$ and $(\gamma'_\z,\delta')$ 
of related elements, 
\[
\Delta[\e,\z](\gamma_\z,\delta,\gamma'_\z,\delta')=
\frac{\Delta[\e,\z]_{\rm abs}(\gamma_\z,\delta)}
{\Delta[\e,\z]_{\rm abs}(\gamma'_\z,\delta')}.
\]
\end{itemize}
This is not unique.
By choosing a Whittaker datum $\w$ for $G$,
we obtain a normalization
$\Delta[\e,\z,\w]$
of the absolute transfer factor of $G$.
See \cite[\S 5.3]{KS}.
\par

Using $\Delta[\e,\z,\w]$, 
we define a transfer factor $\Delta[\e,\z,\psi,z,\w]$
of a pure inner twist $(G',\psi,z)$ of $G$ as follows.
Let $\delta'\in G'_\sr(F)$ and $\gamma_\z\in H_{\z}(F)$ be related elements.
We denote by $\gamma\in H(F)$ the image of $\gamma_\z$ under the map 
$H_\z\rightarrow H$.
By \cite[Corollary 2.2]{Kot1},
there exists $\delta\in G_\sr(F)$ such that
$\psi^{-1}(\delta')$ and $\delta$ are $G(\overline{F})$-conjugate.
We put
\[
T=G_\delta=\cent(\delta,G)
\quad\text{and}\quad
S=H_\gamma=\cent(\gamma,H).
\]
We take $g\in G(\overline{F})$ such that
\[
\psi^{-1}(\delta')=g\delta g^{-1}.
\]
It is easily seen that
\[
[\sigma\mapsto g^{-1}z_\sigma \sigma(g)]\in Z^1(F,T)
\]
and the class of this element in $H^1(F,T)$
is independent of the choice of $g$.
We denote this class by $\text{inv}(\delta,\delta')$.
Since $\Int(g^{-1}z_\sigma \sigma(g))$ is identity on $T$, 
we see that $\psi\circ\Int(g)$ gives an $F$-isomorphism 
\[
T=\cent(\delta,G)\xrightarrow{\cong} T'=\cent(\delta',G').
\]
This implies that $\gamma_\z$ and $\delta$ are related.
Hence, there exists a unique admissible embedding
\[
\phi_{\gamma,\delta}\colon
S\xrightarrow{\cong} T\subset G
\]
such that $\phi_{\gamma,\delta}(\gamma)=\delta$.
We denote by $s_{\gamma,\delta}\in \widehat{T}^\Gamma$ the image of $s$
under the map
\[
Z(\widehat{H})^\Gamma\hookrightarrow {\widehat{S}}^\Gamma
\rightarrow \widehat{T}^\Gamma,
\]
where the last map is induced by $\phi_{\gamma,\delta}^{-1}$.
Let
\[
\pair{\cdot ,\cdot }\colon H^1(F,T)\times\pi_0(\widehat{T}^\Gamma)\rightarrow \C^\times
\]
be the Tate--Nakayama pairing.
Then we define $\Delta[\e,\z,\psi,z,\w]$ by
\[
\Delta[\e,\z,\psi,z,\w](\gamma_\z,\delta')
=\Delta[\e,\z,\w](\gamma_\z,\delta)\cdot 
\pair{\inv(\delta,\delta'),s_{\gamma,\delta}}^{-1}.
\]
By \cite[Proposition 5.6]{Ka2}, this value does not depend on the choice of $\delta$,
and the function $\Delta[\e,\z,\psi,z,\w]$ is an absolute transfer factor of $G'$.

%\subsection{Expectations in the local Langlands correspondence}\label{expectation}
\subsection{Expectations in the local Langlands correspondence}\label{expectation}
Let $G$ be a quasi-split connected reductive algebraic group over $F$.
The Vogan $L$-packets treat representations of all pure inner forms of $G$ simultaneously.
A representation of a pure inner twist of $G$ is
a tuple $(G',\psi,z,\pi')$,
where $(G',\psi,z)$ is a pure inner twist of $G$ and
$\pi'$ is an admissible representation of $G'(F)$.
Two representations $(G_1,\psi_1,z_1,\pi_1)$ and
$(G_2,\psi_2,z_2,\pi_2)$ are isomorphic if
there exists an isomorphism 
$(f,g)\colon (G_1,\psi_1,z_1)\rightarrow (G_2,\psi_2,z_2)$
of pure inner twists
such that $\pi_2\circ f$ and $\pi_1$ are isomorphic.
By \cite[\S 5.1]{Ka2}, 
two representations $(G_1,\psi_1,z_1,\pi_1)$ and $(G_1,\psi_1,z_1,\pi_1')$ 
of the same pure inner twist are isomorphic if and only if 
$\pi_1$ and $\pi_1'$ are isomorphic in the usual sense as representations of $G_1(F)$.
We denote by $\Pi_\temp(G)$ 
the set of isomorphism classes of irreducible admissible tempered representations
of pure inner twists of $G$.
Let $(f,g)\colon (G_1,\psi_1,z_1,\pi_1)\rightarrow (G_2,\psi_2,z_2,\pi_2)$
be an isomorphism and
$\Theta_{\pi_i}$ be the Harish-Chandra character of $\pi_i$ for $i=1,2$.
Then $f$ transports $\Theta_{\pi_2}$ to $\Theta_{\pi_1}$.
Hence, for an isomorphism class $\dot{\pi}_1=(G_1,\psi_1,z_1,\pi_1)$, 
we obtain a distribution $\Theta_{\dot{\pi}}$.
\par

Fix a Whittaker datum $\w$ for $G$.
Given a tempered $L$-parameter $\phi\colon \WD_F\rightarrow \Lgp{G}$,
we put $S_\phi=\cent(\im(\phi),\widehat{G})$.
We expect that there exist a finite subset $\Pi_\phi\subset \Pi_\temp(G)$
and a bijection $\iota_\w\colon\Pi_\phi\rightarrow \Irr(\pi_0(S_\phi))$
such that the diagram
\[
\begin{CD}
\Pi_\phi@>\iota_\w>>\Irr(\pi_0(S_\phi))\\
@VVV@VVV\\
H^1(F,G)@>>>\pi_0(Z(\widehat{G})^\Gamma)^D
\end{CD}
\]
is commutative,
where the bottom arrow is the Kottwitz map, 
the right arrow sends each irreducible representation to (the restriction of) its central character, 
and the left arrow is given by $(G',\psi,z,\pi')\mapsto [z]$.
Also, we expect that $\Pi_\phi$ contains a unique element
$\dot{\pi}=(G,\id,1,\pi)$ such that $\pi$ is $\w$-generic
and that $\iota_\w(\dot{\pi})$ is the trivial character of $\pi_0(S_\phi)$.
\par

Given $\dot{\pi}=(G',\psi,z,\pi')\in \Pi_\phi$, 
we write $\pair{s,\dot{\pi}}_\w=\tr(\iota_\w(\dot{\pi})(s))$ for $s\in \pi_0(S_\phi)$.
We expect that for a fixed pure inner twist 
$(\psi,z)\colon G\rightarrow G'$, the virtual character
\[
S\Theta_{\phi,\psi,z}=e(G')\sum_{\substack{\dot{\pi}\in\Pi_\phi \\ \dot{\pi}\mapsto [z]}}
\pair{1,\dot{\pi}}_\w\cdot\Theta_{\dot{\pi}}
\]
is a stable function on $G'(F)$ and is independent of $\w$, 
where $e(G')\in\{\pm1\}$ is the sign defined in \cite{Kot2}.
It is known that $e(G')=1$ if $G'$ is quasi-split.
For any semi-simple element $s\in S_\phi$, we put
\[
\Theta^s_{\phi,\w,\psi,z}=e(G')\sum_{\substack{\dot{\pi}\in\Pi_\phi \\ \dot{\pi}\mapsto [z]}}
\pair{s,\dot{\pi}}_\w\cdot\Theta_{\dot{\pi}}.
\]
\par

Let $\e=(H,\H,s,\eta)$ be the endoscopic datum associated to $s\in S_\phi$.
For simplicity, we assume that $\H=\Lgp{H}$.
Then we may take the z-pair $\z=(H_\z,\eta_\z)=(H,\id)$ for $\e$.
Let $\phi_\z$ be the $L$-parameter for $H_\z=H$ given by $\phi$.
Let $f^\e$ and $f$ be
smooth compactly supported functions  on $H(F)$ and $G'(F)$,
respectively.
For $\gamma\in H_\sr(F)$ and $\delta'\in G'_\sr(F)$, we put
\begin{align*}
O_{\gamma}(f^\e)&=\int_{H_\gamma(F)\bs H(F)}f^\e(h^{-1}\gamma h)dh,\\
O_{\delta'}(f)&=\int_{G'_{\delta'}(F)\bs G'(F)}f(g^{-1}\delta' g)dg,
\end{align*}
where we put $H_{\gamma}=\cent(\gamma,H)$ and $G'_{\delta'}=\cent(\delta',G')$, respectively.
We also define
\[
SO_\gamma(f^\e)=\sum_{\gamma'}O_{\gamma'}(f^\e),
\]
where $\gamma'$ runs over a set of representatives for
the $H(F)$-conjugacy classes in the $H(\overline{F})$-conjugacy class of $\gamma$.
More precisely, see \cite[\S 5.5]{KS}.
We say that $f$ and $f^\e$ have $\Delta[\e,\z,\psi,z,\w]$-matching
orbital integrals if
\[
\SO_{\gamma}(f^\e)=\sum_{\delta'} \Delta[\e,\z,\psi,z,\w](\gamma,\delta')\cdot O_{\delta'}(f)
\]
for any $\gamma\in H_{G\text{-}\sr}$.
Here, the sum is taken over a set of representatives for
the $G'(F)$-conjugacy classes of $\delta'\in G'(F)$ such that
$\delta'$ and $\gamma$ are related.
Finally,
if $f^\e$ and $f$ have $\Delta[\e,\z,\psi,z,\w]$-matching orbital integrals, 
then we expect that 
\[
S\Theta_{\phi_\z,\id, 1}(f^\e)=\Theta_{\phi,\w,\psi,z}^s(f).
\]

%\subsection{Changing base points}\label{A.4}
\subsection{Changing base points}\label{A.4}
Let $(\psi,z)\colon G\rightarrow G'$ be a pure inner twist.
In contrast to inner twists,
it may occur that $G'$ is also quasi-split over $F$
even if $(\psi,z)$ is non-trivial.
Then we have the following:
\begin{lem}\label{J}
There exists a bijection
\begin{align*}
J_{\psi,z}\colon\Pi_\temp(G)&\rightarrow \Pi_\temp(G'),\\
(G_1,\psi_1,z_1,\pi_1)&\mapsto (G_1,\psi_1\circ \psi^{-1}, \psi(z_1z^{-1}),\pi_1).
\end{align*}
\end{lem}
\begin{proof}
Easy.
\end{proof}
\par

We denote the center of $G$ by $Z$.
Let $(\psi',z')\colon G\rightarrow G'$ be a pure inner twist
such that both $G$ and $G'$ are quasi-split.
Then the cohomology class $[z']\in H^1(F,G)$ is killed under
the map
\[
H^1(F,G)\rightarrow H^1(F,G_\ad).
\]
Hence there exists $z\in Z^1(F,Z)$ such that $[z]=[z']$ in $H^1(F,G)$.
Then $(\psi',z')\colon G\rightarrow G'$ is isomorphic to 
the pure inner twist
\[
(\id,z)\colon G\rightarrow G.
\]
In particular, we see that
$G$ and $G'$ are isomorphic over $F$,
and so that we may identify $\Lgp{G}$ with $\Lgp{G}'$.
Let $\phi\colon \WD_F\rightarrow \Lgp{G}$ be a tempered $L$-parameter of $G$.
We may regard $\phi$ as an $L$-parameter of $G'$.
We denote this $L$-parameter of $G'$ by $\phi'$, i.e.,
\[
\phi'\colon\WD_F\xrightarrow{\phi}\Lgp{G}=\Lgp{G}'.
\]
Then we should obtain two finite subsets $\Pi_\phi\subset \Pi_\temp(G)$
and  $\Pi_{\phi'}\subset \Pi_\temp(G')$.
\begin{question}
\begin{enumerate}
\item
Does $J_{\psi,z}(\Pi_{\phi})$ coincide with $\Pi_{\phi'}$?
\item
If so, what is the map
\[
\Irr(\pi_0(S_\phi))\xrightarrow{\iota_\w^{-1}}\Pi_\phi
\xrightarrow{J_{\psi,z}}\Pi_{\phi'}
\xrightarrow{\iota_{\w'}}\Irr(\pi_0(S_{\phi'}))
=\Irr(\pi_0(S_\phi))
\] 
for Whittaker data $\w$ and $\w'$ of $G$ and $G'$, respectively?
\end{enumerate}
\end{question}
\par

In this subsection, 
we give an answer of this question when
$z\in Z^1(F,Z)$ and $\psi=\id$.
\begin{prop}\label{Delta}
Fix $z\in Z^1(F,Z)$.
Let $\e=(H,\H,s,\eta)$ be an endoscopic datum,
$\z=(H_\z,\eta_\z)$ be a z-pair for $\e$ and
$\w$ be a Whittaker datum of $G$.
For a pure inner twist $(\psi_1,z_1)\colon G\rightarrow G_1$, 
we put $\Delta_1=\Delta[\e,\z,\psi_1,z_1,\w]$ and
$\Delta_1'=\Delta[\e,\z,\psi_1,z_1z^{-1},\w]$.
Then there exists $\alpha=\alpha[\e,\z,\psi_1,z_1,z,\w]\in\C^\times$ such that
\[
\Delta_1'(\gamma_\z,\delta_1)=\alpha\cdot\Delta_1(\gamma_\z,\delta_1).
\]
for $\gamma_\z\in H_{\z,G\text{-}\sr}$ and $\delta_1\in G_{1,\sr}$.
\end{prop}
\begin{proof}
Since the relative transfer factor $\Delta[\e,\z,\psi_1]$ of $G_1$
is independent of both $z_1$ and $z_1z^{-1}$, 
we have
\[
\frac{\Delta_1'(\gamma_\z,\delta_1)}{\Delta_1'(\gamma_\z',\delta_1')}
=\Delta[\e,\z,\psi_1](\gamma_\z,\delta_1,\gamma'_\z,\delta'_1)
=\frac{\Delta_1(\gamma_\z,\delta_1)}{\Delta_1(\gamma_\z',\delta_1')}
\]
for any two pairs $(\gamma_\z,\delta_1)$ and $(\gamma_\z',\delta_1')$
of related elements.
Hence the quotient
\[
\alpha=\frac{\Delta_1'(\gamma_\z,\delta_1)}{\Delta_1(\gamma_\z,\delta_1)}
\]
does not depend on the choice of a pair $(\gamma_\z,\delta_1)$
of related elements.
This satisfies the desired equations.
\end{proof}
\par

We put $J_z=J_{\id,z}\colon \Pi_\temp(G)\rightarrow\Pi_\temp(G)$.
\begin{cor}
Let $z\in Z^1(F,Z)$. Then
for a tempered $L$-parameter $\phi$ of $G$, we have
\[
J_z(\Pi_\phi)=\Pi_\phi.
\]
\end{cor}
\begin{proof}
We take $\e=(G,\Lgp{G},1,\id)$ and $\z=(G,\id)$.
Let $(G_1,\psi_1,z_1)$ be a pure inner twist of $G$, and
 $\Delta_1$, $\Delta_1'$ and $\alpha$
as in (the proof of) the above proposition.
Then we see that
$f^\e$ and $f$ have $\Delta_1'$-matching orbital integrals
if and only if $f^\e$ and $\alpha\cdot f$ have $\Delta_1$-matching orbital integrals.
Hence we should have
\[
S\Theta_{\phi,\psi_1,z_1z^{-1}}(f)=S\Theta_{\phi,\psi_1,z_1}(\alpha\cdot f).
\]
Since $\{\Theta_{\pi_1}\ |\ \pi_1\in\Irr(G_1(F))\}$ is linearly independent, 
we see that
$(G_1,\psi_1,z_1,\pi_1)\in\Pi_\phi$ if and only if
$(G_1,\psi_1,z_1z^{-1},\pi_1)\in\Pi_\phi$.
\end{proof}
\par

Let $s$ be a semi-simple element in $S_\phi$.
We take the endoscopic datum $\e=(H,\H,s,\eta)$ 
associated to $s$.
Note that $s\in Z(\widehat{H})^\Gamma$.
Let $\z=(H_\z,\eta_\z)$ be a z-pair for $\e$.
For a pair $(\gamma_\z,\delta)\in H_{\z,G\text{-}\sr}\times G_\sr$
of related elements, 
we denote the image of $\gamma_\z$ under the map $H_\z\rightarrow H$
by $\gamma$, and
we put $S=\cent(\gamma,H)$ and $T=\cent(\delta,G)$.
Then there exists a unique admissible embedding 
$\xi\colon S\xrightarrow{\cong} T\subset G$ such that
$\xi(\gamma)=\delta$.
The isomorphism $\xi^{-1}$ gives a map
\[
Z(\widehat{H})^\Gamma\hookrightarrow 
\widehat{S}^\Gamma\rightarrow \widehat{T}^\Gamma.
\]
We denote the image of $s$ under this map by $s_{\gamma,\delta}$.
For $z\in Z^1(F,Z)$, we consider
\[
\pair{[z],s_{\gamma,\delta}},
\]
where $\pair{\cdot,\cdot}\colon H^1(F,T)\times\pi_0(\widehat{T}^\Gamma)\rightarrow \C^\times$
is the Tata--Nakayama pairing, and $[z]$ is the image of $z$ under the map
$Z^1(F,Z)\rightarrow H^1(F,Z)\rightarrow H^1(F,T)$.
\begin{thm}\label{base}
The value $\pair{[z],s_{\gamma,\delta}}$ does not depend on the choice of 
a pair $(\gamma_\z,\delta)\in H_{\z,G\text{-}\sr}\times G_\sr$
of related elements.
The map $s\mapsto \pair{[z],s_{\gamma,\delta}}$ gives a character $\chi_z$ of $\pi_0(S_\phi)$,
and
the diagram
\[
\begin{CD}
\Pi_\phi @>\iota_\w>>\Irr(\pi_0(S_\phi))\\
@VJ_zVV @VV\cdot\otimes\chi_zV\\
\Pi_\phi @>\iota_\w>>\Irr(\pi_0(S_\phi))
\end{CD}
\]
is commutative.
\end{thm}
\begin{proof}
As in Proposition \ref{Delta}, 
for a pure inner twist $(\psi_1,z_1)\colon G\rightarrow G_1$,
we put 
$\Delta_1=\Delta[\e,\z,\psi_1,z_1,\w]$ and
$\Delta_1'=\Delta[\e,\z,\psi_1,z_1z^{-1},\w]$.
Let $(\gamma_\z,\delta_1)\in H_{\z,G\text{-}\sr}\times G_{1,\sr}$
be a pair of related elements.
We take $\delta\in G(F)$ and $g\in G(\overline{F})$ such that
$\psi_1^{-1}(\delta_1)=g\delta g^{-1}$.
We define $\inv(\delta,\delta_1)$ and $\inv'(\delta,\delta_1)$ in $H^1(F,T)$ by
\begin{align*}
\inv(\delta,\delta_1)&=[\sigma\mapsto g^{-1} z_{1,\sigma}\sigma(g)],\\
\inv'(\delta,\delta_1)&=[\sigma\mapsto g^{-1} z_{1,\sigma}z_\sigma^{-1}\sigma(g)].
\end{align*}
Then we have
\begin{align*}
\Delta_1(\gamma_\z,\delta_1)&=\Delta(\gamma_\z,\delta)
\cdot \pair{\inv(\delta,\delta_1),s_{\gamma,\delta}}^{-1},\\
\Delta_1'(\gamma_\z,\delta_1)&=\Delta(\gamma_\z,\delta)
\cdot \pair{\inv'(\delta,\delta_1),s_{\gamma,\delta}}^{-1}.
\end{align*}
Hence we have
\[
\alpha=\frac{\Delta_1'(\gamma_\z,\delta_1)}{\Delta_1(\gamma_\z,\delta_1)}
=\frac{\pair{\inv(\delta,\delta_1),s_{\gamma,\delta}}}
{\pair{\inv'(\delta,\delta_1),s_{\gamma,\delta}}}
=\pair{[z],s_{\gamma,\delta}}.
\]
Since $\alpha$ does not depend on the choice of $(\gamma_\z,\delta_1)$, 
we see that $\pair{[z],s_{\gamma,\delta}}$ 
is independent of the choice of $(\gamma_\z,\delta)$.
We put $\chi_z(s)=\pair{[z],s_{\gamma,\delta}}$.
\par

Since we should have
\[
\Theta_{\phi,\w,\psi,z_1z^{-1}}^s(f)=\Theta_{\phi,\w,\psi,z_1}^s(\chi_z(s)\cdot f)
\]
if $f^\e$ and $f$ have $\Delta_1'$-matching orbital integrals,
we see that
\[
\pair{s,J_z(\dot{\pi}_1)}_\w=\chi_z(s)\cdot\pair{s,\dot{\pi}_1}_\w
\]
for $\dot{\pi}_1=(G_1,\psi_1,z_1,\pi_1)\in \Pi_\phi$.
In particular, taking the unique element $\dot{\pi}_1=(G,\id, 1, \pi)$ such that
$\pi$ is $\w$-generic,
we see that $\chi_z$ is the character of an irreducible representation of $\pi_0(S_\phi)$.
Since $\chi_z(1)=1$, this representation has dimension $1$, i.e.,
$\chi_z\colon \pi_0(S_\phi)\rightarrow \C^\times$ is a ($1$-dimensional) character.
\end{proof}
\par

The character $\chi_z$ is also described as follows.
Let $s\in S_\phi$ be a semi-simple element.
We take the endoscopic datum $(H,\H,s,\eta)$ associated to $s$.
Let $S\subset H$ and $T\subset G$ be maximal $F$-tori
and $\eta^*\colon S\rightarrow T$ be the isomorphism
given by fixed pairs of $G$, $\widehat{G}$, $H$ and $\widehat{H}$.
Note that
$s\in Z(\widehat{H}^\Gamma)\subset \widehat{S}^\Gamma$.
We take $g\in G_{\rm sc}(\overline{F})$ such that
$\xi=\Int(g^{-1})\circ\eta^*\colon S\rightarrow g^{-1}Tg$
is an admissible embedding.
Then 
$Z\subset T$, and $\Int(g^{-1})$ fixes all elements in $Z$.
For $z\in Z^1(F,Z)$, we see that
$\xi^{-1}(z)=(\eta^*)^{-1}(z)\in Z^1(F,\xi^{-1}(Z))\subset Z^1(F,S)$.
Then we have
\[
\chi_z(s)=\pair{[(\eta^*)^{-1}(z)],s},
\]
where $\pair{\cdot,\cdot}\colon H^1(F,S)\times \pi_0(\widehat{S}^\Gamma)\rightarrow \C^\times$
is the Tate--Nakayama pairing.
\par

Next, we consider the case when
$z\in Z^1(F,Z)$ is trivial in $H^1(F,G)$.
Then $[z]\in\im[G_\ad(F)\rightarrow H^1(F,Z)]$.
\begin{prop}\label{Whittaker}
Let $z\in Z^1(F,Z)$ such that
$[z]\in H^1(F,Z)$ is the image of $g\in G_\ad(F)$.
Let $\w=(B,\lam)$ and $\w'=(gBg^{-1},\lam\circ\Int(g^{-1}))$ be two Whittaker data of $G$.
Then we have
\[
\iota_\w\circ J_z=\iota_{\w'}.
\]
In particular, we have $\iota_{\w'}(\dot{\pi})=\iota_\w(\dot{\pi})\otimes\chi_z$
for $\dot{\pi}\in \Pi_\phi$.
\end{prop}
\begin{proof}
There is a similar result in \cite{Ka1}.
According to Theorem 3.3 (or Lemma 3.2) in \cite{Ka1}, 
there exists a character $(\w,\w')$ on $\pi_0(S_\phi/Z(\widehat{G})^\Gamma)$
such that $\iota_{\w}(\dot{\pi})=\iota_{\w'}(\dot{\pi})\otimes(\w,\w')$
for any $\dot{\pi}=(G,\id,1,\pi)\in\Pi_\phi$.
\par

By the definition, 
we can find a preimage $h\in G(\overline{F})$ of $g$ such that
\[
h^{-1}\sigma(h)=z_\sigma.
\]
This implies that $(\Int(g),h)$ is an isomorphism 
from $(\id,1)$ to $(\id,z^{-1})$, 
and so that this is an isomorphism
from $(G,\id,1,\pi\circ\Int(g))$ to $(G,\id,z^{-1},\pi)$.
Therefore, $J_z$ is given by
\[
(G,\id,1,\pi)\mapsto (G,\id,1,\pi\circ\Int(g)).
\]
Note that $\pi$ is $\w'$-generic if and only if
$\pi\circ\Int(g)$ is $\w$-generic.
In particular, if we take the unique element $\dot{\pi}=(G,\id,1,\pi)\in\Pi_\phi$
such that $\pi$ is $\w'$-generic representation,
then we have
\[
\1=\iota_{\w}(J_z(\dot{\pi}))
=\iota_\w(\dot{\pi})\otimes\chi_z
=\iota_{\w'}(\dot{\pi})\otimes(\w,\w')\otimes\chi_z
=(\w,\w')\otimes\chi_z.
\]
Hence we have $\chi_z=(\w,\w')^{-1}$, and so that
\[
\iota_\w\circ J_z(\dot{\pi})=\iota_\w(\dot{\pi})\otimes\chi_z
=\iota_\w(\dot{\pi})\otimes(\w,\w')^{-1}=\iota_{\w'}(\dot{\pi})
\]
for any $\dot{\pi}\in\Pi_\phi$,
as desired.
\end{proof}

%\subsection{Examples}\label{A.5}
\subsection{Examples}\label{A.5}
We calculate $\chi_z$ for $G=\SO(V)$ with $\dim(V)\in 2\Z$, and for $G'=\Sp(W)$.
As an application, we prove Propositions \ref{whittaker1} and \ref{whittaker2}.
We fix a non-trivial additive character $\psi_F\colon F\rightarrow \C^\times$.
\par

First, we prepare a certain property of the Galois cohomology.
For $d\in F^\times$, we put $E_d=F(\sqrt{d})$.
Let 
\[
T=E_d^1\coloneqq\{(a,b)\in\mathbb{G}_a^2\ |\ a^2-b^2d=1\}
\]
be a torus over $F$.
Note that $T\cong \GL_1$ over $E_d$.
For $\sigma\in \Gamma$, 
we denote the usual action on $\GL_1(\overline{F})$ by $x\mapsto\sigma(x)$,
and the action on $T(\overline{F})$ by $x\mapsto \sigma_d(x)$.
If $d\not\in F^{\times2}$, then it is known that $H^1(F,T)\cong\{\pm1\}$ 
(which also follows from Kottwitz's isomorphism).
\begin{lem}\label{coh}
Let $c\in F^\times$.
We define $z_{c,\sigma}\in \overline{F}^\times=T(\overline{F})$ by
\[
z_{c,\sigma}=\frac{\sigma(\sqrt{c})}{\sqrt{c}}=
\left\{
\begin{aligned}
&1	\iif \sigma|E_c=\id_{E_c},\\
&-1	\other.
\end{aligned}
\right.
\]
Then the map $z_c\colon \Gamma\ni\sigma\mapsto z_{c,\sigma}\in T(\overline{F})$ 
belongs to $Z^1(F,T)$.
Moreover, the map $H^1(F,T)\hookrightarrow\{\pm1\}$ is given by
\[
[z_c]\mapsto(c,d),
\]
where $(\cdot,\cdot)$ is the quadratic Hilbert symbol of $F$.
\end{lem}
\begin{proof}
The first assertion is clear.
To prove the last assertion, 
we may assume that $d\not\in F^{\times2}$.
We define $z'_{c,\sigma}\in\overline{F}^\times=T(\overline{F})$ by
\[
z_{c,\sigma}'=
\left\{
\begin{aligned}
&1	\iif \sigma|E_d=\id_{E_d},\\
&c	\other.
\end{aligned}
\right.
\]
It is easily seen that $z_c'\in Z^1(F,T)$.
Moreover we have
\[
z_{c,\sigma} z_{c,\sigma}'^{-1}=\sigma_d(\sqrt{c})\cdot\sqrt{c}^{-1}.
\]
Hence we have $[z_c]=[z_c']$ in $H^1(F,T)$.
\par

There exists an exact sequence
\[
\begin{CD}
1@>>> H^1(\Gal(E_d/F),T(E_d))@>>> H^1(F,T)@>>> H^1(E_d,T)^{\Gal(E_d/F)}.
\end{CD}
\]
Since $T$ is isomorphic to $\GL_1$ over $E_d$, 
by Hilbert $90$, we have $H^1(E_d,T)=1$.
Therefore we have
\[
H^1(\Gal(E_d/F),T(E_d))\cong H^1(F,T).
\]
\par

Let $\sigma$ be the generator of $\Gal(E_d/F)$. 
Then we have
\[
\sigma_d(x)\cdot x^{-1}=\sigma(x)^{-1}x^{-1}=N_{E_d/F}(x)^{-1}
\]
for $x\in T(E_d)\cong E_d^\times$.
Hence we have $B^1(\Gal(E_d/F),T(E_d))\cong N_{E_d/F}(E_d^\times)$.
For given $x\in T(E_d)=E_d^\times$, we put $z_1=1$ and $z_\sigma=x$.
Then 
$z\in Z^1(\Gal(E_d/F),T(E_d))$
if and only if
\[
1=z_1=z_{\sigma\sigma}=z_\sigma\cdot \sigma_d(z_{\sigma})
=x\cdot \sigma_d(x)=x\cdot\sigma(x)^{-1},
\]
i.e., $x\in F^\times$.
Therefore we have
\[
H^1(\Gal(E_d/F),T(E_d))\cong F^\times/N_{E_d/F}(E_d^\times).
\]
Moreover we see that
\begin{align*}
[z_c']\not=1\quad\text{in $H^1(\Gal(E_d/F),T(E_d))$}
\iff
c\not\in N_{E_d/F}(E_d^\times)
\iff
(c,d)=-1,
\end{align*}
as desired.
\end{proof}
\par

For a positive integer $n$ and for $d,c\in F^\times$, 
we let $V=V_{2n,d,c}$ be an orthogonal space
with dimension $2n$ and type $(d,c)$,
and
$W=W_{2n}$ be a symplectic space with dimension $2n$.
We put $G_{n,d,c}=\SO(V_{2n,d,c})$ and $G'_n=\Sp(W_{2n})$.
As in \S \ref{Quadratic}, we decompose 
\[
V_{2n,d,c}=X_{n-1,d,c}\oplus V_{2,d,c}\oplus X_{n-1,d,c}^*
\]
with
\[
X_{n-1,d,c}=Fv_{1,d,c}\oplus \dots\oplus Fv_{n-1,d,c},
\quad
X_{n-1,d,c}^*=Fv_{1,d,c}^*\oplus \dots\oplus Fv_{n-1,d,c}^*,
\quad
V_{2,d,c}=Fe_{d,c}\oplus Fe'_{d,c},
\]
such that
\[
\pair{e_{d,c},e_{d,c}}_{V_{2n,d,c}}=2c,\quad
\pair{e'_{d,c},e'_{d,c}}_{V_{2n,d,c}}=-2cd,\quad
\pair{e_{d,c},e'_{d,c}}_{V_{2n,d,c}}=0.
\]
For $c,c'\in F^\times$, 
we define an isomorphism of vector spaces
\[
f\colon V_{2n,d,c}\rightarrow V_{2n,d,c'}
\]
by
\begin{align*}
V_{2n,d,c}\supset X_{n-1,d,c}\ni v_{i,d,c}&\mapsto v_{i,d,c'}\in X_{n-1,d,c'}\subset V_{2n,d,c'},\\
V_{2n,d,c}\supset V_{2,d,c}\ni e_{d,c}&\mapsto e_{d,c'}\in V_{2,d,c'}\subset V_{2n,d,c'},\\
V_{2n,d,c}\supset V_{2,d,c}\ni e'_{d,c}&\mapsto e'_{d,c'}\in V_{2,d,c'}\subset V_{2n,d,c'},\\
V_{2n,d,c}\supset X_{n-1,d,c}^*\ni v_{i,d,c}^*&\mapsto \frac{c'}{c}\cdot v_{i,d,c'}^*
\in X_{n-1,d,c'}^*\subset V_{2n,d,c'}.
\end{align*}
This satisfies that 
\[
\pair{f(x),f(y)}_{V_{2n,d,c'}}=\frac{c'}{c}\pair{x,y}_{V_{2n,d,c}}
\]
for any $x,y\in V_{2n,d,c}$.
Hence $f$ gives an isomorphism $\psi_{c'/c}\colon G_{2n,d,c}\rightarrow G_{2n,d,c'}$ over $F$.
Let $Z_{2n,d,c}$ be the center of $G_{2n,d,c}$ and
define $z_{c'/c}\in Z^1(F,Z_{2n,d,c})$ by
\[
z_{c'/c,\sigma}=\frac{\sigma(\sqrt{c'/c})}{\sqrt{c'/c}}\cdot \1
\]
for $\sigma \in \Gamma$.
Then $(\psi,z)=(\psi_{c'/c},z_{c'/c})\colon G_{2n,d,c}\rightarrow G_{2n,d,c'}$ 
is a pure inner twist of $G_{2n,d,c}$.
This is trivial if and only if $c'/c\in N_{E_d/F}(E_d^\times)$.
Via this pure inner twist, we may regard $G_{2n,d,c'}$ as a pure inner form of $G_{2n,d,c}$.
\par

Now we assume that $c'/c \in N_{E_d/F}(E_d^\times)$.
Then there exists $\alpha \in E_d^\times$ such that $N_{E_d/F}(\alpha)=c'/c$.
We write $\alpha^{-1} = a + b \sqrt{d}$ with $a, b \in F$.
We define an isomorphism of vector spaces
\[
f_{\alpha^{-1}}\colon V_{2n,d,c}\rightarrow V_{2n,d,c'}
\]
by
\begin{align*}
V_{2n,d,c}\supset X_{n-1,d,c}\ni v_{i,d,c}&\mapsto v_{i,d,c'}\in X_{n-1,d,c'}\subset V_{2n,d,c'},\\
V_{2n,d,c}\supset V_{2,d,c}\ni xe_{d,c}+ye'_{d,c}&\mapsto 
(ax+byd)e_{d,c'}+(ay+bx)e'_{d,c'}
\in V_{2,d,c'}\subset V_{2n,d,c'},\\
V_{2n,d,c}\supset X_{n-1,d,c}^*\ni v_{i,d,c}^*&\mapsto v_{i,d,c'}^*
\in X_{n-1,d,c'}^*\subset V_{2n,d,c'}.
\end{align*}
Note that
\[
(a+b\sqrt{d})(x+y\sqrt{d}) = (ax+byd) + (ay+bx)\sqrt{d}.
\]
Then $f_{\alpha^{-1}}$ 
gives an isomorphism $\psi_{\alpha^{-1}}\colon G_{2n,d,c}\rightarrow G_{2n,d,c'}$ over $F$.
Via this isomorphism $\psi_{\alpha^{-1}}$, we identify 
$\Pi_\temp(G_{2n,d,c})$ with $\Pi_\temp(G_{2n,d,c'})$.
\begin{lem}
Suppose that $c'/c \in N_{E_d/F}(E_d^\times)$.
Fix $\alpha \in E_d^\times$ such that $N_{E_d/F}(\alpha)=c'/c$.
Let $(\psi,z)=(\psi_{c'/c},z_{c'/c})\colon G_{2n,d,c}\rightarrow G_{2n,d,c'}$ 
be the pure inner twist of $G_{2n,d,c}$ defined as above.
Then the bijection $J_{\psi,z} \colon \Pi_\temp(G_{2n,d,c}) \rightarrow \Pi_\temp(G_{2n,d,c'})$
in Lemma \ref{J}
is the identity map via the above identification $\Pi_\temp(G_{2n,d,c}) = \Pi_\temp(G_{2n,d,c'})$.
\end{lem}
\begin{proof}
We put $\psi_{\alpha} = (\psi_{\alpha^{-1}})^{-1}$ and 
$\psi' = \psi_{\alpha} \circ \psi_{c'/c} \colon G_{2n,d,c} \rightarrow G_{2n,d,c}$.
It suffices to show that
\[
(G_1, \psi_1, z_1, \pi_1) \cong (G_1, \psi_1 \circ (\psi')^{-1}, \psi'(z_1z^{-1}), \pi_1)
\]
for any $(G_1, \psi_1, z_1, \pi_1) \in \Pi_\temp(G_{2n,d,c})$.
There exists $h_\alpha \in \mathrm{GSO}(V_{2n,d,c})$ with similitude factor
$N_{E_d/F}(\alpha)=c'/c$ such that
$(f_{\alpha^{-1}})^{-1} \circ f = \Int(h_\alpha)$.
Fix a square root $\sqrt{N_{E_d/F}(\alpha)}$ of $N_{E_d/F}(\alpha)$ in $\overline{F}$.
Put $g=\sqrt{N_{E_d/F}(\alpha)}^{-1}h_\alpha$, which is an element in $G_{2n,d,c}(\overline{F})$.
Then we see that $(\id_{G_1}, g)$ gives a desired isomorphism 
$(G_1, \psi_1, z_1, \pi_1) \rightarrow (G_1, \psi_1 \circ (\psi')^{-1}, \psi'(z_1z^{-1}), \pi_1)$.
\end{proof}

If $c'/c\not\in N_{E_d/F}(E_d^\times)$, then
we cannot define $\psi_{\alpha^{-1}}$, but can define $(\psi,z)=(\psi_{c'/c},z_{c'/c})$.
So we identify $\Pi_\temp(G_{2n,d,c})$ with $\Pi_\temp(G_{2n,d,c'})$ via $J_{\psi,z}$.
\par

Let $B=TU$ (\resp $B'=T'U'$)
be the Borel subgroup of $G_{2n,d,c}$ (\resp $G'_{2n}$)
and $\w_c=(B,\mu_c)$ (\resp $\w'_c=(B',\mu'_c)$)
be the Whittaker datum of $G_{2n,d,c}$ (\resp $G'_{2n}$) defined in \S \ref{section spaces}.
Via the isomorphism $\psi_{c'/c}\colon G_{2n,d,c}\rightarrow G_{2n,d,c'}$,
the datum $\w_c$ is transfered to $\w_{c'}$.
\par

Recall that
an $L$-parameter $\phi$ of $G_{2n,d,c}$
(\resp $G'_{2n}$) gives an orthogonal representation
$\phi\colon \WD_F\rightarrow \orth(2n,\C)$
(\resp 
$\phi\colon \WD_F\rightarrow \SO(2n+1,\C)$
)
with $\det(\phi)=\chi_{V_{2n,d,c}}$ (\resp $\det(\phi)=\1$).
For a semi-simple element $s\in S_\phi$,
we denote the $(-1)$-eigenspace of $s$ in $\phi$ by $\phi^{s}$.
This is a subrepresentation of $\phi$.
We put $\eta_{c_0}(s)=\det(\phi^{s})(c_0)$ for $c_0\in F^\times$.
This gives a character $\eta_{c_0}$ of $\pi_0(S_\phi)$.
Now we prove Propositions \ref{whittaker1} and \ref{whittaker2}.
More precisely, we prove the following proposition.
\begin{prop}\label{SO&Sp}
Let $G=G_{2n,d,c}$ with $n \geq 2$ or $G=G'_{2n}$ with $n \geq 1$, and $Z\cong\{\pm1\}$ be the center of $G$.
For $c_0\in F^\times$, we define $z_{c_0}\in Z^1(F,Z)$ by
$z_{c_0,\sigma}=\sigma(\sqrt{c_0})/\sqrt{c_0}$.
Then we have $\chi_{z_{c_0}}=\eta_{c_0}$.
Hence
for a tempered $L$-parameter $\phi$ of $G$ and $c,c' \in F^\times$,
the diagram
\[
\begin{CD}
\Pi_\phi @>\iota_{\w_c}>> \Irr(\pi_0(S_\phi))\\
@| @VV{\cdot\otimes\eta_{c'/c}}V\\
\Pi_\phi @>\iota_{\w_{c'}}>> \Irr(\pi_0(S_\phi))
\end{CD}
\]
is commutative.
Here, we identify $\Pi_\temp(G_{2n,d,c})$ with $\Pi_\temp(G_{2n,d,c'})$ via $J_{\psi,z}$
with $(\psi,z)=(\psi_{c'/c},z_{c'/c})$.
\end{prop}
\begin{proof}
We may assume that $\phi=\phi_1\oplus\phi_2$,
where $\phi_1$ and $\phi_2$ are orthogonal representations with 
$\dim(\phi_1)=2k$ and $\dim(\phi_2)=\dim(\phi)-2k$,
and that $s\in S_\phi$ satisfies
$s|\phi_1=-\1$ and $s|\phi_2=\1$.
Then $\phi^{s}=\phi_1$.
Let $(H,\H,s,\eta)$ be the endoscopic datum associated to $s$.
Then we have
\[
H=\left\{
\begin{aligned}
&G_{2k,d_1,1}\times G_{2(n-k),d/d_1,1}\iif G=G_{2n,d,c},\\
&G_{2k,d_1,1}\times G'_{2(n-k)}\iif G=G_{2n}',
\end{aligned}
\right.
\]
where $d_1\in F^\times$ satisfies that $\det(\phi_1)=(\cdot,d_1)$.
We can take a maximal $F$-torus $S$ of $H$ which is isomorphic to
\[
\left\{
\begin{aligned}
&(\GL_1^{k-1}\times E_{d_1}^1)\times(\GL_1^{n-k-1}\times E_{d/d_1}^1) \iif G=G_{2n,d,c},\\
&(\GL_1^{k-1}\times E_{d_1}^1)\times\GL_1^{n-k} \iif G=G_{2n}'
\end{aligned}
\right.
\]
over $F$.
This implies that
\[
H^1(F,S)\cong \left\{
\begin{aligned}
&H^1(F,E_{d_1}^1)\times H^1(F,E_{d/d_1}^1)\hookrightarrow\{\pm1\}\times \{\pm1\}
\iif G=G_{2n,d,c},\\
&H^1(F,E_{d_1}^1)\hookrightarrow\{\pm1\}\iif G=G_{2n}'.
\end{aligned}
\right.
\]
Via this isomorphism, (the image of) $[z_{c_0}]$ corresponds to
$((c_0,d_1),(c_0,d/d_1))$ if $G=G_{2n,d,c}$, and to $(c_0,d_1)$ if $G=G'_{2n}$.
\par

On the other hand, via the above isomorphism of $S$, we have
\[
\pi_0(\widehat{S}^\Gamma)
\cong \left\{
\begin{aligned}
&\pi_0(\widehat{E_{d_1}^1}^\Gamma)\times \pi_0(\widehat{E_{d/d_1}^1}^\Gamma)
\iif G=G_{2n,d,c},\\
&\pi_0(\widehat{E_{d_1}^1}^\Gamma)
\iif G=G_{2n}'.
\end{aligned}
\right.
\]
Via this isomorphism, $s$ corresponds to 
the generator of $\pi_0(\widehat{E_{d_1}^1}^\Gamma)\subset \{\pm1\}$.
Therefore we have
\[
\chi_{z_{c_0}}(s)=\pair{[z_{c_0}],s}=(c_0,d_1)=\det(\phi_1)(c_0)=\eta_{c_0}(s).
\]
\par

First, we assume that $[z_{c'/c}]\not=1$ in $H^1(F,G)$.
In particular, $G=G_{2n,d,c}$ and $c'/c\not\in N_{E_d/F}(E_d^\times)$.
By Theorem \ref{base}, the diagram
\[
\begin{CD}
\Pi_\phi @>\iota_{\w_{c}}>>\Irr(\pi_0(S_\phi))\\
@VJ_{\psi,z}VV @VV\cdot\otimes\eta_{c'/c}V\\
\Pi_\phi @>\iota_{\w_{c'}}>>\Irr(\pi_0(S_\phi))\\
\end{CD}
\]
is commutative, where we put $(\psi,z)=(\psi_{c'/c},z_{c'/c})$.
Note that we have identified $\Pi_\phi\subset \Pi_\temp(G_{2n,d,c})$
with $\Pi_\phi\subset \Pi_\temp(G_{2n,d,c'})$ via $J_{\psi,z}$.
\par

Next, we assume that $[z_{c_0}]=1$ in $H^1(F,G)$, where we put $c_0=c'/c$.
Let $T$ be the maximal $F$-torus of $G$ stabilizing the lines $Fv_i$ or $Fw_i$ for any $i$.
If $G=G'_{2n}$, then $T\cong \GL_1^n$.
Putting $t_{c_0}=(\sqrt{c_0}^{-1},\dots,\sqrt{c_0}^{-1})\in T(\overline{F})\subset G(\overline{F})$ and
let $g_{c_0}\in G_\ad(F)$ be the image of $t_{c_0}$.
Then we have $z_{c_0}=t_{c_0}^{-1}\cdot \sigma(t_{c_0})$,
i.e., $[z_{c_0}]\in H^1(F,Z)$ is the image of $g_{c_0}$.
We have $g_{c_0}B'g_{c_0}^{-1}=B'$ and $\mu'_{c}\circ\Int(g_{c_0}^{-1})=\mu'_{cc_0}$.
Hence we have $\iota_{\w'_c}\circ J_{z_{c_0}}=\iota_{\w'_{cc_0}}$
by Proposition \ref{Whittaker}.
\par

If $G=G_{2n,d,c}$, then $T\cong \GL_1^{n-1}\times E_d^1$.
We see that $[z_{c_0}]=1$ in $H^1(F,G)$ if and only if $c_0\in N_{E_d/F}(E_d^\times)$.
Take $a_0,b_0\in F$ such that $N_{E_d/F}(a_0+b_0\sqrt{d})=a_0^2-b_0^2d=c_0$.
We define $t_{c_0}\in T(\overline{F})$ by $t_{c_0}v_i=\sqrt{c_0}^{-1}v_i$ and
\begin{align*}
t_{c_0}e&=\frac{a_0}{\sqrt{c_0}}e
	+\frac{b_0}{\sqrt{c_0}}e',\\
t_{c_0}e'&=\frac{b_0}{\sqrt{c_0}}de
	+\frac{a_0}{\sqrt{c_0}}e'.
\end{align*}
Note that $t_0=\sqrt{c_0}^{-1}f$, where $f\colon V_{2n,d,c}\rightarrow V_{2n,d,c'}=V_{2n,d,c}$
is the isomorphism defined as above.
Let $g_{c_0}\in G_\ad(F)$ be the image of $t_{c_0}$.
Then we have $z_{c_0}=t_{c_0}^{-1}\cdot \sigma(t_{c_0})$,
i.e., $[z_{c_0}]\in H^1(F,Z)$ is the image of $g_{c_0}$.
We have $g_{c_0}Bg_{c_0}^{-1}=B$ and $\mu_{c}\circ\Int(g_{c_0}^{-1})=\mu_{cc_0}$.
Hence we have $\iota_{\w_c}\circ J_{z_{c_0}}=\iota_{\w_{cc_0}}$
by Proposition \ref{Whittaker}.
\end{proof}

%\section{Some results on representations of metaplectic groups}\label{appB}
\section{Some results on representations of metaplectic groups}\label{appB}
In this appendix, we give proofs of Theorems \ref{GPRforMp} and \ref{IS}.

%\subsection{Proof of Theorem \ref{GPRforMp}}
\subsection{Proof of Theorem \ref{GPRforMp}}
In this subsection, we prove Theorem \ref{GPRforMp}.
For the proof, we use the local coefficients for metaplectic groups defined in \cite{Sz2}.
For linear groups, see \cite{Sh, Sh2}.
\par

Let
\[
\Ind_{\cl{Q}}^{\Mp(W)}(\tau_1|\det|_F^{s_1}\otimes\cdots\otimes\tau_r|\det|_F^{s_r}\otimes \cl{\pi}_0)
\]
be a standard module, where
\begin{itemize}
\item
$W=W_{2n}$ and $W_0 = W_{2n_0}$ are symplectic spaces of dimension $2n$ and $2n_0$, 
respectively;
\item
$\cl{Q}$ is the parabolic subgroup of $\Mp(W)$ with the Levi subgroup isomorphic to
$\cl{\GL}_{n_1}(F)\times_{\{\pm1\}}\cdots\times_{\{\pm1\}}\cl{\GL}_{n_r}(F)
\times_{\{\pm1\}}\Mp(W_0)$; 
\item
$\tau_i$ is an irreducible discrete series representation of $\cl{\GL}_{n_i}(F)$;
\item
$s_i$ is a real number such that $s_1 \geq \dots \geq s_r > 0$; 
\item
$\cl{\pi}_0$ is an irreducible tempered representation of $\Mp(W_0)$.
\end{itemize}
We denote the unique irreducible quotient of this standard module by $\cl\pi$. 
We set $\phi = \LL_\psi(\cl\pi)$.
In \cite[\S 2.3]{Sz2}, Szpruch defined an intertwining operator 
\[
A_w \colon 
\Ind_{\cl{Q}}^{\Mp(W)}(\tau_1|\det|_F^{s_1}\otimes\cdots\otimes\tau_r|\det|_F^{s_r}\otimes \cl{\pi}_0)
\rightarrow
\Ind_{\cl{Q}}^{\Mp(W)}
(\tau_1^\vee|\det|_F^{-s_1}\otimes\cdots\otimes\tau_r^\vee|\det|_F^{-s_r}\otimes \cl{\pi}_0)
\]
for some element $w \in \Sp(W_{2n})$.
\par

Now suppose that $\cl\pi_0$ is $\w_1'$-generic.
Szpruch \cite[\S 3.1]{Sz2} defined a $\w_1'$-generic Whittaker functional
\[
\lam((s_1, \dots, s_r), \tau_1 \times \dots \times \tau_r \rtimes \cl\pi_0, \mu'_1) 
\in \Hom_{U'} \left(\Ind_{\cl{Q}}^{\Mp(W)}(\tau_1|\det|_F^{s_1}\otimes\cdots\otimes\tau_r|\det|_F^{s_r}\otimes \cl{\pi}_0), \mu_1' \right).
\]
By the uniqueness of Whittaker model (\cite{Sz1}), 
there exists a meromorphic function 
\[
C_\psi^{\Mp(W_{2n})}( (s_1, \dots, s_r), \tau_1 \times \dots \times \tau_r \rtimes \cl\pi_0, w)
\]
on $\C^r$ such that
\begin{align*}
&\lam((s_1, \dots, s_r), \tau_1 \times \dots \times \tau_r \rtimes \cl\pi_0, \mu'_1) 
\\&=
C_\psi^{\Mp(W_{2n})}( (s_1, \dots, s_r), \tau_1 \times \dots \times \tau_r \rtimes \cl\pi_0, w)
\cdot
\lam((-s_1, \dots, -s_r), \tau_1^\vee \times \dots \times \tau_r^\vee \rtimes \cl\pi_0, \mu'_1) \circ A_w.
\end{align*}
We call $C_\psi^{\Mp(W_{2n})}( (s_1, \dots, s_r), \tau_1 \times \dots \times \tau_r \rtimes \cl\pi_0, w)$
the local coefficient.
By the the same argument as the proof of \cite[Lemma B.2]{GI2}, 
we see that the Langlands quotient $\cl\pi$ is $\mu'_1$-generic
if and only if the absolute value of 
$C_\psi^{\Mp(W_{2n})}( (s_1, \dots, s_r), \tau_1 \times \dots \times \tau_r \rtimes \cl\pi_0, w)$
is finite.
\par

When $r=1$ and $\cl\pi_0$ is $\w_1'$-generic, 
Szpruch \cite{Sz2} defined a gamma factor $\gamma^{\Sh}(s_1, \tau_1 \times \cl\pi_0, \psi)$
by
\[
\gamma^{\Sh}(s_1, \tau_1 \times \cl\pi_0, \psi)=
\frac{C_\psi^{\Mp(W_{2n})}( s_1, \tau_1 \rtimes \cl\pi_0, w_n)}
{C_\psi^{\Mp(W_{2n_1})}( s_1, \tau_1, w_{n_1})}
\]
for some element $w_n \in \Sp(W_{2n})$ and $w_{n_1} \in \Sp(W_{2n_1})$.
By \cite[Theorem A]{Sz2}, it satisfies fundamental properties which include multiplicativity.
\par

We start to prove Theorem \ref{GPRforMp}.
First, suppose that $\cl\pi$ is $\mu'_1$-generic. 
%In \cite{Sz2}, Szpruch defined a gamma factor $\gamma^{\Sh}(s, \tau \times \cl\pi_0, \psi)$
%by the Langlands--Shahidi method.
We need a following formula for the local coefficient in terms of gamma factors.
\begin{lem}\label{lem1}
Suppose that the Langlands quotient $\cl\pi$ is $\mu'_1$-generic.
Then there exists an invertible function $c(s_1, \dots, s_r)$ such that
\begin{align*}
&C_\psi^{\Mp(W_{2n})}( (s_1, \dots, s_r), \tau_1 \times \dots \times \tau_r \rtimes \cl\pi_0, w)
\\&= c(s_1, \dots, s_r) \prod_{i=1}^r\gamma^{\Sh}(s_i, \tau_i \times \cl\pi_0, \psi) 
\frac{\gamma(2s_i, \tau_i, \Sym^2, \psi)}{\gamma(s_i + \half{1}, \tau_i, \psi)}
\prod_{i < j}\gamma(s_i+s_j, \tau_i \times \tau_j, \psi)\gamma(s_i-s_j, \tau_i \times \tau_j^\vee, \psi).
\end{align*}
\end{lem}
\begin{proof}
Note that if $\cl\pi$ is $\mu_1'$-generic, then so is $\cl\pi_0$.
This is Rodier's result for linear groups (see \cite[Corollary 1.7]{CS}), 
which is extended by Banks \cite{B} to a non-algebraic setting.
When $r=1$, Lemma follows from the definition of $\gamma^{\Sh}(s, \tau \times \cl\pi_0, \psi)$
and \cite[Theorem B]{Sz2}.
In general, Lemma follows from a standard argument similar to \cite[Theorem 2.1.1, Proposition 3.2.1]{Sh}.
\end{proof}

Moreover, if we set $\phi_0 = \LL_\psi(\cl\pi_0)$ and 
we denote the representation of $\WD_F$ corresponding to $\tau_i$ by $\phi_i$, 
we have
\[
\gamma^{\Sh}(s_i, \tau_i \times \cl\pi_0, \psi)  
= \gamma(s_i, \phi_i \times \phi_0, \psi) 
\]
up to an invertible function.
This follows from \cite[Proposition 4.5]{ILM} and \cite{A}.
Therefore, the absolute value of 
$C_\psi^{\Mp(W_{2n})}( (s_1, \dots, s_r), \tau_1 \times \dots \times \tau_r \rtimes \cl\pi_0, w)$
is finite if and only if the absolute value of
\[
\prod_{i=1}^r\gamma(s_i, \phi_i \times \phi_0, \psi) 
\frac{\gamma(2s_i, \phi_i, \Sym^2, \psi)}{\gamma(s_i + \half{1}, \phi_i, \psi)}
\prod_{i < j}\gamma(s_i+s_j, \phi_i \times \phi_j, \psi)\gamma(s_i-s_j, \phi_i \times \phi_j^\vee, \psi).
\]
This implies that $L(s, \phi, \Ad)/L(s-1/2, \phi)$ is regular at $s=1$.
\par

Next, we show the converse.
\begin{lem}\label{lem2}
Set $\phi' = \phi - (\phi_1|\cdot|_F^{s_1} \oplus \phi_1^\vee|\cdot|_F^{-s_1})$.
If $L(s, \phi, \Ad)/L(s-1/2, \phi)$ is regular at $s=1$, 
then so is $L(s, \phi', \Ad)/L(s-1/2, \phi')$.
\end{lem}
\begin{proof}
Let $S_d$ be the irreducible representation of $\SL_2(\C)$ of dimension $d$.
The lemma follows from the definition of the $L$-function together with
the equation
\[
L(s, S_d, \Sym^2) = \prod_{\substack{0 \leq k \leq d-1 \\ \text{$k$ is even}}}\zeta_F(s+d-1-k), 
\]
where $\zeta_F(s)$ is the zeta function associated to $F$.
\end{proof}

Suppose that $\cl\pi_0$ is $\mu_1'$-generic.
By induction on $r$, 
we shall show that if $L(s, \phi, \Ad)/L(s-1/2, \phi)$ is regular at $s=1$,
then the Langlands quotient $\cl\pi$ is $\mu_1'$-generic.
We may assume that $r>0$.
Let $\cl\pi'$ be the Langlands quotient of the standard module
\[
\Ind_{\cl{Q'}}^{\Mp(W')}(\tau_2|\det|_F^{s_2}\otimes\cdots\otimes\tau_r|\det|_F^{s_r}\otimes \cl{\pi}_0), 
\]
where $W' = W_{2(n-n_1)}$.
Then $\cl\pi$ is the unique irreducible quotient of
\[
\Ind_{\cl{Q}_1}^{\Mp(W)}(\tau_1|\det|_F^{s_1}\otimes\cl\pi'),
\]
where $Q_1$ is a suitable maximal parabolic subgroup of $\Sp(W)$.
By Lemma \ref{lem2} and the induction hypothesis, 
we see that $\cl\pi'$ is $\mu_1'$-generic.
Hence we can consider the local coefficient $C_\psi^{\Mp(W_{2n})}( s, \tau_1 \rtimes \cl\pi', w)$.
Since $Q_1$ is a maximal parabolic subgroup,
by the definition of gamma factors and multiplicativity, 
$C_\psi^{\Mp(W_{2n})}(s, \tau_1 \rtimes \cl\pi', w)$ can be expressed as in Lemma \ref{lem1}.
Since $L(s, \phi, \Ad)/L(s-1/2, \phi)$ is regular at $s=1$,
the absolute value of $C_\psi^{\Mp(W_{2n})}(s, \tau_1 \rtimes \cl\pi', w)$ must be finite.
Therefore, we conclude that $\cl\pi$ is $\mu_1'$-generic.

%\subsection{Proof of Theorem \ref{IS}}
\subsection{Proof of Theorem \ref{IS}}
In this subsection, we prove Theorem \ref{IS}.
Let
\[
\cl{\pi} = \Ind_{\cl{Q}}^{\Mp(W)}
(\tau_1|\det|_F^{s_1}\otimes\cdots\otimes\tau_r|\det|_F^{s_r}\otimes \cl{\pi}_0)
\]
be a standard module, where
\begin{itemize}
\item
$W=W_{2n}$ and $W_0 = W_{2n_0}$ are symplectic spaces of dimension $2n$ and $2n_0$, 
respectively;
\item
$\cl{Q}$ is the parabolic subgroup of $\Mp(W)$ with the Levi subgroup isomorphic to
$\cl{\GL}_{n_1}(F)\times_{\{\pm1\}}\cdots\times_{\{\pm1\}}\cl{\GL}_{n_r}(F)
\times_{\{\pm1\}}\Mp(W_0)$; 
\item
$\tau_i$ is an irreducible discrete series representation of $\cl{\GL}_{n_i}(F)$;
\item
$s_i$ is a real number such that $s_1 \geq \dots \geq s_r > 0$; 
\item
$\cl{\pi}_0$ is an irreducible tempered representation of $\Mp(W_0)$.
\end{itemize}
We denote the image of the unique irreducible quotient of $\cl{\pi}$ under $\LL_\psi$ by $\phi$.
Suppose that $\phi$ is generic, i.e., the adjoint $L$-function $L(s, \phi, \Ad)$ is regular at $s=1$.
\par

By a similar argument to the proof of \cite[Proposition 5.2, Corollary 5.3]{AG}, 
the following key property can be proven 
by using Kudla's filtration of the Jacquet module of the Weil representation.

\begin{lem}\label{lem3}
Assume that for any positive integer $d$, 
if $\tau_i \cong \mathrm{St}_d$ is the Steinberg representation of $\GL_d(F)$, then $s_i \not = d/2$.
%Fix
%\[
%\pi = \tau_{1,\psi} \boxtimes \dots \boxtimes \tau_{r,\psi} \boxtimes \pi_0 \in \Pi_\phi.
%\]
For $i = r, r-1, \dots, 1$, 
we inductively choose $\cl\pi_i \in \Irr(\Mp(W_{2(n_0+k_r+\dots+k_i)}))$ so that
\begin{itemize}
\item
$\cl\pi_r$ is an irreducible subrepresentation of 
$\Ind_{\cl{Q_r}}^{\Mp(W_{2(n_0+k_r)})}(\tau_{r}|\det|_F^{s_r} \boxtimes \cl\pi_0)$;
\item
$\cl\pi_i$ is an irreducible subrepresentation of 
$\Ind_{\cl{Q_i}}^{\Mp(W_{2(n_0+k_r+\dots+k_i)})}(\tau_{i}|\det|_F^{s_i} \boxtimes \cl\pi_{i+1})$
for $i = 1, \dots, r-1$, 
\end{itemize}
where $\cl{Q_i}$ is a maximal parabolic subgroup of $\Sp(W_{2(n_0+k_r+\dots+k_i)})$
with the Levi subgroup isomorphic to
$\cl{\GL}_{n_i}(F)\times_{\{\pm1\}}\Mp(W_{2(n_0+n_r+\dots+n_{i+1})})$. 
Then $\cl\pi_1$ is an irreducible subrepresentation of the standard module $\cl\pi$.
Moreover, if the theta lift $\theta_\psi(\cl\pi_1)$ to $\SO(V_{2n_0+1}^\bullet)$ is nonzero, then
\[
\Ind_{P^\circ}^{\SO(V_{2n+1}^\bullet)}
(\tau_1^\vee|\det|_F^{-s_1} \boxtimes \dots \boxtimes \tau_r^\vee|\det|_F^{-s_r} 
\boxtimes \theta_\psi(\cl\pi_0))
\twoheadrightarrow \theta_\psi(\cl\pi_1).
\]
where $P^\circ$ is a parabolic subgroup of $\SO(V_{2n+1}^\bullet)$
with the Levi subgroup isomorphic to
$\GL_{n_1}(F)\times\cdots\times\GL_{n_r}(F)\times\SO(V_{2n_0+1}^\bullet)$.

\end{lem}
\begin{proof}
By the exactness of the induced functor, 
we see that
$\pi_1$ is an irreducible subrepresentation of the standard module $\cl\pi$.
%$\Ind_{\cl{P}}^{\Mp(W_{2n})}(\pi_{\vec{s}_0})$.
\par

By a similar argument to \cite[Proposition 5.2, Corollary 5.3]{AG} 
using the assumption on $s_i$ and $\tau_i$, 
we see that 
\[
\Ind_{P_1^\circ}^{\SO(V_{2n+1}^\bullet)}
(\tau_1^\vee|\det|_F^{-s_1} \boxtimes \Theta_\psi(\cl\pi_2))
\twoheadrightarrow \Theta_\psi(\cl\pi_1),
\]
where $P_1^\circ$ is a suitable maximal parabolic subgroup of $\SO(V_{2n+1}^\bullet)$.
By induction on $r$ and the exactness of the induced functor, 
we conclude that
\[
\Ind_{P^\circ}^{\SO(V_{2n+1}^\bullet)}
(\tau_1^\vee|\det|_F^{-s_1} \boxtimes \dots \boxtimes \tau_r^\vee|\det|_F^{-s_r} 
\boxtimes \Theta_\psi(\cl\pi_0))
\twoheadrightarrow \Theta_\psi(\cl\pi_1) \twoheadrightarrow \theta_\psi(\cl\pi_1).
\]
Since $\cl\pi_0$ is tempered, by \cite[Theorem 8.1]{GS1}, we have 
$\Theta_\psi(\cl\pi_0) = \theta_\psi(\cl\pi_0)$.
We obtain the lemma.
\end{proof}

\begin{proof}[Proof of Theorem \ref{IS}]
We continue the same notation.
For a positive integer $d$, the symmetric square $L$-function
\[
L(s, S_d^{\vee}|\cdot|_F^{-2d/2}, \Sym^2) 
= \prod_{\substack{0 \leq k \leq d-1 \\ \text{$k$ is even}}}\zeta_F(s-1-k)
\]
has a pole at $s=1$.
Hence if $\phi$ is generic, then it satisfies the assumption in Lemma \ref{lem3}, and we have
\[
\Ind_{P^\circ}^{\SO(V_{2n+1}^\bullet)}
(\tau_1^\vee|\det|_F^{-s_1} \boxtimes \dots \boxtimes \tau_r^\vee|\det|_F^{-s_r} 
\boxtimes \theta_\psi(\cl\pi_0))
\twoheadrightarrow \theta_\psi(\cl\pi_1).
\]
However, by \cite[p. 40 Th\'eor\`eme $(\mathrm{i})$]{MW}, 
the standard module in the left hand side is irreducible.
Moreover, by by \cite[Theorem 8.1]{GS1}, it is isomorphic to 
the theta lift of the Langlands quotient of $\cl\pi$.
Since theta lifts are injective, we conclude that 
$\pi_1$ is isomorphic to the Langlands quotient of $\cl\pi$.
However, since the Langlands quotient appears in the standard module with multiplicity one 
(\cite[Corollary 4.4, Remark 4.5]{BJ2}), 
the standard module $\cl\pi$ must be irreducible.
\end{proof}

%\begin{thebibliography}{55}

\end{document}